\documentclass[reqno]{amsart}
\usepackage{mathrsfs}
\usepackage{color}
\usepackage{amsmath}
\usepackage{amsfonts}
\usepackage{amssymb}
\usepackage{graphicx}
\usepackage{hyperref}

 \newtheorem{Theorem}{Theorem}[section]
 \newtheorem{Corollary}[Theorem]{Corollary}
 \newtheorem{Lemma}[Theorem]{Lemma}
 \newtheorem{Proposition}[Theorem]{Proposition}
 \newtheorem{Question}[Theorem]{Question}
 \newtheorem{Definition}[Theorem]{Definition}

 \newtheorem{Remark}[Theorem]{Remark}

 \numberwithin{equation}{section}
\begin{document}

\title[Boundary points, minimal $L^2$ integrals and concavity property]
 {Boundary points, minimal $L^2$ integrals and concavity property \uppercase\expandafter{\romannumeral3}---linearity on Riemann surfaces}

%\author{name}
%\address{address}
%\email{email}

\author{Qi'an Guan}
\address{Qi'an Guan: School of
Mathematical Sciences, Peking University, Beijing 100871, China.}
\email{guanqian@math.pku.edu.cn}

\author{Zhitong Mi}
\address{Zhitong Mi: Institute of Mathematics, Academy of Mathematics
and Systems Science, Chinese Academy of Sciences, Beijing, China
}
\email{zhitongmi@amss.ac.cn}

\author{Zheng Yuan}
\address{Zheng Yuan: School of
Mathematical Sciences, Peking University, Beijing 100871, China.}
\email{zyuan@pku.edu.cn}

\thanks{}



\date{\today}

\dedicatory{}

\commby{}

%%% ----------------------------------------------------------------------

\begin{abstract}
In this article, we consider a modified version of minimal $L^2$ integrals on sublevel sets of plurisubharmonic functions related to modules at boundary points, and obtain a concavity property of the modified version.
As an application, we give a characterization for the concavity degenerating to linearity on open Riemann surfaces.

\end{abstract}

%%% ----------------------------------------------------------------------
\maketitle
%%% ----------------------------------------------------------------------
\section{Introduction}
The strong openness property of multiplier ideal sheaves \cite{GZSOC}, i.e. $\mathcal{I}(\varphi)=\mathcal{I}_+(\varphi):=\mathop{\cup} \limits_{\epsilon>0}\mathcal{I}((1+\epsilon)\varphi)$
(conjectured by Demailly \cite{DemaillySoc}) is an important feature of multiplier ideal sheaves and has opened the door to new types of approximation technique (see e.g. \cite{GZSOC,McNeal and Varolin,K16,cao17,cdM17,FoW18,DEL18,ZZ2018,GZ20,ZZ2019,ZhouZhu20siu's,FoW20,KS20,DEL21}),
where the multiplier ideal sheaf $\mathcal{I}(\varphi)$ is the sheaf of germs of holomorphic functions $f$ such that $|f|^2e^{-\varphi}$ is locally integrable (see e.g. \cite{Tian,Nadel,Siu96,DEL,DK01,DemaillySoc,DP03,Lazarsfeld,Siu05,Siu09,DemaillyAG,Guenancia}),
and $\varphi$ is a plurisubharmonic function on a complex manifold $M$ (see \cite{Demaillybook}).

Guan-Zhou \cite{GZSOC} proved the strong openness property (the 2-dimensional case was proved by Jonsson-Musta\c{t}\u{a} \cite{JonssonMustata}).
After that, using the strong openness property, Guan-Zhou \cite{GZeff} proved a conjecture posed by Jonsson-Musta\c{t}\u{a} (Conjecture J-M for short, see \cite{JonssonMustata}),
which is about the volumes growth of the sublevel sets of quasi-plurisubharmonic functions.

Recall that the 2-dimensional case of Conjecture J-M was proved by Jonsson-Musta\c{t}\u{a} \cite{JonssonMustata},
which deduced the 2-dimensional strong openness property.
It is natural to ask

\begin{Question}
\label{Q:JM-SOC}
Can one obtain a proof of Conjecture J-M \textbf{independent} of the strong openness property?
\end{Question}

In \cite{BGY-boundary}, Bao-Guan-Yuan gave a positive answer to Question \ref{Q:JM-SOC} by establishing a concavity property of the minimal $L^2$ integrals related to modules at boundary points of the sublevel sets of plurisubharmonic functions on pseudoconvex domains.
After that, Guan-Mi-Yuan \cite{GMY-boundary2} considered the minimal $L^2$ integrals on weakly pseudoconvex K\"{a}hler manifolds with Lebesgue measurable gain,
and established a concavity property of the minimal $L^2$ integrals,
which deduced a necessary condition for the concavity degenerating to linearity,
a concavity property related to modules at inner points of the sublevel sets,
an optimal support function related to the modules,
a strong openness property of the modules and a twisted version,
an effectiveness result of the strong openness property of the modules.

In this article, we consider a modified version of the minimal $L^2$ integrals in \cite{GMY-boundary2}, and obtain a concavity property of the modified version.
As an application, we give a characterization for the concavity degenerating to linearity on open Riemann surfaces.

\subsection{Main result: minimal $L^2$ integrals and concavity property}\label{sec:Main result}
Let $M$ be a complex manifold. Let $X$ and $Z$ be closed subsets of $M$. We call that a triple $(M,X,Z)$ satisfies condition $(A)$, if the following two statements hold:

$\uppercase\expandafter{\romannumeral1}.$ $X$ is a closed subset of $M$ and $X$ is locally negligible with respect to $L^2$ holomorphic functions; i.e., for any local coordinated neighborhood $U\subset M$ and for any $L^2$ holomorphic function $f$ on $U\backslash X$, there exists an $L^2$ holomorphic function $\tilde{f}$ on $U$ such that $\tilde{f}|_{U\backslash X}=f$ with the same $L^2$ norm;

$\uppercase\expandafter{\romannumeral2}.$ $Z$ is an analytic subset of $M$ and $M\backslash (X\cup Z)$ is a weakly pseudoconvex K\"ahler manifold.

Let $M$ be an $n-$dimensional complex manifold.
Assume that $(M,X,Z)$  satisfies condition $(A)$. Let $K_M$ be the canonical line bundle on $M$.
Let $dV_M$ be a continuous volume form on $M$. Let $F$ be a holomorphic function on $M$.
Assume that $F$ is not identically zero. Let $\psi$ be a plurisubharmonic function on $M$.
Let $\varphi$ be a Lebesgue measurable function on $M$ such that $\varphi+\psi$ is a plurisubharmonic function on $M$.

Let $T\in [-\infty,+\infty)$.
Denote that
$$\Psi:=\min\{\psi-2\log|F|,-T\}.$$
For any $z \in M$ satisfying $F(z)=0$,
we set $\Psi(z)=-T$.
Note that for any $t\ge T$,
the holomorphic function $F$ has no zero points on the set $\{\Psi<-t\}$.
Hence $\Psi=\psi-2\log|F|=\psi+2\log|\frac{1}{F}|$ is a plurisubharmonic function on $\{\Psi<-t\}$.

\begin{Definition}
We call that a positive measurable function $c$ (so-called ``\textbf{gain}") on $(T,+\infty)$ is in class $\tilde{P}_{T,M,\Psi}$ if the following two statements hold:
\par
$(1)$ $c(t)e^{-t}$ is decreasing with respect to $t$;
\par
$(2)$ For any $t_0> T$, there exists a closed subset $E_0$ of $M$ such that $E_0\subset Z\cap \{\Psi(z)=-\infty\}$ and for any compact subset $K\subset M\backslash E_0$, $e^{-\varphi}c(-\Psi)$ has a positive lower bound on $K\cap \{\Psi<-t_0\}$ .
\end{Definition}
\begin{Remark}
Recall that \cite{GMY-boundary2} a positive measurable function $c(t)$ on $(T,+\infty)$ is in class $P_{T,M,\Psi}$ if the following two statements hold:
\par
$(1)$ $c(t)e^{-t}$ is decreasing with respect to $t$;
\par
$(2)$ There exist $T_1> T$ and a closed subset $E$ of $M$ such that $E\subset Z\cap \{\Psi(z)=-\infty\}$ and for any compact subset $K\subset M\backslash E$, $e^{-\varphi}c(-\Psi_1)$ has a positive lower bound on $K$, where $\Psi_1:=\min\{\psi-2\log|F|,-T_1\}$.

Let $T_2>T$ be any real number and denote $\Psi_2:=\min\{\psi-2\log|F|,-T_2\}$.
We note that $c(t)$ is positive on $[T_1,T_2]$(or $[T_2,T_1]$) and has positive lower bound and upper bound on $[T_1,T_2]$(or $[T_2,T_1]$). Hence $\frac{c(-\Psi_2)}{c(-\Psi_1)}$ has a positive lower bound on $M$. Then we know that for any compact subset $K\subset M\backslash E$, $e^{-\varphi_\alpha}c(-\Psi_2)$ has a positive lower bound on $K$.
Then it is clear that $P_{T,M,\Psi}\subseteq\tilde{P}_{T,M,\Psi}$.

Let $M=\Delta\subset \mathbb{C}$ and $Z=X=\emptyset$. Let $\psi=2\log|z|+2\log|z-\frac{1}{2}|$, $\varphi=-2\log|z-\frac{1}{2}|$ and $F=z-\frac{1}{2}$. Let $T=4\log 2$. Then $\Psi=\min\{\psi-2\log|F|,-T\}=\min\{2\log|z|,-4\log 2\}$. Then we know that $c(t)\equiv 1 \in \tilde{P}_{T,M,\Psi}$ and $c(t)\equiv 1$ is not in $P_{T,M,\Psi}$. Hence $P_{T,M,\Psi}$ is a proper subset of $\tilde{P}_{T,M,\Psi}$.
\end{Remark}

Let $z_0$ be a point in $M$. Denote that $\tilde{J}(\Psi)_{z_0}:=\{f\in\mathcal{O}(\{\Psi<-t\}\cap V): t\in \mathbb{R}$ and $V$ is a neighborhood of $z_0\}$. We define an equivalence relation $\backsim$ on $\tilde{J}(\Psi)_{z_0}$ as follows: for any $f,g\in \tilde{J}(\Psi)_{z_0}$,
we call $f \backsim g$ if $f=g$ holds on $\{\Psi<-t\}\cap V$ for some $t\gg T$ and open neighborhood $V\ni o$.
Denote $\tilde{J}(\Psi)_{z_0}/\backsim$ by $J(\Psi)_{z_0}$, and denote the equivalence class including $f\in \tilde{J}(\Psi)_{z_0}$ by $f_{z_0}$.

If $z_0\in \cap_{t>T} \{\Psi<-t\}$, then $J(\Psi)_{z_0}=\mathcal{O}_{M,z_0}$ (the stalk of the sheaf $\mathcal{O}_{M}$ at $z_0$), and $f_{z_0}$ is the germ $(f,z_0)$ of holomorphic function $f$. If $z_0\notin \cap_{t>T} \overline{\{\Psi<-t\}}$, then  $J(\Psi)_{z_0}$ is trivial.

Let $f_{z_0},g_{z_0}\in J(\Psi)_{z_0}$ and $(h,z_0)\in \mathcal{O}_{M,z_0}$. We define $f_{z_0}+g_{z_0}:=(f+g)_{z_0}$ and $(h,z_0)\cdot f_{z_0}:=(hf)_{z_0}$.
Note that $(f+g)_{z_0}$ and $(hf)_{z_0}$ ($\in J(\Psi)_{z_0}$) are independent of the choices of the representatives of $f,g$ and $h$. Hence $J(\Psi)_{z_0}$ is an $\mathcal{O}_{M,z_0}$-module.

For $f_{z_0}\in J(\Psi)_{z_0}$ and $a,b\ge 0$, we call $f_{z_0}\in I\big(a\Psi+b\varphi\big)_{z_0}$ if there exist $t\gg T$ and a neighborhood $V$ of $z_0$,
such that $\int_{\{\Psi<-t\}\cap V}|f|^2e^{-a\Psi-b\varphi}dV_M<+\infty$.
Note that $I\big(a\Psi+b\varphi\big)_{z_0}$ is an $\mathcal{O}_{M,z_0}$-submodule of $J(\Psi)_{z_0}$.
If $z_0\in \cap_{t>T} \{\Psi<-t\}$, then $I_{z_0}=\mathcal{O}_{M,z_0}$, where $I_{z_0}:=I\big(0\Psi+0\varphi\big)_{z_0}$.

Let $Z_0$ be a subset of $\cap_{t>T} \overline{\{\Psi<-t\}}$. Let $f$ be a holomorphic $(n,0)$ form on $\{\Psi<-t_0\}\cap V$, where $V\supset Z_0$ is an open subset of $M$ and $t_0\ge T$
is a real number.
Let $J_{z_0}$ be an $\mathcal{O}_{M,z_0}$-submodule of $J(\Psi)_{z_0}$ such that $I\big(\Psi+\varphi\big)_{z_0}\subset J_{z_0}$,
where $z_0\in Z_0$.
Let $J$ be the $\mathcal{O}_{M,z_0}$-module sheaf with stalks $J_{z_0}$, where $z_0\in Z_0$.
Denote the \textbf{minimal $L^{2}$ integral} related to $J$
\begin{equation}
\label{def of g(t) for boundary pt}
\begin{split}
\inf\Bigg\{ \int_{ \{ \Psi<-t\}}|\tilde{f}|^2e^{-\varphi}c(-\Psi): \tilde{f}\in
H^0(\{\Psi<-t\},\mathcal{O} (K_M)  ) \\
\&\, (\tilde{f}-f)_{z_0}\in
\mathcal{O} (K_M)_{z_0} \otimes J_{z_0},\text{for any }  z_0\in Z_0 \Bigg\}
\end{split}
\end{equation}
by $G(t;c,\Psi,\varphi,J,f)$, where $t\in[T,+\infty)$, $c$ is a nonnegative function on $(T,+\infty)$ and $|f|^2:=\sqrt{-1}^{n^2}f\wedge \bar{f}$ for any $(n,0)$ form $f$.
Without misunderstanding, we denote $G(t;c,\Psi,\varphi,J,f)$ by $G(t)$ for simplicity. For various $c(t)$, we denote $G(t;c,\Psi,\varphi,J,f)$ by $G(t;c)$ respectively for simplicity.

In this article, we obtain the following concavity property of $G(t)$.

\begin{Theorem}
\label{main theorem}
Let $c\in\tilde{P}_{T,M,\Psi}$. If there exists $t \in [T,+\infty)$ satisfying that $G(t)<+\infty$, then $G(h^{-1}(r))$ is concave with respect to  $r\in (\int_{T_1}^{T}c(t)e^{-t}dt,\int_{T_1}^{+\infty}c(t)e^{-t}dt)$, $\lim\limits_{t\to T+0}G(t)=G(T)$ and $\lim\limits_{t \to +\infty}G(t)=0$, where $h(t)=\int_{T_1}^{t}c(t_1)e^{-t_1}dt_1$ and $T_1 \in (T,+\infty)$.
\end{Theorem}

When $M$ is a pseudoconvex domain $D$ in $\mathbb{C}^n$, $\varphi\equiv0$, $c(t)\equiv 1$ and $T=0$, Theorem \ref{main theorem} degenerates to the concavity property in \cite{BGY-boundary}. When $c(t)\in P_{T,M,\Psi}$, Theorem \ref{main theorem} degenerates to the concavity property in \cite{GMY-boundary2} (Theorem 1.2 in \cite{GMY-boundary2}).
\begin{Remark} It follows from Theorem 1.2 and Remark 1.8 in \cite{GMY-boundary2} that the assumption ``$\{\psi<-t\}\backslash Z_0$ is a weakly pseudoconvex K\"ahler manifold for any $t\in\mathbb{R}$'' in \cite{BGMY7} can be removed.
\end{Remark}

\begin{Remark}
	\label{infty2}Let $c\in\tilde{P}_{T,M,\Psi}$.	If  $\int_{T_1}^{+\infty}c(t)e^{-t}dt=+\infty$ and $f_{z_0}\notin
\mathcal{O} (K_M)_{z_0} \otimes J_{z_0}$ for some  $ z_0\in Z_0$, then $G(t)=+\infty$ for any $t\geq T$. Thus, when there exists $t \in [T,+\infty)$ satisfying that $G(t)\in(0,+\infty)$, we have $\int_{T_1}^{+\infty}c(t)e^{-t}dt<+\infty$ and $G(\hat{h}^{-1}(r))$ is concave with respect to  $r\in (0,\int_{T}^{+\infty}c(t)e^{-t}dt)$, where $\hat{h}(t)=\int_{t}^{+\infty}c(l)e^{-l}dl$.
\end{Remark}

Let $c(t)$ be a nonnegative measurable function on $(T,+\infty)$. Denote that
\begin{equation}\nonumber
\begin{split}
\mathcal{H}^2(t;c):=\Bigg\{\tilde{f}:\int_{ \{ \Psi<-t\}}|\tilde{f}|^2e^{-\varphi}c(-\Psi)<+\infty,\  \tilde{f}\in
H^0(\{\Psi<-t\},\mathcal{O} (K_M)  ) \\
\& (\tilde{f}-f)_{z_0}\in
\mathcal{O} (K_M)_{z_0} \otimes J_{z_0},\text{for any }  z_0\in Z_0  \Bigg\},
\end{split}
\end{equation}
where $t\in[T,+\infty)$.

As a corollary of Theorem \ref{main theorem}, we give a necessary condition for the concavity property degenerating to linearity.
\begin{Corollary}
\label{necessary condition for linear of G}
Let $c\in\tilde{P}_{T,M,\Psi}$.
Assume that $G(t)\in(0,+\infty)$ for some $t\ge T$, and $G(\hat{h}^{-1}(r))$ is linear with respect to $r\in[0,\int_T^{+\infty}c(s)e^{-s}ds)$, where $\hat{h}(t)=\int_{t}^{+\infty}c(l)e^{-l}dl$.

Then there exists a unique holomorphic $(n,0)$ form $\tilde{F}$ on $\{\Psi<-T\}$
such that $(\tilde{F}-f)_{z_0}\in\mathcal{O} (K_M)_{z_0} \otimes J_{z_0}$ holds for any  $z_0\in Z_0$,
and $G(t)=\int_{\{\Psi<-t\}}|\tilde{F}|^2e^{-\varphi}c(-\Psi)$ holds for any $t\ge T$.

Furthermore
\begin{equation}
\begin{split}
  \int_{\{-t_1\le\Psi<-t_2\}}|\tilde{F}|^2e^{-\varphi}a(-\Psi)=\frac{G(T_1;c)}{\int_{T_1}^{+\infty}c(t)e^{-t}dt}
  \int_{t_2}^{t_1}a(t)e^{-t}dt
  \label{other a also linear}
\end{split}
\end{equation}
holds for any nonnegative measurable function $a$ on $(T,+\infty)$, where $T\le t_2<t_1\le+\infty$ and $T_1 \in (T,+\infty)$.
\end{Corollary}

\begin{Remark}
\label{rem:linear}
If $\mathcal{H}^2(t_0;\tilde{c})\subset\mathcal{H}^2(t_0;c)$ for some $t_0\ge T$, we have
\begin{equation}
\begin{split}
  G(t_0;\tilde{c})=\int_{\{\Psi<-t_0\}}|\tilde{F}|^2e^{-\varphi}\tilde{c}(-\Psi)=
  \frac{G(T_1;c)}{\int_{T_1}^{+\infty}c(t)e^{-t}dt}
  \int_{t_0}^{+\infty}\tilde{c}(s)e^{-s}ds,
  \label{other c also linear}
\end{split}
\end{equation}
 where $\tilde{c}$ is a nonnegative measurable function on $(T,+\infty)$ and $T_1 \in (T,+\infty)$. Thus, if $\mathcal{H}^2(t;\tilde{c})\subset\mathcal{H}^2(t;c)$ for any $t>T$, then $G(\hat{h}^{-1}(r);\tilde c)$ is linear with respect to $r\in[0,\int_T^{+\infty}c(s)e^{-s}ds)$.
\end{Remark}

\subsection{An application: a characterization for the concavity degenerating to linearity on open Riemann surfaces}

In this section, we give a characterization for the concavity degenerating to linearity on open Riemann surfaces.

Let $\Omega$ be an open Riemann surface, and let $K_{\Omega}$ be the canonical (holomorphic) line bundle on $\Omega$.  Let $dV_{\Omega}$ be a continuous volume form on $\Omega$.  Let $\psi$ be a  subharmonic function on $\Omega$, and let $\varphi$ be a Lebesgue measurable function on $\Omega$ such that $\varphi+\psi$ is subharmonic on $\Omega$. Let $F$ be a holomorphic function on $\Omega$. Let $T\in[-\infty,+\infty)$ such that $-T\le\sup\{\psi(z)-2\log|F(z)|:z\in\Omega \,\&\,F(z)\not=0\}$. Denote that
$$\Psi:=\min\{\psi-2\log|F|,-T\}.$$
For any $z\in \Omega$ satisfying $F(z)=0$, we set $\Psi(z)=-T$. Note that $\Psi$ is subharmonic function on $\{\Psi<-T\}$.

Let $Z_0$ be a  subset of $ \cap_{t>T}\overline{\{\Psi<-t\}}$. Denote that $Z_1:=\{z\in Z_0:v(dd^c(\psi),z)\ge2ord_{z}(F)\}$ and $Z_2:=\{z\in Z_0:v(dd^c(\psi),z)<2ord_{z}(F)\},$ where $d^c=\frac{\partial-\bar\partial}{2\pi\sqrt{-1}}$ and $v(dd^c(\psi),z)$ is the Lelong number of $dd^c(\psi)$ at $z$ (see \cite{Demaillybook}). Assume that  $Z'_1:=\{z\in Z_0:v(dd^c(\psi),z)>2ord_{z}(F)\}$ is a finite subset of $\Omega$. Note that $\{\Psi<-t\}\cup Z'_1$ is an open Riemann surface for any $t\ge T.$

Let $c(t)$ be a positive measurable function on $(T,+\infty)$ such that
$c(t)e^{-t}$ is decreasing on $(T,+\infty)$,
$c(t)e^{-t}$ is integrable near $+\infty$,
and $e^{-\varphi}c(-\Psi)$ has a positive lower bound on $K\cap\{\Psi<-T\}$ for any compact subset $K$ of $\Omega\backslash E$,
where $E\subset\{\Psi=-\infty\}$ is a discrete subset of $\Omega$.

Let $f$ be a holomorphic $(1,0)$ form on $\{\Psi<-t_0\}\cap V$, where $V\supset Z_0$ is an open subset of $\Omega$ and $t_0>T$
is a real number.
Let $J_{z}$ be an $\mathcal{O}_{\Omega,z}$-submodule of $H_{z}$ such that $I(\Psi+\varphi)_{z}\subset J_{z}$,
where $z\in Z_0$ and $H_{z}:=\{h_{z}\in J(\Psi)_{z}:\int_{\{\Psi<-t\}\cap U}|h|^2e^{-\varphi}c(-\Psi)dV_{\Omega}<+\infty$ for some $t>T$ and some neighborhood $U$ of $z\}$.
Denote
\begin{equation}
\label{def of g(t) for boundary pt}
\begin{split}
\inf\Bigg\{ \int_{ \{ \Psi<-t\}}|\tilde{f}|^2&e^{-\varphi}c(-\Psi): \tilde{f}\in
H^0(\{\Psi<-t\},\mathcal{O} (K_{\Omega})  ) \\
&\&\, (\tilde{f}-f)_{z}\in
\mathcal{O} (K_{\Omega})_{z} \otimes J_{z}\text{ for any }  z\in Z_0 \Bigg\}
\end{split}
\end{equation}
by $G(t;c,\Psi,\varphi,J,f)$, where $t\in[T,+\infty)$ and $|f|^2:=\sqrt{-1}f\wedge \bar{f}$ for any $(1,0)$ form $f$.
Without misunderstanding, we denote $G(t;c,\Psi,\varphi,J,f)$ by $G(t)$ for simplicity.

Recall that $G(h^{-1}(r))$ is concave with respect to $r$ (Theorem \ref{main theorem}), where  $h(t)=\int_{t}^{+\infty}c(s)e^{-s}ds$ for any $t\ge T$.
We obtain a characterization for $G(h^{-1}(r))$ degenerating to linearity.

\begin{Theorem}
	\label{thm:linearity1}
For any $z\in Z_1$,	assume that one of the following two conditions holds:
	
	$(A)$ $\varphi+a\psi$ is  subharmonic near $z$ for some $a\in[0,1)$;
	
	$(B)$ $(\psi-2p_z\log|w|)(z)>-\infty$, where $p_z=\frac{1}{2}v(dd^c(\psi),z)$ and $w$ is a local coordinate on a neighborhood of $z$ satisfying that $w(z)=0$.
	
 If there exists $t_1\ge T$ such that $G(t_1)\in(0,+\infty)$, then	$G(h^{-1}(r))$ is linear with respect to $r\in(0,\int_T^{+\infty}c(s)e^{-s}ds)$ if and only if the following statements hold:

	$(1)$  $\varphi+\psi=2\log|g|+2\log|F|$ on $\{\Psi<-T\}\cup Z'_1$, $J_{z}=I(\varphi+\Psi)_z$ for any $z\in Z'_1$ and there exists a holomorphic $(1,0)$ form $f_1$ on $(\{\Psi<-t_0\}\cap V)\cup Z'_1$ such that $f_1=f$ on $\{\Psi<-t_0\}\cap V$, where $g$ is a holomorphic function  on $\{\Psi<-T\}\cup Z'_1$ such that $ord_z(g)=ord_z(f_1)+1$ for any $z\in Z'_1$;
	
	$(2)$ $Z'_1\not=\emptyset$ and $\psi=2\sum_{ z\in Z'_1}\big(p_z-ord_{z}(F)\big)G_{\Omega_t}(\cdot, z)+2\log|F|-t$ on $\Omega_t$ for any $t>T$, where  $\Omega_t=\{\Psi<-t\}\cup Z'_1$ and $G_{\Omega_t}(\cdot, z)$ is the Green function on $\Omega_t$;
	
	$(3)$ $\frac{p_z-ord_{z}(F)}{ord_{z}(g)}\lim_{z'\rightarrow z}\frac{dg(z')}{f(z')}=c_0$ for any $z\in Z'_1$, where $c_0\in\mathbb{C}\backslash\{0\}$ is a constant independent of $z\in Z'_1$.
\end{Theorem}

When $F\equiv1$, $\psi(z)=-\infty$ for any $z\in Z_0=Z'_1$ and condition $(B)$ holds,  Theorem \ref{thm:linearity1} can be referred to \cite{GY-concavity3} (see also Theorem \ref{thm:m-points} and Remark \ref{r:equivalent}).

\section{preparations}
In this section, we do some preparations.
\subsection{$L^2$ method}
Let $M$ be an $n-$dimensional weakly pseudoconvex K\"ahler manifolds. Let $\psi$ be a plurisubharmonic function on $M$. Let $F$ be a holomorphic function on $M$. We assume that $F$ is not identically zero. Let $\varphi$ be a Lebesgue measurable function on $M$ such that $\varphi+\psi$ is a plurisubharmonic function on $M$.

Let $\delta$ be a positive integer. Let $T_1$ be a real number. Denote
$$\varphi_1:=(1+\delta)\max\{\psi+T_1,2\log|F|\},$$
and
$$\Psi_1:=\min\{\psi-2\log|F|,-T_1\}.$$
If $F(z)=0$ for some $z \in M$, we set $\Psi_1(z)=-T_1$.

Let $c(t)$ be a positive measurable function on $[T_1,+\infty)$ such that $c(t)e^{-t}$ is decreasing with respect $t$. We have the following lemma.

\begin{Lemma}[see \cite{GMY-boundary2}]
\label{L2 method}
Let $B\in(0,+\infty)$ and $t_0>T_1$ be arbitrarily given. Let $f$ be a holomorphic $(n,0)$ form on $\{\Psi_1<-t_0\}$ such that
$$\int_{\{\Psi_1<-t_0\}\cap K}|f|^2<+\infty,$$
for any compact subset $K\subset M$, and
$$\int_{M}\frac{1}{B}\mathbb{I}_{\{-t_0-B<\Psi_1<-t_0\}}|f|^2e^{-\varphi-\Psi}<+\infty.$$
 Then there exists a holomorphic $(n,0)$ form $\tilde{F}$ on $M$ such that
\begin{equation*}
  \begin{split}
      & \int_{M}|\tilde{F}-(1-b_{t_0,B}(\Psi_1))fF^{1+\delta}|^2e^{-\varphi-\varphi_1+v_{t_0,B}(\Psi_1)-\Psi_1}c(-v_{t_0,B}(\Psi_1)) \\
      \le & (\frac{1}{\delta}c(T_1)e^{-T_1}+\int_{T_1}^{t_0+B}c(s)e^{-s}ds)
       \int_{M}\frac{1}{B}\mathbb{I}_{\{-t_0-B<\Psi<-t_0\}}|f|^2e^{-\varphi-\Psi_1},
  \end{split}
\end{equation*}
where $b_{t_0,B}(t)=\int^{t}_{-\infty}\frac{1}{B} \mathbb{I}_{\{-t_0-B< s < -t_0\}}ds$,
$v_{t_0,B}(t)=\int^{t}_{-t_0}b_{t_0,B}(s)ds-t_0$.
\end{Lemma}

Let $T\in[-\infty,+\infty)$. Denote $$\Psi:=\min\{\psi-2\log|F|,-T\}.$$
If $F(z)=0$ for some $z \in M$, we set $\Psi(z)=-T$. Let $c(t)\in \tilde{P}_{T,M,\Psi}$.
Let $T_1>T$ be a real number. Denote
$$\varphi_1:=(1+\delta)\max\{\psi+T_1,2\log|F|\},$$
and
$$\Psi_1:=\min\{\psi-2\log|F|,-T_1\}.$$
If $F(z)=0$ for some $z \in M$, we set $\Psi_1(z)=-T_1$.
It follows from Lemma \ref{L2 method} that we have the following lemma.
\begin{Lemma}
Let $(M,X,Z)$  satisfies condition $(A)$. Let $B \in (0, +\infty)$ and $t_0> T_1>T$ be arbitrarily given.
Let $f$ be a holomorphic $(n,0)$ form on $\{\Psi< -t_0\}$ such that
\begin{equation}
\int_{\{\Psi<-t_0\}} {|f|}^2e^{-\varphi}c(-\Psi)<+\infty,
\label{condition of lemma 2.2}
\end{equation}
Then there exists a holomorphic $(n,0)$ form $\tilde{F}$ on $M$  such that
\begin{equation}
\begin{split}
&\int_{M}|\tilde{F}-(1-b_{t_0,B}(\Psi_1))fF^{1+\delta}|^2e^{-\varphi-\varphi_1-\Psi_1+v_{t_0,B}(\Psi_1)}c(-v_{t_0,B}(\Psi_1))\\
\le & \left(\frac{1}{\delta}c(T_1)e^{-T_1}+\int_{T_1}^{t_0+B}c(t)e^{-t}dt\right)\int_M \frac{1}{B} \mathbb{I}_{\{-t_0-B< \Psi_1 < -t_0\}}  {|f|}^2
e^{{-}\varphi-\Psi_1},
\end{split}
\end{equation}
where  $b_{t_0,B}(t)=\int^{t}_{-\infty}\frac{1}{B} \mathbb{I}_{\{-t_0-B< s < -t_0\}}ds$ and
$v_{t_0,B}(t)=\int^{t}_{-t_0}b_{t_0,B}(s)ds-t_0$.
\label{L2 method for c(t)}
\end{Lemma}
\begin{proof} We note that $\{\Psi<-t_0\}=\{\Psi_1<-t_0\}$ and $\Psi_1=\Psi$ on $\{\Psi<-t_0\}$. It follows from inequality \eqref{condition of lemma 2.2} and $c(t)e^{-t}$ is decreasing with respect to $t$ that
$$\int_M \frac{1}{B} \mathbb{I}_{\{-t_0-B< \Psi_1 < -t_0\}}  {|f|}^2
e^{{-}\varphi-\Psi}<+\infty.$$

  As $c(t)\in \tilde{P}_{T,M,\Psi}$, $\{\Psi<-t_0\}=\{\Psi_1<-t_0\}$ and $\Psi_1=\Psi$ on $\{\Psi<-t_0\}$, there  exists a closed subset $E\subset Z\cap \{\Psi=-\infty\}$ such that for any compact subset $K\subset M\backslash E$, $e^{-\varphi}c(-\Psi)$ has a positive lower bound on $K\cap \{\Psi_1<-t_0\} $.
It follows from inequality \eqref{condition of lemma 2.2} that we have
\begin{equation}\nonumber
\int_{K\cap \{\Psi_1<-t_0\}}|f|^2<+\infty.
\end{equation}

As $(M,X,Z)$  satisfies condition $(A)$,  $M\backslash (Z\cup X)$ is a weakly pseudoconvex K\"ahler manifold. It follows from Lemma \ref{L2 method} that there exists a holomorphic $(n,0)$ form $\tilde{F}_Z$ on $M\backslash (Z\cup X)$ such that
\begin{equation*}
  \begin{split}
      & \int_{M\backslash (Z\cup X)}|\tilde{F}_Z-(1-b_{t_0,B}(\Psi_1))fF^{1+\delta}|^2e^{-\varphi-\varphi_1+v_{t_0,B}(\Psi_1)-\Psi_1}c(-v_{t_0,B}(\Psi_1)) \\
      \le & (\frac{1}{\delta}c(T_1)e^{-T_1}+\int_{T_1}^{t_0+B}c(s)e^{-s}ds)
       \int_{M}\frac{1}{B}\mathbb{I}_{\{-t_0-B<\Psi_1<-t_0\}}|f|^2e^{-\varphi-\Psi_1}<+\infty.
  \end{split}
\end{equation*}

For any $z\in \left((Z\cup X)\backslash E\right)$, there exists an open neighborhood $V_z$ of $z$ such that $V_z\Subset M\backslash E$.

It follows from inequality \eqref{condition of lemma 2.2} and $c(t)\in \tilde{P}_{T,M,\Psi}$ that we have $\int_{V_z\cap \{\Psi_1<-t_0\}}|f|^2<+\infty$.
Note that $\varphi+\varphi_1+\Psi$ is a plurisubharmonic function on $M$. As $c(t)e^{-t}$ is decreasing with respect to $t$ and $v_{t_0,B}(\Psi_1)\ge -t_0-\frac{B}{2}$, we have $c(-v_{t_0,B}(\Psi_1))e^{v_{t_0,B}(\Psi_1)}\ge c(t_0+\frac{B}{2})e^{-t_0-\frac{B}{2}}>0$. Denote $C:=\inf\limits_{V_z}e^{-\varphi-\varphi_1+v_{t_0,B}(\Psi_1)-\Psi_1}c(-v_{t_0,B}(\Psi_1))$, we know $C>0$. Then we have
\begin{equation*}
  \begin{split}
  &\int_{V_z\backslash (Z\cup X)}|\tilde{F}_Z|^2\\
  \le & 2\int_{V_z\backslash (Z\cup X)}|\tilde{F}_Z-(1-b_{t_0,B}(\Psi_1))fF^{1+\delta}|^2
  +2\int_{V_z\backslash (Z\cup X)}|(1-b_{t_0,B}(\Psi_1))fF^{1+\delta}|^2
 \\
  \le &
  \frac{2}{C}\bigg(\int_{M\backslash (Z\cup X)}|\tilde{F}_Z-(1-b_{t_0,B}(\Psi_1))fF^{1+\delta}|^2e^{-\varphi-\varphi_1+v_{t_0,B}(\Psi_1)-\Psi_1}c(-v_{t_0,B}(\Psi_1))\bigg)\\
  &+\sup_{V_z}|F^{1+\delta}|^2\int_{\{\Psi_1<-t_0\}\cap V_z}|f|^2\\
  <&+\infty.
  \end{split}
\end{equation*}

As $Z\cup X$ is locally negligible with respect to $L^2$ holomorphic function, we can find a holomorphic extension $\tilde{F}_E$ of $\tilde{F}_Z$ from $M\backslash (Z\cup X)$ to $M\backslash E$ such that

\begin{equation*}
  \begin{split}
      & \int_{M\backslash E}|\tilde{F}_E-(1-b_{t_0,B}(\Psi_1))fF^{1+\delta}|^2e^{-\varphi-\varphi_1+v_{t_0,B}(\Psi_1)-\Psi_1}c(-v_{t_0,B}(\Psi_1)) \\
      \le & (\frac{1}{\delta}c(T_1)e^{-T_1}+\int_{T_1}^{t_0+B}c(s)e^{-s}ds)
       \int_{M}\frac{1}{B}\mathbb{I}_{\{-t_0-B<\Psi_1<-t_0\}}|f|^2e^{-\varphi-\Psi}.
  \end{split}
\end{equation*}

 Note that $E\subset\{\Psi=-\infty\}\subset\{\Psi<-t_0\}$ and $\{\Psi<-t_0\}$ is open, then for any $z\in E$, there exists an open neighborhood $U_z$ of $z$ such that $U_z\Subset\{\Psi<-t_0\}=\{\Psi_1<-t_0\}$.

As $v_{t_0,B}(t)\ge -t_0-\frac{B}{2}$, we have $c(-v_{t_0,B}(\Psi_1))e^{v_{t_0,B}(\Psi_1)}\ge c(t_0+\frac{B}{2})e^{-t_0-\frac{B}{2}}>0$. Note that $\varphi+\varphi_1+\Psi_1$ is plurisubharmonic on $M$. Thus we have
\begin{equation*}
  \begin{split}
      & \int_{U_z\backslash E}|\tilde{F}_E-(1-b_{t_0,B}(\Psi))fF^{1+\delta}|^2 \\
      \le & \frac{1}{C_1}\int_{U_z\backslash E}|\tilde{F}_E-(1-b_{t_0,B}(\Psi_1))fF^{1+\delta}|^2e^{-\varphi-\varphi_1+v_{t_0,B}(\Psi_1)-\Psi_1}c(-v_{t_0,B}(\Psi_1))
      <+\infty,
  \end{split}
\end{equation*}
where $C_1$ is some positive number.

As $U_z\Subset\{\Psi<-t_0\}$, we have
\begin{equation*}
  \begin{split}
       \int_{U_z\backslash E}|(1-b_{t_0,B}(\Psi))fF^{1+\delta}|^2
      \le
      \left(\sup_{U_z}|F^{1+\delta}|^2\right)\int_{U_z}|f|^2 <+\infty.
  \end{split}
\end{equation*}
Hence we have
$$\int_{U_z\backslash E}|\tilde{F}_E|^2<+\infty. $$
As $E$ is contained in some analytic subset of $M$, we can find a holomorphic extension $\tilde{F}$ of $\tilde{F}_E$ from $M\backslash E$ to $M$ such that

\begin{equation}
\begin{split}
&\int_{M}|\tilde{F}-(1-b_{t_0,B}(\Psi_1))fF^{1+\delta}|^2e^{-\varphi-\varphi_1-\Psi_1+v_{t_0,B}(\Psi_1)}c(-v_{t_0,B}(\Psi_1))\\
\le & \left(\frac{1}{\delta}c(T_1)e^{-T_1}+\int_{T_1}^{t_0+B}c(t)e^{-t}dt\right)\int_M \frac{1}{B} \mathbb{I}_{\{-t_0-B< \Psi_1 < -t_0\}}  {|f|}^2
e^{{-}\varphi-\Psi_1}.
\end{split}
\end{equation}
Lemma \ref{L2 method for c(t)} is proved.
\end{proof}

Let $T\in[-\infty,+\infty)$. Let $c(t)\in \tilde{P}_{T,M,\Psi}$.
Using Lemma \ref{L2 method for c(t)}, we have the following lemma, which will be used to prove Theorem \ref{main theorem}.
\begin{Lemma}
\label{L2 method in JM concavity}
Let $(M,X,Z)$  satisfies condition $(A)$. Let $B \in (0, +\infty)$ and $t_0>t_1> T$ be arbitrarily given.
Let $f$ be a holomorphic $(n,0)$ form on $\{\Psi< -t_0\}$ such that
\begin{equation}
\int_{\{\Psi<-t_0\}} {|f|}^2e^{-\varphi}c(-\Psi)<+\infty,
\label{condition of JM concavity}
\end{equation}
Then there exists a holomorphic $(n,0)$ form $\tilde{F}$ on $\{\Psi<-t_1\}$ such that

 \begin{equation*}
  \begin{split}
      & \int_{\{\Psi<-t_1\}}|\tilde{F}-(1-b_{t_0,B}(\Psi))f|^2e^{-\varphi+v_{t_0,B}(\Psi)-\Psi}c(-v_{t_0,B}(\Psi)) \\
      \le & \left(\int_{t_1}^{t_0+B}c(s)e^{-s}ds\right)
       \int_{M}\frac{1}{B}\mathbb{I}_{\{-t_0-B<\Psi<-t_0\}}|f|^2e^{-\varphi-\Psi},
  \end{split}
\end{equation*}
where $b_{t_0,B}(t)=\int^{t}_{-\infty}\frac{1}{B} \mathbb{I}_{\{-t_0-B< s < -t_0\}}ds$,
$v_{t_0,B}(t)=\int^{t}_{-t_0}b_{t_0,B}(s)ds-t_0$.
\end{Lemma}

\begin{proof}[Proof of Lemma \ref{L2 method in JM concavity}]

Denote that
$$\tilde{\varphi}:=\varphi+(1+\delta)\max\{\psi+t_1,2\log|F|\}$$
and
$$\tilde{\Psi}:=\min\{\psi-2\log|F|,-t_1\}.$$

As $t_0>t_1>T$, we have $\{\tilde{\Psi}<-t_0\}=\{\Psi<-t_0\}$. It follows from inequality \eqref{condition of JM concavity} and Lemma \ref{L2 method for c(t)} that there exists a holomorphic function $\tilde{F}_{\delta}$ on $M$ such that

\begin{equation*}
  \begin{split}
      & \int_{M}|\tilde{F}_{\delta}-(1-b_{t_0,B}(\tilde\Psi))fF^{1+\delta}|^2e^{-\tilde \varphi+v_{t_0,B}(\tilde \Psi)-\tilde \Psi}c(-v_{t_0,B}(\tilde \Psi)) \\
      \le & \left(\frac{1}{\delta}c(t_1)e^{-t_1}+\int_{t_1}^{t_0+B}c(s)e^{-s}ds\right)
       \int_{M}\frac{1}{B}\mathbb{I}_{\{-t_0-B<\tilde\Psi<-t_0\}}|f|^2e^{-\varphi-\tilde\Psi}.
  \end{split}
\end{equation*}

Note that on $\{\Psi<-t_1\}$, we have $\Psi=\tilde{\Psi}=\psi-2\log|F|$ and $\tilde{\varphi}=\varphi+\varphi_1=\varphi+(1+\delta)2\log|F|$.  Hence
\begin{equation}\label{1st formula in L2 method JM concavity}
\begin{split}
    & \int_{\{\Psi<-t_1\}}|\tilde{F}_{\delta}-(1-b_{t_0,B}(\Psi))fF^{1+\delta}|^2
    e^{-\varphi-\varphi_1+v_{t_0,B}(\Psi)-\Psi}c(-v_{t_0,B}(\Psi)) \\
    = & \int_{\{\Psi<-t_1\}}|\tilde{F}_{\delta}-(1-b_{t_0,B}(\tilde\Psi))fF^{1+\delta}|^2
     e^{-\tilde\varphi+v_{t_0,B}(\tilde\Psi)-\tilde\Psi}c(-v_{t_0,B}(\tilde\Psi))\\
     \le &
     \int_{M}|\tilde{F}_{\delta}-(1-b_{t_0,B}(\tilde\Psi))fF^{1+\delta}|^2
     e^{-\tilde\varphi+v_{t_0,B}(\tilde\Psi)-\tilde\Psi}c(-v_{t_0,B}(\tilde\Psi))\\
     \le & \left(\frac{1}{\delta}c(t_1)e^{-t_1}+\int_{t_1}^{t_0+B}c(s)e^{-s}ds\right)
       \int_{M}\frac{1}{B}\mathbb{I}_{\{-t_0-B<\tilde\Psi<-t_0\}}|f|^2e^{-\varphi-\tilde\Psi}\\
       =&
     \left(\frac{1}{\delta}c(t_1)e^{-t_1}+\int_{t_1}^{t_0+B}c(s)e^{-s}ds\right)
       \int_{M}\frac{1}{B}\mathbb{I}_{\{-t_0-B<\Psi<-t_0\}}|f|^2e^{-\varphi-\Psi}<+\infty.
\end{split}
\end{equation}

 Let $F_{\delta}:=\frac{\tilde{F}_{\delta}}{F^{\delta}}$ be a holomorphic function on $\{\Psi<-t_1\}$. Then it follows from \eqref{1st formula in L2 method JM concavity} that
\begin{equation}\label{2nd formula in L2 method JM concavity}
\begin{split}
    & \int_{\{\Psi<-t_1\}}|F_{\delta}-(1-b_{t_0,B}(\Psi))fF|^2
    e^{-\varphi+v_{t_0,B}(\Psi)-\psi}c(-v_{t_0,B}(\Psi)) \\
    \le&
     \bigg(\frac{1}{\delta}c(t_1)e^{-t_1}+\int_{t_1}^{t_0+B}c(s)e^{-s}ds\bigg)
       \int_{M}\frac{1}{B}\mathbb{I}_{\{-t_0-B<\Psi<-t_0\}}|f|^2e^{-\varphi-\Psi}.
\end{split}
\end{equation}

As $c(t)\in \tilde{P}_{T,M,\Psi}$,
 there exists a closed subset $E_0$ of $M$ such that $E_0\subset Z\cap \{\Psi(z)=-\infty\}$ and for any compact subset $K\subset M\backslash E_0$, $e^{-\varphi}c(-\Psi)$ has a positive lower bound on $K\cap \{\Psi<-t_0\}$ .
Let $K$ be any compact subset of $M\backslash E_0$. Note that $\inf_{K}e^{-\varphi_{\alpha}+v_{t_0,B}(\Psi)-\psi}c(-v_{t_0,B}(\Psi))\ge \left(c(t_0+\frac{2}{B})e^{-t_0-\frac{2}{B}}\right)\inf_{K}e^{-\varphi_{\alpha}-\psi}>0$. It follows from \eqref{2nd formula in L2 method JM concavity} that we have
 $$\sup_{\delta} \int_{\{\Psi<-t_1\}\cap K}|F_{\delta}-(1-b_{t_0,B}(\Psi))fF|^2<+\infty.$$

We also note that
$$\int_{\{\Psi<-t_1\}\cap K}|(1-b_{t_0,B}(\Psi))fF|^2\le
\left(\sup_{K}|F|^2\right)\int_{\{\Psi<-t_0\}\cap K}|f|^2<+\infty.$$
Then we know that
$$\sup_{\delta} \int_{\{\Psi<-t_1\}\cap K}|F_{\delta}|^2<+\infty,$$
and there exists a subsequence of $\{F_\delta\}$ (also denoted by $F_\delta$) compactly convergent to a holomorphic $(n,0)$ form $\tilde{F}_1$ on $\{\Psi<-t_1\}\backslash E_0$.
It follows from Fatou's Lemma and inequality \eqref{2nd formula in L2 method JM concavity} that we have

\begin{equation}\label{3rd formula in L2 method JM concavity}
\begin{split}
  & \int_{\{\Psi<-t_1\}\backslash E_0}|\tilde{F}_1-(1-b_{t_0,B}(\Psi))fF|^2
    e^{-\varphi_{\alpha}+v_{t_0,B}(\Psi)-\psi}c(-v_{t_0,B}(\Psi)) \\
    \le &\liminf_{\delta\to +\infty} \int_{\{\Psi<-t_1\}\backslash E_0}|F_{\delta}-(1-b_{t_0,B}(\Psi))fF|^2
    e^{-\varphi_{\alpha}+v_{t_0,B}(\Psi)-\psi}c(-v_{t_0,B}(\Psi)) \\
     \le &\liminf_{\delta\to +\infty} \int_{\{\Psi<-t_1\}}|F_{\delta}-(1-b_{t_0,B}(\Psi))fF|^2
    e^{-\varphi_{\alpha}+v_{t_0,B}(\Psi)-\psi}c(-v_{t_0,B}(\Psi)) \\
       \le&\liminf_{\delta\to +\infty}
     \bigg(\frac{1}{\delta}c(t_1)e^{-t_1}+\int_{t_1}^{t_0+B}c(s)e^{-s}ds\bigg)
       \int_{M}\frac{1}{B}\mathbb{I}_{\{-t_0-B<\Psi<-t_0\}}|f|^2e^{-\varphi_{\alpha}-\Psi}\\
       \le &\left(\int_{t_1}^{t_0+B}c(s)e^{-s}ds\right)
       \int_{M}\frac{1}{B}\mathbb{I}_{\{-t_0-B<\Psi<-t_0\}}|f|^2e^{-\varphi_{\alpha}-\Psi}.
\end{split}
\end{equation}
Note that $E_0\subset\{\Psi=-\infty\}\subset\{\Psi<-t_1\}$ and $\{\Psi<-t_1\}$ is open, then for any $z\in E_0$, there exists an open neighborhood $U_z$ of $z$ such that $U_z\Subset\{\Psi<-t_1\}$. Note that $\varphi_{\alpha}+\psi$ is plurisubharmonic function on $M$. As $v_{t_0,B}(t)\ge -t_0-\frac{B}{2}$, we have $c(-v_{t_0,B}(\Psi_1))e^{v_{t_0,B}(\Psi_1)}\ge c(t_0+\frac{B}{2})e^{-t_0-\frac{B}{2}}>0$. Thus by \eqref{3rd formula in L2 method JM concavity}, we have
\begin{equation*}
  \begin{split}
      & \int_{U_z\backslash E_0}|\tilde{F}_1-(1-b_{t_0,B}(\Psi))fF|^2 \\
      \le  & \frac{1}{C_1}\int_{U_z\backslash E_0}|\tilde{F}_1-(1-b_{t_0,B}(\Psi))fF|^2e^{-\varphi_{\alpha}+v_{t_0,B}(\Psi)-\psi}c(-v_{t_0,B}(\Psi))
      <+\infty,
  \end{split}
\end{equation*}
where $C_1:=c(t_0+\frac{B}{2})e^{-t_0-\frac{B}{2}}\inf_{U_z}e^{-\varphi_{\alpha}-\psi}$ is some positive number.

As $U_z\Subset\{\Psi<-t_1\}$, we have
\begin{equation*}
  \begin{split}
       \int_{U_z\backslash E_0}|(1-b_{t_0,B}(\Psi))fF|^2
      \le
      \left(\sup_{U_z}|F|^2\right)\int_{U_z}|f|^2 <+\infty.
  \end{split}
\end{equation*}
Hence we have
$$\int_{U_z\backslash E_1}|\tilde{F}_1|^2<+\infty. $$
As $E_1$ is contained in some analytic subset of $M$, we can find a holomorphic extension $\tilde{F}_0$ of $\tilde{F}_1$ from $\{\Psi<-t_1\}\backslash E_0$ to $\{\Psi<-t_1\}$ such that
\begin{equation}\label{4th formula in L2 method JM concavity}
\begin{split}
  & \int_{\{\Psi<-t_1\}}|\tilde{F}_0-(1-b_{t_0,B}(\Psi))fF|^2
    e^{-\varphi_{\alpha}+v_{t_0,B}(\Psi)-\psi}c(-v_{t_0,B}(\Psi)) \\
       \le &\left(\int_{t_1}^{t_0+B}c(s)e^{-s}ds\right)
       \int_{M}\frac{1}{B}\mathbb{I}_{\{-t_0-B<\Psi<-t_0\}}|f|^2e^{-\varphi_{\alpha}-\Psi}.
\end{split}
\end{equation}

Denote $\tilde{F}:=\frac{\tilde{F}_0}{F}$. Note that on $\{\Psi<-t_1\}$, we have $\Psi=\psi-2\log|F|$. It follows from \eqref{4th formula in L2 method JM concavity} that we have
\begin{equation}\nonumber
\begin{split}
  & \int_{\{\Psi<-t_1\}}|\tilde{F}-(1-b_{t_0,B}(\Psi))f|^2
    e^{-\varphi_{\alpha}+v_{t_0,B}(\Psi)-\Psi}c(-v_{t_0,B}(\Psi)) \\
       \le &\left(\int_{t_1}^{t_0+B}c(s)e^{-s}ds\right)
       \int_{M}\frac{1}{B}\mathbb{I}_{\{-t_0-B<\Psi<-t_0\}}|f|^2e^{-\varphi_{\alpha}-\Psi}.
\end{split}
\end{equation}

Lemma \ref{L2 method in JM concavity} is proved.

\end{proof}

\subsection{Properties of $\mathcal{O}_{M,z_0}$-module $J_{z_0}$}
\label{sec:properties of module}
In this section, we present some properties of $\mathcal{O}_{M,z_0}$-module $J_{z_0}$.

Since the case is local, we assume that $F$ is a holomorphic function on a pseudoconvex domain $D\subset \mathbb{C}^n$ containing the origin $o\in \mathbb{C}^n$. Let $\psi$ be a plurisubharmonic function on $D$. Let $\varphi$ be a Lebesgue measurable function on $D$ such that $\psi+\varphi$ is plurisubharmonic. Let $T\in [-\infty,+\infty)$. Denote $$\Psi:=\min\{\psi-2\log|F|,-T\}.$$
If $F(z)=0$ for some $z \in M$, we set $\Psi(z)=-T$. Let $T_1>T$ be a real number.

Denote
$$\varphi_1:=2\max\{\psi+T_1,2\log|F|\},$$
and
$$\Psi_1:=\min\{\psi-2\log|F|,-T_1\}.$$
If $F(z)=0$ for some $z \in M$, we set $\Psi_1(z)=-T_1$. We also note that by definition $I(\Psi_1+\varphi)_o=I(\Psi+\varphi)_o$.

Let $c(t)$ be a positive measurable function on $(T,+\infty)$ such that $c(t)\in\tilde{P}_{T,D,\Psi}$.

Denote that $H_o:=\{f_o\in J(\Psi)_o:\int_{\{\Psi<-t\}\cap V_0}|f|^2e^{-\varphi}c(-\Psi)dV_M<+\infty \text{ for some }t>T \text{ and } V_0 \text{ is an open neighborhood of o}\}$ and
$\mathcal{H}_o:=\{(F,o)\in \mathcal{O}_{\mathbb{C}^n,o}:\int_{U_0}|F|^2e^{-\varphi-\varphi_1}c(-\Psi_1)dV_M<+\infty \text{ for some open neighborhood} \ U_0 \text{ of } o\}$.

As  $c(t)\in\tilde{P}_{T,D,\Psi}$, hence $c(t)e^{-t}$ is decreasing with respect to $t$ and we have $I(\Psi_1+\varphi)_o=I(\Psi+\varphi)_o\subset H_o$.
We also note that $\mathcal{H}_o$ is an ideal of $\mathcal{O}_{\mathbb{C}^n,o}$.

\begin{Lemma}\label{construction of morphism} For any $f_o\in H_o$, there exist a pseudoconvex domain  $D_0\subset D$ containing $o$ and a holomorphic function $\tilde{F}$ on $D_0$ such that $(\tilde{F},o)\in \mathcal{H}_o$ and
$$\int_{\{\Psi_1<-t_1\}\cap D_0}|\tilde{F}-fF^2|e^{-\varphi-\varphi_1-\Psi_1}<+\infty,$$
for some $t_1>T_1$.
\end{Lemma}
\begin{proof}It follows from $f_o\in H_o$ that there exists $t_0>T_1>T$ and a pseudoconvex domain  $D_0\Subset D$ containing $o$ such that
\begin{equation}\label{construction of morphism formula 1}
\int_{\{\Psi<-t_0\}\cap D_0}|f|^2e^{-\varphi}c(-\Psi)<+\infty.
\end{equation}
Then it follows from Lemma \ref{L2 method for c(t)} that there exists a holomorphic function $\tilde F$ on $D_0$ such that
\begin{equation*}
  \begin{split}
      & \int_{D_0}|\tilde{F}-(1-b_{t_0}(\Psi_1))fF^{2}|^2e^{-\varphi-\varphi_1+v_{t_0}(\Psi_1)-\Psi_1}c(-v_{t_0}(\Psi_1)) \\
      \le & \left(c(T_1)e^{-T_1}+\int_{T_1}^{t_0+1}c(s)e^{-s}ds\right)
       \int_{D_0}\mathbb{I}_{\{-t_0-1<\Psi_1<-t_0\}}|f|^2e^{-\varphi-\Psi_1},
  \end{split}
\end{equation*}
where $b_{t_0}(t)=\int^{t}_{-\infty} \mathbb{I}_{\{-t_0-1< s < -t_0\}}ds$,
$v_{t_0}(t)=\int^{t}_{-t_0}b_{t_0}(s)ds-t_0$. Denote $C:=c(T_1)e^{-T_1}+\int_{T_1}^{t_0+B}c(s)e^{-s}ds$, we note that $C$ is a positive number.

 As $v_{t_0}(t)>-t_0-1$, we have $e^{v_{t_0}(\Psi)}c(-v_{t_0}(\Psi))\ge c(t_0+1)e^{-(t_0+1)}>0$. As $b_{t_0}(t)\equiv 0$ on $(-\infty,-t_0-1)$, we have
\begin{equation}\label{construction of morphism formula 2}
\begin{split}
   &\int_{D_0\cap\{\Psi_1<-t_0-1\}}|\tilde{F}-fF^2|^2e^{-\varphi-\varphi_1-\Psi_1} \\
   \le & \frac{1}{c(t_0+1)e^{-(t_0+1)}}
   \int_{D_0}|\tilde{F}-(1-b_{t_0}(\Psi_1))fF^2|^2e^{-\varphi-\varphi_1-\Psi_1+v_{t_0}(\Psi_1)}c(-v_{t_0}(\Psi_1))\\
   \le &\frac{C}{c(t_0+1)e^{-(t_0+1)}}
   \int_{D_0}\mathbb{I}_{\{-t_0-1<\Psi_1<-t_0\}}|f|^2e^{-\varphi-\Psi_1}<+\infty.
\end{split}
\end{equation}
Note that on $\{\Psi_1<-t_0\}$, $|F|^4e^{-\varphi_1}=1$. As $v_{t_0}(\Psi_1)\ge \Psi_1$, we have $c(-v_{t_0}(\Psi_1))e^{v_{t_0}(\Psi_1)}\ge c(-\Psi_1)e^{-\Psi_1}$. Hence we have
\begin{equation}\nonumber
\begin{split}
   &\int_{D_0}|\tilde{F}|^2e^{-\varphi-\varphi_1}c(-\Psi_1) \\
   \le & 2\int_{D_0}|\tilde{F}-(1-b_{t_0}(\Psi_1))fF^2|^2e^{-\varphi-\varphi_1}c(-\Psi_1)\\
   +&2\int_{D_0}|(1-b_{t_0}(\Psi_1))fF^2|^2e^{-\varphi-\varphi_1}c(-\Psi_1)\\
   \le&
   2\int_{D_0}|\tilde{F}-(1-b_{t_0}(\Psi_1))fF^2|^2e^{-\varphi-\varphi_1-\Psi_1+v_{t_0}(\Psi_1)}c(-v_{t_0}(\Psi_1))\\
   +&2\int_{D_0\cap\{\Psi<-t_0\}}|f|^2e^{-\varphi}c(-\Psi)\\
   < &+\infty.
\end{split}
\end{equation}
Hence we know that $(\tilde{F},o)\in \mathcal{H}_o$.
\end{proof}

For any $(\tilde{F},o)\in\mathcal{H}_o$ and $(\tilde{F}_1,o)\in\mathcal{H}_o$ such that $\int_{D_1\cap\{\Psi_1<-t_1\}}|\tilde{F}-fF^2|^2e^{-\varphi-\varphi_1-\Psi_1}<+\infty$ and
$\int_{D_1\cap\{\Psi_1<-t_1\}}|\tilde{F}_1-fF^2|^2e^{-\varphi-\varphi_1-\Psi_1}<+\infty$, for some open neighborhood $D_1$ of $o$ and $t_1> T_1$, we have
$$\int_{D_1\cap\{\Psi_1<-t_1\}}|\tilde{F}_1-\tilde{F}|^2e^{-\varphi-\varphi_1-\Psi_1}<+\infty.$$
As $(\tilde{F},o)\in\mathcal{H}_o$ and $(\tilde{F}_1,o)\in\mathcal{H}_o$, there exists a neighborhood $D_2$ of $o$ such that
\begin{equation}\label{construction of morphism formula 3}
\int_{D_2}|\tilde{F}_1-\tilde{F}|^2e^{-\varphi-\varphi_1}c(-\Psi_1)<+\infty.
\end{equation}
Note that we have $c(-\Psi_1)e^{\Psi_1}\ge c(t_1)e^{-t_1}$ on $\{\Psi\ge-t_1\}$. It follows from inequality \eqref{construction of morphism formula 3} that we have
$$\int_{D_2\cap \{\Psi\ge-t_1\}}|\tilde{F}_1-\tilde{F}|^2e^{-\varphi-\varphi_1-\Psi_1}<+\infty.$$
Hence we have $(\tilde{F}-\tilde{F}_1,o)\in \mathcal{I}(\varphi+\varphi_1+\Psi_1)_o$.

Thus it follows from Lemma \ref{construction of morphism} that there exists a map $\tilde{P}:H_o\to \mathcal{H}_o/\mathcal{I}(\varphi+\varphi_1+\Psi_1)_o$ given by
$$\tilde{P}(f_o)=[(\tilde{F},o)]$$
for any $f_o\in H_o$, where $(\tilde{F},o)$ satisfies $(\tilde{F},o)\in \mathcal{H}_o$ and
$\int_{D_1\cap\{\Psi_1<-t_1\}}|\tilde{F}-fF^2|^2e^{-\varphi-\varphi_1-\Psi_1}<+\infty,$
for some $t_1>T_1$ and some open neighborhood $D_1$ of $o$, and $[(\tilde{F},o)]$ is the equivalence class of $(\tilde{F},o)$ in $\mathcal{H}_o/\mathcal{I}(\varphi+\varphi_1+\Psi_1)_o$.

\begin{Proposition}\label{proposition of morphism}
$\tilde{P}$ is an $\mathcal{O}_{\mathbb{C}^n,o}$-module homomorphism and $Ker(\tilde{P})=I(\varphi+\Psi_1)_o$.
\end{Proposition}
\begin{proof}For any $f_o,g_o\in H_o$. Denote that $\tilde{P}(f_o)=[(\tilde{F},o)]$, $\tilde{P}(g_o)=[(\tilde{G},o)]$ and $\tilde{P}(f_o+g_o)=[(\tilde{H},o)]$.

Note that there exists an open neighborhood $D_1$ of $o$ and $t> T_1$ such that $\int_{D_1\cap\{\Psi_1<-t\}}|\tilde{F}-fF^2|^2e^{-\varphi-\varphi_1-\Psi_1}<+\infty$,
$\int_{D_1\cap\{\Psi_1<-t\}}|\tilde{G}-gF^2|^2e^{-\varphi-\varphi_1-\Psi_1}<+\infty$, and
$\int_{D_1\cap\{\Psi_1<-t\}}|\tilde{H}-(f+g)F^2|^2e^{-\varphi-\varphi_1-\Psi_1}<+\infty$. Hence we have
$$\int_{D_1\cap\{\Psi_1<-t\}}|\tilde{H}-(\tilde{F}+\tilde{G})|^2e^{-\varphi-\varphi_1-\Psi_1}<+\infty.$$
As $(\tilde{F},o),(\tilde{G},o)$ and $(\tilde{H},o)$ belong to $ \mathcal{H}_o$, there exists an open neighborhood $\tilde{D}_1\subset D_1$ of $o$ such that
$\int_{\tilde{D}_1}|\tilde{H}-(\tilde{F}+\tilde{G})|^2e^{-\varphi-\varphi_1}c(-\Psi_1)<+\infty$.
As $c(t)e^{-t}$ is decreasing with respect to $t$, we have $c(-\Psi_1)e^{\Psi_1}\ge c(t)e^{-t}$ on $\{\Psi_1\ge -t\}$. Hence we have
$$\int_{\tilde{D}_1\cap\{\Psi_1\ge -t\}}|\tilde{H}-(\tilde{F}+\tilde{G})|^2e^{-\varphi-\varphi_1-\Psi_1}
\le\frac{1}{c(t)e^{-t}}\int_{\tilde{D}_1\cap\{\Psi_1\ge -t\}}|\tilde{H}-(\tilde{F}+\tilde{G})|^2e^{-\varphi-\varphi_1}c(-\Psi_1)<+\infty.$$
Thus we have $\int_{\tilde{D}_1}|\tilde{H}-(\tilde{F}+\tilde{G})|^2e^{-\varphi-\varphi_1-\Psi_1}<+\infty$, which implies that $\tilde{P}(f_o+g_o)=\tilde{P}(f_o)+\tilde{P}(g_o)$.

For any $(h,o) \in \mathcal{O}_{\mathbb{C}^n,o}$. Denote $\tilde{P}((hf)_o)=[(\tilde{F}_h,o)]$. Note that there exists an open neighborhood $D_2$ of $o$ and $t> T_1$ such that $\int_{D_2\cap\{\Psi_1<-t\}}|\tilde{F}_h-(hf)F^2|^2e^{-\varphi-\varphi_1-\Psi_1}<+\infty$. It follows from $\int_{D_2\cap\{\Psi_1<-t\}}|\tilde{F}-fF^2|^2e^{-\varphi-\varphi_1-\Psi_1}<+\infty$ and $h$ is holomorphic on $\overline{D_2}$ (shrink $D_2$ if necessary) that $\int_{D_2\cap\{\Psi_1<-t\}}|h\tilde{F}-hfF^2|^2e^{-\varphi-\varphi_1-\Psi_1}<+\infty$. Then we have
$$\int_{D_2\cap\{\Psi_1<-t\}}|\tilde{F}_h-h\tilde{F}|^2e^{-\varphi-\varphi_1-\Psi_1}<+\infty.$$
Note that $(h\tilde{F},o) $ and $(\tilde{F}_h,o)$ belong to $ \mathcal{H}_o$, we have
$\int_{D_2}|\tilde{F}_h-h\tilde{F}|^2e^{-\varphi-\varphi_1}c(-\Psi_1)<+\infty$.
As $c(t)e^{-t}$ is decreasing with respect to $t$, we have $c(-\Psi_1)e^{\Psi_1}\ge c(t)e^{-t}$ on $\{\Psi_1\ge -t\}$. Hence we have
$$\int_{D_2\cap\{\Psi_1\ge -t\}}|\tilde{F}_h-h\tilde{F}|^2e^{-\varphi-\varphi_1-\Psi_1}
\le\frac{1}{c(t)e^{-t}}\int_{D_2\cap\{\Psi_1\ge -t\}}|\tilde{F}_h-h\tilde{F}|^2e^{-\varphi-\varphi_1}c(-\Psi_1)<+\infty.$$
Thus we have $\int_{D_2}|\tilde{F}_h-h\tilde{F}|^2e^{-\varphi-\varphi_1-\Psi_1}<+\infty$, which implies that $\tilde{P}(hf_o)=(h,o)\tilde{P}(f_o)$.

Now we have proved that $\tilde{P}$ is an $\mathcal{O}_{\mathbb{C}^n,o}$-module homomorphism.

Next, we prove $Ker(\tilde{P})=I(\varphi+\Psi_1)_o$.

If $f_o\in I(\varphi+\Psi_1)_o$. Denote $\tilde{P}(f_o)=[(\tilde{F},o)]$. It follows from Lemma \ref{construction of morphism} that $(\tilde{F},o)\in \mathcal{H}_o$ and there exists an open neighborhood $D_3$ of $o$ and a real number $t_1>T_1$ such that
$$\int_{\{\Psi_1<-t_1\}\cap D_3}|\tilde{F}-fF^2|e^{-\varphi-\varphi_1-\Psi_1}<+\infty.$$
As $f_o\in I(\varphi+\Psi_1)_o$, shrink $D_3$ and $t_1$ if necessary, we have
\begin{equation}\label{proposition of morphism formula 1}
\begin{split}
&\int_{\{\Psi_1<-t_1\}\cap D_3}|\tilde{F}|^2e^{-\varphi-\varphi_1-\Psi_1}\\
\le &2\int_{\{\Psi_1<-t_1\}\cap D_3}|\tilde{F}-fF^2|^2e^{-\varphi-\varphi_1-\Psi_1}
+2\int_{\{\Psi_1<-t_1\}\cap D_3}|fF^2|^2e^{-\varphi-\varphi_1-\Psi_1}\\
\le &2\int_{\{\Psi_1<-t_1\}\cap D_3}|\tilde{F}-fF^2|^2e^{-\varphi-\varphi_1-\Psi_1}
+2\int_{\{\Psi_1<-t_1\}\cap D_3}|f|^2e^{-\varphi-\Psi_1}\\
<&+\infty.
\end{split}
\end{equation}
As $c(t)e^{-t}$ is decreasing with respect to $t$, $c(-\Psi_1)e^{\Psi_1}\ge C_0>0$ for some positive number $C_0$ on $\{\Psi_1\ge-t_1\}$. Then we have
\begin{equation}\label{proposition of morphism formula 1'}
\begin{split}
\int_{\{\Psi_1\ge-t_1\}\cap D_3}|\tilde{F}|^2e^{-\varphi-\varphi_1-\Psi_1}
\le\frac{1}{C_0}\int_{\{\Psi_1\ge-t_1\}\cap D_3}|\tilde{F}|^2e^{-\varphi-\varphi_1}c(-\Psi_1)<+\infty.
\end{split}
\end{equation}
Combining inequality \eqref{proposition of morphism formula 1} and inequality  \eqref{proposition of morphism formula 1'}, we know that $\tilde{F}\in \mathcal{I}(\varphi+\varphi_1+\Psi_1)_o$, which means $\tilde{P}(f_o)=0$ in $\mathcal{H}_o/\mathcal{I}(\varphi+\varphi_1+\Psi_1)_o$. Hence we know $I(\varphi+\Psi_1)_o\subset Ker(\tilde{P})$.

If $f_o\in Ker(\tilde{P})$, we know $\tilde{F}\in \mathcal{I}(\varphi+\varphi_1+\Psi_1)_o$.
We can assume that $\tilde{F}$ satisfies $\int_{D_4}|\tilde{F}|^2e^{-\varphi-\varphi_1-\Psi_1}<+\infty$ for some open neighborhood $D_4$ of $o$. Then we have
\begin{equation}\label{proposition of morphism formula 2'}
\begin{split}
&\int_{ \{\Psi_1<-t_1\}\cap D_4}|f|^2e^{-\varphi-\Psi_1}\\
=&\int_{\{\Psi_1<-t_1\}\cap D_4}|fF^2|^2e^{-\varphi-\varphi_1-\Psi_1}\\
\le & \int_{\{\Psi_1<-t_1\}\cap D_4}|\tilde{F}|^2e^{-\varphi-\varphi_1-\Psi_1}+\int_{\{\Psi_1<-t_1\}\cap D_4}|\tilde{F}-fF^2|e^{-\varphi-\varphi_1-\Psi_1}\\
< &+\infty.
\end{split}
\end{equation}
By definition, we know $f_o\in I(\varphi+\Psi_1)_o$. Hence $ Ker(\tilde{P})\subset I(\varphi+\Psi_1)_o$.

$ Ker(\tilde{P})= I(\varphi+\Psi_1)_o$ is proved.
\end{proof}

Now we can define an $\mathcal{O}_{\mathbb{C}^n,o}$-module homomorphism $P:H_o/I(\varphi+\Psi_1)_o\to \mathcal{H}_o/\mathcal{I}(\varphi+\varphi_1+\Psi_1)_o$ as follows,
$$P([f_o])=\tilde{P}(f_o)$$
for any $[f_o]\in H_o/I(\varphi+\Psi_1)_o$, where $f_o\in H_o$ is any representative of $[f_o]$. It follows from Proposition \ref{proposition of morphism} that $P([f_o])$ is independent of the choices of the representatives of $[f_o]$.

Let $(\tilde{F},o)\in \mathcal{H}_o$, i.e. $\int_{U}|\tilde{F}|^2e^{-\varphi-\varphi_1}c(-\Psi_1)<+\infty$ for some neighborhood $U$ of $o$. Note that $|F|^4e^{-\varphi_1}\equiv 1$ on $\{\Psi_1<-T\}$. Hence we have $\int_{U\cap \{\Psi_1<-t\}}|\frac{\tilde{F}}{F^2}|^2e^{-\varphi}c(-\Psi_1)<+\infty$ for some $t>T$, i.e. $(\frac{\tilde{F}}{F^2})_o\in H_o$. And if $(\tilde{F},o)\in \mathcal{I}(\varphi+\varphi_1+\Psi_1)_o$, it is easy to verify that $(\frac{\tilde{F}}{F^2})_o\in I(\varphi+\Psi_1)_o$. Hence we have an $\mathcal{O}_{\mathbb{C}^n,o}$-module homomorphism $Q:\mathcal{H}_o/\mathcal{I}(\varphi+\varphi_1+\Psi_1)_o\to H_o/I(\varphi+\Psi_1)_o$ defined as follows,
$$Q([(\tilde{F},o)])=[(\frac{\tilde{F}}{F^2})_o].$$

The above discussion shows that $Q$ is independent of the choices of the representatives of $[(\tilde{F},o)]$ and hence $Q$ is well defined.

\begin{Proposition}\label{module isomorphism}$P:H_o/I(\varphi+\Psi_1)_o\to \mathcal{H}_o/\mathcal{I}(\varphi+\varphi_1+\Psi_1)_o$ is an $\mathcal{O}_{\mathbb{C}^n,o}$-module isomorphism and $P^{-1}=Q$.
\end{Proposition}
\begin{proof} It follows from Proposition \ref{proposition of morphism} that we know $P$ is injective.

Now we prove $P$ is surjective.

For any $[(\tilde{F},o)]$ in $\mathcal{H}_o/\mathcal{I}(\varphi+\varphi_1+\Psi_1)_o$. Let $(\tilde{F},o)$ be any representatives of $[(\tilde{F},o)]$ in $\mathcal{H}_o$. Denote that $[(f_1)_o]:=[(\frac{\tilde{F}}{F^2})_o]=Q([(\tilde{F},o)])$. Let $(f_1)_o:=(\frac{\tilde{F}}{F^2})_o\in H_o$ be the representative of $[(f_1)_o]$. Denote $[(\tilde{F}_1,o)]:=\tilde{P}((f_1)_o)=P([(f_1)_o])$. By the construction of $\tilde{P}$, we know that $(\tilde{F}_1,o)\in \mathcal{H}_o$ and
$$\int_{D_1\cap\{\Psi_1<-t\}}|\tilde{F}_1-f_1F^2|e^{-\varphi-\varphi_1-\Psi_1}<+\infty,$$
where $t>T$ and $D_1$ is some neighborhood of $o$. Note that $(f_1)_o:=(\frac{\tilde{F}}{F^2})_o$. Hence  we have
$$\int_{D_1\cap\{\Psi_1<-t\}}|\tilde{F}_1-\tilde{F}|e^{-\varphi-\varphi_1-\Psi_1}<+\infty.$$
It follows from $(\tilde{F},o)\in \mathcal{H}_o$ and $(\tilde{F}_1,o)\in \mathcal{H}_o$ that there exists a neighborhood $D_2\subset D_1$ of $o$ such that
$$\int_{D_2}|\tilde{F}-\tilde{F}_1|^2e^{-\varphi-\varphi_1}c(-\Psi_1)<+\infty.$$
Note that on $\{\Psi_1\ge -t\}$, we have $c(-\Psi_1)e^{\Psi_1}\ge c(t)e^{-t}>0$. Hence we have
$$\int_{D_2\cap \{\Psi_1\ge-t\}}|\tilde{F}-\tilde{F}_1|^2e^{-\varphi-\varphi_1-\Psi_1}<+\infty.$$
Thus we know that $(\tilde{F}_1-\tilde{F},o) \in \mathcal{I}(\varphi+\varphi_1+\Psi_1)_o$, i.e. $[(\tilde{F},o)]=[(\tilde{F}_1,o)]$ in $ \mathcal{H}_o/\mathcal{I}(\varphi+\varphi_1+\Psi_1)_o$. Hence we have $P\circ Q([(\tilde{F},o)])=[(\tilde{F},o)]$, which implies that $P$ is surjective.

We have proved that $P:H_o/I(\varphi+\Psi_1)_o\to \mathcal{H}_o/\mathcal{I}(\varphi+\varphi_1+\Psi_1)_o$ is an $\mathcal{O}_{\mathbb{C}^n,o}$-module isomorphism and $P^{-1}=Q$.
\end{proof}

We recall the following property of closedness of holomorphic functions on a neighborhood of $o$.
\begin{Lemma}[see \cite{G-R}]
\label{closedness}
Let $N$ be a submodule of $\mathcal O_{\mathbb C^n,o}^q$, $1\leq q<+\infty$, let $f_j\in\mathcal O_{\mathbb C^n}(U)^q$ be a sequence of $q-$tuples holomorphic in an open neighborhood $U$ of the origin $o$. Assume that the $f_j$ converge uniformly in $U$ towards  a $q-$tuples $f\in\mathcal O_{\mathbb C^n}(U)^q$, assume furthermore that all germs $(f_{j},o)$ belong to $N$. Then $(f,o)\in N$.	
\end{Lemma}

\begin{Lemma}[see \cite{GY-concavity}]
	\label{l:converge}
	Let $M$ be a complex manifold. Let $S$ be an analytic subset of $M$.  	
	Let $\{g_j\}_{j=1,2,...}$ be a sequence of nonnegative Lebesgue measurable functions on $M$, which satisfies that $g_j$ are almost everywhere convergent to $g$ on  $M$ when $j\rightarrow+\infty$,  where $g$ is a nonnegative Lebesgue measurable function on $M$. Assume that for any compact subset $K$ of $M\backslash S$, there exist $s_K\in(0,+\infty)$ and $C_K\in(0,+\infty)$ such that
	$$\int_{K}{g_j}^{-s_K}dV_M\leq C_K$$
	 for any $j$, where $dV_M$ is a continuous volume form on $M$.
	
 Let $\{F_j\}_{j=1,2,...}$ be a sequence of holomorphic $(n,0)$ form on $M$. Assume that $\liminf_{j\rightarrow+\infty}\int_{M}|F_j|^2g_j\leq C$, where $C$ is a positive constant. Then there exists a subsequence $\{F_{j_l}\}_{l=1,2,...}$, which satisfies that $\{F_{j_l}\}$ is uniformly convergent to a holomorphic $(n,0)$ form $F$ on $M$ on any compact subset of $M$ when $l\rightarrow+\infty$, such that
 $$\int_{M}|F|^2g\leq C.$$
\end{Lemma}

The following lemma shows the closedness of submodules of $H_o$.

\begin{Lemma}\label{closedness of module}Let $D$ be a pseudoconvex domain containing $o$. Let $c(t)\in \tilde{P}_{T,D,\Psi}$.
Let $U_0\Subset D$ be a Stein neighborhood of $o$.
Let $J_o$ be an $\mathcal{O}_{\mathbb{C}^n,o}$-submodule of $H_o$ such that $I(\varphi+\Psi)_o\subset J_o$. Assume that $f_o\in J(\Psi)_o$. Let $\{f_j\}_{j\ge 1}$ be a sequence of holomorphic functions on $U_0\cap \{\Psi<-t_j\}$ for any $j\ge 1$, where $t_j>T$. Assume that $t_0:=\lim_{j\to +\infty}t_j\in[T,+\infty)$,
\begin{equation}\label{convergence property of module}
\limsup\limits_{j\to+\infty}\int_{U_0\cap\{\Psi<-t_j\}}|f_j|^2e^{-\varphi}c(-\Psi)\le C<+\infty,
\end{equation}
and $(f_j-f)_o\in J_o$. Then there exists a subsequence of $\{f_j\}_{j\ge 1}$ compactly convergent to a holomorphic function $f_0$ on $\{\Psi<-t_0\}\cap U_0$ which satisfies
$$\int_{U_0\cap\{\Psi<-t_0\}}|f_0|^2e^{-\varphi}c(-\Psi)\le C,$$
and $(f_0-f)_o\in J_o$.
\end{Lemma}
\begin{proof}

It follows from $c(t)\in\tilde{P}_{T,D,\Psi}$ that there exists an analytic subset $Z$ of $D$ and for any compact subset $K\subset D\backslash Z$, $e^{-\varphi}c(-\Psi)$ has a positive lower bound on $K\cap \{\Psi<-t_j\}$.

It follows from inequality \eqref{convergence property of module}, Lemma \ref{l:converge} and diagonal method that there exists a subsequence of $\{f_j\}_{j\ge 1}$ (also denoted by $\{f_j\}_{j\ge 1}$) compactly convergent to a holomorphic function $f_0$ on $\{\Psi<-t_0\}\cap U_0$. It follows from Fatou's Lemma that
$$\int_{U_0\cap\{\Psi<-t_0\}}|f_0|^2e^{-\varphi}c(-\Psi)\le \liminf\limits_{j\to+\infty}\int_{U_0\cap\{\Psi<-t_j\}}|f_j|^2e^{-\varphi}c(-\Psi)\le C.$$

Now we prove $(f_0-f)_o\in J_o$. We firstly recall some constructions in Lemma \ref{construction of morphism}.

As $t_0:=\lim_{j\to +\infty}t_j\in[T,+\infty)$. We can assume that $\{t_j\}_{j\ge 0}$ is upper bounded by some real number $T_1+1$. Denote $\Psi_1:=\min\{\psi-2\log|F|,-T_1\}$, and
if $F(z)=0$ for some $z \in M$, we set $\Psi_1(z)=-T_1$. We note that \begin{equation}\nonumber
\limsup\limits_{j\to+\infty}\int_{U_0\cap\{\Psi<-T_1-1\}}|f_j|^2e^{-\varphi}c(-\Psi)\le C<+\infty.
\end{equation}

It follows from $c(t)\in\tilde{P}_{T,D,\Psi}$ and Lemma \ref{L2 method for c(t)} that there exists a holomorphic function $\tilde{F}_j$ on $U_0$ such that
\begin{equation}\label{convergence property of module formula 1}
  \begin{split}
      & \int_{U_0}|\tilde{F}_j-(1-b_{1}(\Psi_1))f_jF^{2}|^2e^{-\varphi-\varphi_1+v_{1}(\Psi_1)-\Psi_1}c(-v_{1}(\Psi_1)) \\
      \le & \left(c(T_1)e^{-T_1}+\int_{T_1}^{T_1+2}c(s)e^{-s}ds\right)
       \int_{U_0}\mathbb{I}_{\{-T_1-2<\Psi_1<-T_1-1\}}|f_j|^2e^{-\varphi-\Psi_1},
  \end{split}
\end{equation}
where $b_{1}(t)=\int^{t}_{-\infty} \mathbb{I}_{\{-T_1-2< s < -T_1-1\}}ds$,
$v_{1}(t)=\int^{t}_{-T_1-1}b_{1}(s)ds-(T_1+1)$. Denote $C_1:=c(T_1)e^{-T_1}+\int_{T_1}^{T_1+1}c(s)e^{-s}ds$.

Note that $v_{1}(t)>-T_1-2$. We have $e^{v_{1}(\Psi_1)}c(-v_{1}(\Psi))\ge c(T_1+2)e^{-(T_1+2)}>0$. As $b_{1}(t)\equiv 0$ on $(-\infty,-T_1-2)$, we have
\begin{equation}\label{convergence property of module formula 2}
\begin{split}
   &\int_{U_0\cap\{\Psi<-T_1-2\}}|\tilde{F}_j-f_jF^2|^2e^{-\varphi-\varphi_1-\Psi_1} \\
   \le & \frac{1}{c(T_1+2)e^{-(T_1+2)}}
   \int_{U_0}|\tilde{F}_j-(1-b_{1}(\Psi_1))f_jF^2|^2e^{-\varphi-\varphi_1-\Psi_1+v_{1}(\Psi_1)}c(-v_{t_j}(\Psi_1))\\
   \le &\frac{C_1}{c(T_1+2)e^{-(T_1+2)}}
   \int_{U_0}\mathbb{I}_{\{-T_1-2<\Psi_1<-T_1-1\}}|f_j|^2e^{-\varphi-\Psi_1}<+\infty.
\end{split}
\end{equation}
Note that $|F^2|^2e^{-\varphi_1}=1$ on $\{\Psi_1<-T_1-1\}$. As $v_{t_j}(\Psi_1)\ge \Psi_1$, we have $c(-v_{t_j}(\Psi_1))e^{v_{t_j}(\Psi_1)}\ge c(-\Psi_1)e^{-\Psi_1}$. Hence we have
\begin{equation}\label{convergence property of module formula 3}
\begin{split}
   &\int_{U_0}|\tilde{F}_j|^2e^{-\varphi-\varphi_1}c(-\Psi_1) \\
   \le & 2\int_{U_0}|\tilde{F}_j-(1-b_{1}(\Psi_1))f_jF^2|^2e^{-\varphi-\varphi_1}c(-\Psi_1)\\
   +&2\int_{U_0}|(1-b_{1}(\Psi_1))f_jF^2|^2e^{-\varphi-\varphi_1}c(-\Psi_1)\\
   \le&
   2\int_{U_0}|\tilde{F}_j-(1-b_{1}(\Psi_1))f_jF^2|^2e^{-\varphi-\varphi_1-\Psi_1+v_{1}(\Psi_1)}c(-v_{1}(\Psi_1))\\
   +&2\int_{U_0\cap\{\Psi_1<-T_1-1\}}|f_j|^2e^{-\varphi}c(-\Psi_1)\\
   < &+\infty.
\end{split}
\end{equation}
Hence we know that $(\tilde{F}_j,o)\in \mathcal{H}_o$.

It follows from inequality \eqref{convergence property of module}, $\sup_{j\ge1}\left(\int_{U_0}\mathbb{I}_{\{-T_1-2<\Psi<-T_1-1\}}|f_j|^2e^{-\varphi-\Psi}\right)<+\infty$ and inequality \eqref{convergence property of module formula 3} that we actually have \begin{equation}\label{closedness of module formula 1}
\sup_j\left(\int_{U_0}|\tilde{F}_j|^2e^{-\varphi-\varphi_1}c(-\Psi_1)\right)<+\infty.
\end{equation}
Note that $c(t)e^{-t}$ is decreasing with respect to $t$ and there exists an analytic subset $S$ of $D$ and for any compact subset $K\subset D\backslash S$, $e^{-\varphi}c(-\Psi_1)$ has a positive lower bound on $K\cap \{\Psi<-T_1\}$. Let $K\subset U_0\backslash S\subset D\backslash S$ be any compact set, since $\varphi_1$ and $\varphi+\varphi_1+\Psi_1$ are plurisubharmonic, we have
\begin{equation}\label{closedness of module formula 2}
  \begin{split}
     \int_K \frac{1}{e^{-\varphi-\varphi_1}c(-\Psi_1)}& =\int_{K\cap\{\Psi_1<-T_1\}} \frac{e^{\varphi_1}}{e^{-\varphi}c(-\Psi_1)}+\int_{K\cap\{\Psi_1\ge-T_1\}} \frac{1}{e^{-\varphi-\varphi_1}c(-\Psi_1)} \\
       &\le M_K\int_{K\cap\{\Psi_1<-T_1\}} e^{\varphi_1}+\tilde{M}_K \int_{K\cap\{\Psi_1\ge-T_1\}} e^{\varphi+\varphi_1+\Psi_1}\\
       &\le  M_K\int_{K} e^{\varphi_1}+\tilde{M}_K \int_{K} e^{\varphi+\varphi_1+\Psi_1}\\
       &<+\infty,
  \end{split}
\end{equation}
where $M_K=\sup_{K\cap\{\Psi<-T_1\}}\frac{1}{e^{-\varphi}c(-\Psi_1)}$ and $\tilde{M}_K=\frac{1}{c(-T_1)e^{T_1}}$.
It follows from inequality \eqref{closedness of module formula 1}, inequality \eqref{closedness of module formula 2} and Lemma \ref{l:converge} that there exists a subsequence of $\{\tilde{F}_j\}_{j\ge 1}$ (also denoted by $\{\tilde{F}_j\}_{j\ge 1}$) compactly convergent to a holomorphic function $\tilde{F}_0$ on $U_0$ and
\begin{equation}\label{convergence property of module formula 4}
\int_{U_0}|\tilde{F}_0|^2e^{-\varphi-\varphi_1}c(-\Psi_1)\le
\liminf_{j\to +\infty}\int_{U_0}|\tilde{F}_j|^2e^{-\varphi-\varphi_1}c(-\Psi_1)<+\infty.
\end{equation}
As $f_j$ converges to $f_0$, it follows from Fatou's Lemma and inequality \eqref{convergence property of module formula 1} that
\begin{equation}\nonumber
  \begin{split}
  &\int_{U_0}|\tilde{F}_0-(1-b_{1}(\Psi))f_0F^{2}|^2e^{-\varphi-\varphi_1+v_{1}(\Psi_1)-\Psi_1}c(-v_{1}(\Psi_1)) \\
     \le & \liminf_{j\to+\infty} \int_{U_0}|\tilde{F}_j-(1-b_{1}(\Psi))f_jF^{2}|^2e^{-\varphi-\varphi_1+v_{1}(\Psi)-\Psi}c(-v_{1}(\Psi_1)) \\
     < &+\infty,
  \end{split}
\end{equation}
which implies that
\begin{equation}\label{convergence property of module formula 5}
  \begin{split}
\int_{U_0\cap\{\Psi<-T_1-2\}}|\tilde{F}_0-f_0F^2|^2e^{-\varphi-\varphi_1-\Psi_1}<+\infty.
  \end{split}
\end{equation}
It follows from inequality \eqref{convergence property of module formula 2}, inequality \eqref{convergence property of module formula 3}, inequality \eqref{convergence property of module formula 4}, inequality \eqref{convergence property of module formula 5} and definition of $P:H_o/I(\varphi+\Psi_1)_o\to \mathcal{H}_o/\mathcal{I}(\varphi+\varphi_1+\Psi_1)_o$ that for any $j\ge 0$, we have
$$P([(f_j)_o])=[(\tilde{F}_j,o)].$$

Note that $I(\Psi_1+\varphi)_o=I(\Psi+\varphi)_o\subset J_o$. As $(f_j-f)_o\in J_o$ for any $j\ge 1$, we have $(f_j-f_1)_o\in J_o$ for any $j\ge 1$.
It follows from Proposition \ref{module isomorphism} that there exists an ideal $\tilde{J}$ of $\mathcal{O}_{\mathbb{C}^n,o}$ such that $\mathcal{I}(\varphi+\varphi_1+\Psi_1)_o\subset \tilde{J}\subset \mathcal{H}_o$ and $\tilde{J}/\mathcal{I}(\varphi+\varphi_1+\Psi_1)_o=\text{Im}(P|_{J_o/I(\varphi+\Psi_1)_o})$. It follows from $(f_j-f_1)_o\in J_o$ and $P([(f_j)_o])=[(F_j,o)]$ for any $j\ge 1$ that we have
$$(\tilde{F}_j-\tilde{F}_1)\in \tilde{J},$$
for any $j\ge 1$.

As $\tilde{F}_j$ compactly converges to $\tilde{F}_0$, using Lemma \ref{closedness}, we obtain that $(\tilde{F}_0-\tilde{F}_1,o)\in\tilde{J}$. Note that $P$ is an  $\mathcal{O}_{\mathbb{C}^n,o}$-module isomorphism and $\tilde{J}/\mathcal{I}(\varphi+\varphi_1+\Psi_1)_o=\text{Im}(P|_{J_o/I(\varphi+\Psi_1)_o})$. We have $(f_0-f_1)_o\in J_o$, which implies that $(f_0-f)_o\in J_o$.

Lemma \ref{closedness of module} is proved.
\end{proof}

\subsection{Properties of $G(t)$}
Following the notations in Section \ref{sec:Main result}, we present some properties of the function $G(t)$ in this section.

For any $t\ge T$, denote
\begin{equation}\nonumber
\begin{split}
\mathcal{H}^2(t;c,f):=\Bigg\{\tilde{f}:\int_{ \{ \Psi<-t\}}|\tilde{f}|^2e^{-\varphi}c(-\Psi)<+\infty,\  \tilde{f}\in
H^0(\{\Psi<-t\},\mathcal{O} (K_M)  ) \\
\& (\tilde{f}-f)_{z_0}\in
\mathcal{O} (K_M)_{z_0} \otimes J_{z_0},\text{for any }  z_0\in Z_0  \Bigg\},
\end{split}
\end{equation}
where $f$ is a holomorphic $(n,0)$ form on $\{\Psi<-t_0\}\cap V$ for some $V\supset Z_0$ is an open subset of $M$ and some $t_0\ge T$ and $c(t)$ is a positive measurable function on $(T,+\infty)$.

For any $t\ge T$, denote
\begin{equation}\nonumber
\begin{split}
\mathcal{H}^2(t;c,f,H):=\Bigg\{\tilde{f}:\int_{ \{ \Psi<-t\}}|\tilde{f}|^2e^{-\varphi}c(-\Psi)<+\infty,\  \tilde{f}\in
H^0(\{\Psi<-t\},\mathcal{O} (K_M)  ) \\
\& (\tilde{f}-f)_{z_0}\in
\mathcal{O} (K_M)_{z_0} \otimes (J_{z_0}\cap H_{z_0}),\ \text{for any }  z_0\in Z_0  \Bigg\},
\end{split}
\end{equation}
where $f$ is a holomorphic $(n,0)$ form on $\{\Psi<-t_0\}\cap V$ for some $V\supset Z_0$ is an open subset of $M$ and some $t_0\ge T$, $c(t)$ is a positive measurable function on $(T,+\infty)$ and $H_{z_0}=\{f_o\in J(\Psi)_o:\int_{\{\Psi<-t\}\cap V_0}|f|^2e^{-\varphi}c(-\Psi)<+\infty \text{ for some }t>T_0 \text{ and } V_0 \text{ is an open neighborhood of}\  z_0\}$ (the definition of $H_{z_0}$ can be referred to Section \ref{sec:properties of module}).

If $G(t_1;c,\Psi,\varphi,J,f)<+\infty$, then there exists a holomorphic $(n,0)$ form $\tilde{f}_0$ on $\{\Psi<-t_1\}$ such that $(\tilde{f}_0-f)_{z_0}\in
\mathcal{O} (K_M)_{z_0} \otimes J_{z_0},\text{for any }  z_0\in Z_0$ and $$\int_{ \{ \Psi<-t_1\}}|\tilde{f}_0|^2e^{-\varphi}c(-\Psi)<+\infty.$$

\begin{Lemma}
\label{module in def of G(t)}If $G(t_1;c,\Psi,\varphi,J,f)<+\infty$ for some $t_1\ge T$, we have $\mathcal{H}^2(t;c,f)=\mathcal{H}^2(t;c,\tilde{f}_0)=\mathcal{H}^2(t;c,\tilde{f}_0,H)$ for any $t\ge T$.
\end{Lemma}
\begin{proof}

 As $(\tilde{f}_0-f)_{z_0}\in
\mathcal{O} (K_M)_{z_0} \otimes J_{z_0},\text{for any }  z_0\in Z_0$, we have $\mathcal{H}^2(t;c,f)=\mathcal{H}^2(t;c,\tilde{f}_0)$ for any $t\ge T$.

 Now we prove $\mathcal{H}^2(t;c,\tilde{f}_0)=\mathcal{H}^2(t;c,\tilde{f}_0,H)$ for any $t\ge T$. It is obviously that $\mathcal{H}^2(t;c,\tilde{f}_0)\supset\mathcal{H}^2(t;c,\tilde{f}_0,H)$. We only need to show $\mathcal{H}^2(t;c,\tilde{f}_0)\subset\mathcal{H}^2(t;c,\tilde{f}_0,H)$.

Let $\tilde{f}_1\in \mathcal{H}^2(t_2;c,\tilde{f}_0)$ for some $t_2\ge T$. As $\int_{ \{ \Psi<-t_2\}}|\tilde{f}_1|^2e^{-\varphi}c(-\Psi)<+\infty$, denote $t=\max\{t_1,t_2\}$, we know that
$$\int_{ \{ \Psi<-t\}}|\tilde{f}_1-\tilde{f}_0|^2e^{-\varphi}c(-\Psi)<+\infty,$$
which implies that $(\tilde{f}_1-\tilde{f}_0)_{z_0}\in
\mathcal{O} (K_M)_{z_0} \otimes  H_{z_0},\ \text{for any }  z_0\in Z_0$. Hence $(\tilde{f}_1-\tilde{f}_0)_{z_0}\in
\mathcal{O} (K_M)_{z_0} \otimes  (J_{z_0}\cap H_{z_0}),\ \text{for any }  z_0\in Z_0$, which implies that $\tilde{f}_1\in \mathcal{H}^2(t;c,\tilde{f}_0, H)$. Hence $\mathcal{H}^2(t;c,\tilde{f}_0)=\mathcal{H}^2(t;c,\tilde{f}_0,H)$.
\end{proof}
\begin{Remark}
\label{module equivalence in def of G}
If $G(t_1;c,\Psi,\varphi,J,f)<+\infty$ for some $t_1\ge T$, we can always assume that $J_{z_0}$ is an $\mathcal{O}_{M,z_0}$-submodule of $H_{z_0}$ such that $I\big(\Psi+\varphi\big)_{z_0}\subset J_{z_0}$, for any $z_0\in Z_0$ in the definition of $G(t;c,\Psi,\varphi,J,f)$, where $t\in[T,+\infty)$.
\end{Remark}
\begin{proof} If $G(t_1;c,\Psi,\varphi,J,f)<+\infty$ for some $t_1\ge T$, it follows from Lemma \ref{module in def of G(t)} that $\mathcal{H}^2(t;c,f)=\mathcal{H}^2(t;c,\tilde{f}_0)=\mathcal{H}^2(t;c,\tilde{f}_0,H)$ for any $t\ge T$. By definition, we have $G(t;c,\Psi,\varphi,J,f)=G(t;c,\Psi,\varphi,J,\tilde{f}_0)=G(t;c,\Psi,\varphi,J\cap H,\tilde{f}_0)$.

Hence we can always assume that $J_{z_0}$ is an $\mathcal{O}_{M,z_0}$-submodule of $H_{z_0}$ such that $I\big(\Psi+\varphi\big)_{z_0}\subset J_{z_0}$, for any $z_0\in Z_0$.
\end{proof}

In the following discussion, we assume that $J_{z_0}$ is an $\mathcal{O}_{M,z_0}$-submodule of $H_{z_0}$ such that $I\big(\Psi+\varphi\big)_{z_0}\subset J_{z_0}$, for any $z_0\in Z_0$.

Let $c(t)\in \tilde{P}_{T,M,\Psi}$.
The following lemma will be used to discuss the convergence property of holomorphic forms on $\{\Psi<-t\}$.
\begin{Lemma}\label{global convergence property of module}
 Let $f$ be a holomorphic $(n,0)$ form on $\{\Psi<-\hat{t}_0\}\cap V$, where $V\supset Z_0$ is an open subset of $M$ and $\hat{t}_0>T$
is a real number. For any $z_0\in Z_0$, let $J_{z_0}$ be an $\mathcal{O}_{M,z_0}$-submodule of $H_{z_0}$ such that $I\big(\Psi+\varphi\big)_{z_0}\subset J_{z_0}$.

Let $\{f_j\}_{j\ge 1}$ be a sequence of holomorphic $(n,0)$ forms on $\{\Psi<-t_j\}$. Assume that $t_0:=\lim_{j\to +\infty}t_j\in[T,+\infty)$,
\begin{equation}\label{global convergence property of module 1}
\limsup\limits_{j\to+\infty}\int_{\{\Psi<-t_j\}}|f_j|^2e^{-\varphi}c(-\Psi)\le C<+\infty,
\end{equation}
and $(f_j-f)_{z_0}\in \mathcal{O} (K_M)_{z_0}\otimes J_{z_0}$ for any $z_0\in Z_0$. Then there exists a subsequence of $\{f_j\}_{j\in \mathbb{N}^+}$ compactly convergent to a holomorphic $(n,0)$ form $f_0$ on $\{\Psi<-t_0\}$ which satisfies
$$\int_{\{\Psi<-t_0\}}|f_0|^2e^{-\varphi}c(-\Psi)\le C,$$
and $(f_0-f)_{z_0}\in \mathcal{O} (K_M)_{z_0}\otimes  J_{z_0}$ for any $z_0\in Z_0$.
\end{Lemma}
\begin{proof}

It follows from $c(t)\in\tilde{P}_{T,M,\Psi}$ that there exists an analytic subset $Z$ of $D$ and for any compact subset $K\subset D\backslash Z$, $e^{-\varphi}c(-\Psi)$ has a positive lower bound on $K\cap \{\Psi<-t_j\}$.

It follows from inequality \eqref{global convergence property of module 1}, Lemma \ref{l:converge} and diagonal method that there exists a subsequence of $\{f_j\}_{j\ge 1}$ (also denoted by $\{f_j\}_{j\ge 1}$) compactly convergent to a holomorphic function $f_0$ on $\{\Psi<-t_0\}\cap U_0$. It follows from Fatou's Lemma that
$$\int_{U_0\cap\{\Psi<-t_0\}}|f_0|^2e^{-\varphi}c(-\Psi)\le \liminf\limits_{j\to+\infty}\int_{U_0\cap\{\Psi<-t_j\}}|f_j|^2e^{-\varphi}c(-\Psi)\le C.$$

Next we prove $(f_0-f)_{z_0}\in \mathcal{O} (K_M)_{z_0}\otimes  J_{z_0}$ for any $z_0\in Z_0$.

Let $z_0\in Z_0$ be a point. As $\limsup\limits_{j\to+\infty}\int_{\{\Psi<-t_j\}}|f_j|^2e^{-\varphi}c(-\Psi)\le C<+\infty$,
 there exists an open Stein neighborhood $U_{z_0}\Subset M$ of $z_0$ such that
$$\limsup\limits_{j\to+\infty}\int_{U_{z_0}\cap\{\Psi<-t_j\}}|f_j|^2e^{-\varphi}c(-\Psi)\le C<+\infty.$$
Note that we also have $(f_j-f)_{z_0}\in J_{z_0}$.
 It follows from Lemma \ref{closedness of module} and the uniqueness of limit function that  $(f_0-f)_{z_0}\in
\mathcal{O} (K_M)_{z_0} \otimes J_{z_0}$ for any $z_0\in Z_0$.

Lemma \ref{global convergence property of module} is proved.
\end{proof}

\begin{Lemma}
\label{characterization of g(t)=0} Let $t_0>T$.
The following two statements are equivalent,\\
(1) $G(t_0)=0$;\\
(2) $f_{z_0}\in
\mathcal{O} (K_M)_{z_0} \otimes J_{z_0}$, for any  $ z_0\in Z_0$.
\end{Lemma}
\begin{proof}If $f_{z_0}\in
\mathcal{O} (K_M)_{z_0} \otimes J_{z_0}$, for any  $ z_0\in Z_0$, then take $\tilde{f}\equiv 0$ in the definition of $G(t)$ and we get $G(t_0)\equiv 0$.

If $G(t_0)=0$, by definition, there exists a sequence of holomorphic $(n,0)$ forms $\{f_j\}_{j\in\mathbb{Z}^+}$ on $\{\Psi<-t_0\}$ such that
\begin{equation}\label{estimate in G(t)=0}
\lim_{j\to+\infty}\int_{\{\Psi<-t_0\}}|f_j|^2e^{-\varphi}c(-\Psi)=0,
\end{equation}
 and $(f_j-f)_{z_0}\in
\mathcal{O} (K_M)_{z_0} \otimes J_{z_0}$, for any  $ z_0\in Z_0$ and $j\ge 1$. It follows from Lemma \ref{global convergence property of module} that there
exists a subsequence of $\{f_j\}_{j\in \mathbb{N}^+}$ compactly convergent to a holomorphic $(n,0)$ form $f_0$ on $\{\Psi<-t_0\}$ which satisfies
$$\int_{\{\Psi<-t_0\}}|f_0|^2e^{-\varphi}c(-\Psi)=0$$
and
$(f_0-f)_{z_0}\in \mathcal{O} (K_M)_{z_0}\otimes J_{z_0}$ for any $z_0\in Z_0$. It follows from $\int_{\{\Psi<-t_0\}}|f_0|^2e^{-\varphi}c(-\Psi)=0$ that we know $f_0\equiv 0$. Hence we have $f_{z_0}\in \mathcal{O} (K_M)_{z_0}\otimes J_{z_0}$ for any $z_0\in Z_0$. Statement (2) is proved.
\end{proof}

The following lemma shows the existence and uniqueness of the holomorphic $(n,0)$ form related to $G(t)$.
\begin{Lemma}
\label{existence of F}
Assume that $G(t)<+\infty$ for some $t\in [T,+\infty)$. Then there exists a unique
holomorphic $(n,0)$ form $F_t$ on $\{\Psi<-t\}$ satisfying
$$\ \int_{\{\Psi<-t\}}|F_t|^2e^{-\varphi}c(-\Psi)=G(t)$$  and
$\ (F_t-f)\in
\mathcal{O} (K_M)_{z_0} \otimes J_{z_0}$, for any  $ z_0\in Z_0$.
\par
Furthermore, for any holomorphic $(n,0)$ form $\hat{F}$ on $\{\Psi<-t\}$ satisfying
$$\int_{\{\Psi<-t\}}|\hat{F}|^2e^{-\varphi}c(-\Psi)<+\infty$$ and $\ (\hat{F}-f)\in
\mathcal{O} (K_M)_{z_0} \otimes J_{z_0}$, for any  $ z_0\in Z_0$. We have the following equality
\begin{equation}
\begin{split}
&\int_{\{\Psi<-t\}}|F_t|^2e^{-\varphi}c(-\Psi)+
\int_{\{\Psi<-t\}}|\hat{F}-F_t|^2e^{-\varphi}c(-\Psi)\\
=&\int_{\{\Psi<-t\}}|\hat{F}|^2e^{-\varphi}c(-\Psi).
\label{orhnormal F}
\end{split}
\end{equation}
\end{Lemma}

\begin{proof} We firstly show the existence of $F_t$. As $G(t)<+\infty$, then there exists a sequence of holomorphic $(n,0)$ forms $\{f_j\}_{j\in \mathbb{N}^+}$ on $\{\Psi<-t\}$ such that $$\lim\limits_{j \to +\infty}\int_{\{\Psi<-t\}}|f_j|^2e^{-\varphi}c(-\Psi)=G(t)$$ and $(f_j-f)\in
\mathcal{O} (K_M)_{z_0} \otimes J_{z_0}$, for any  $ z_0\in Z_0$ and any $j\ge 1$.
It follows from Lemma \ref{global convergence property of module} that there
exists a subsequence of $\{f_j\}_{j\in \mathbb{N}^+}$ compactly convergent to a holomorphic $(n,0)$ form $F$ on $\{\Psi<-t\}$ which satisfies
$$\int_{\{\Psi<-t\}}|F|^2e^{-\varphi}c(-\Psi)\le G(t)$$
and
$(F-f)_{z_0}\in \mathcal{O} (K_M)_{z_0}\otimes J_{z_0}$ for any $z_0\in Z_0$. By the definition of $G(t)$, we have $\int_{\{\Psi<-t\}}|F|^2e^{-\varphi}c(-\Psi)= G(t)$. Then we obtain the existence of $F_t(=F)$.

We prove the uniqueness of $F_t$ by contradiction: if not, there exist
two different holomorphic $(n,0)$ forms $f_1$ and $f_2$ on $\{\Psi<-t\}$
satisfying $\int_{\{\Psi<-t\}}|f_1|^2e^{-\varphi}$  $c(-\Psi)=
\int_{\{\Psi<-t\}}|f_2|^2e^{-\varphi}c(-\Psi)=G(t)$, $(f_1-f)_{z_0}\in \mathcal{O} (K_M)_{z_0}\otimes J_{z_0}$ for any $z_0\in Z_0$ and $(f_2-f)_{z_0}\in \mathcal{O} (K_M)_{z_0}\otimes J_{z_0}$ for any $z_0\in Z_0$. Note that
\begin{equation}\nonumber
\begin{split}
\int_{\{\Psi<-t\}}|\frac{f_1+f_2}{2}|^2e^{-\varphi}c(-\Psi)+
\int_{\{\Psi<-t\}}|\frac{f_1-f_2}{2}|^2e^{-\varphi}c(-\Psi)\\
=\frac{1}{2}(\int_{\{\Psi<-t\}}|f_1|^2e^{-\varphi}c(-\Psi)+
\int_{\{\Psi<-t\}}|f_1|^2e^{-\varphi}c(-\Psi))=G(t),
\end{split}
\end{equation}
then we obtain that
\begin{equation}\nonumber
\begin{split}
\int_{\{\Psi<-t\}}|\frac{f_1+f_2}{2}|^2e^{-\varphi}c(-\Psi)
< G(t)
\end{split}
\end{equation}
and $(\frac{f_1+f_2}{2}-f)_{z_0}\in \mathcal{O} (K_M)_{z_0}\otimes J_{z_0}$ for any $z_0\in Z_0$, which contradicts to the definition of $G(t)$.

Now we prove equality \eqref{orhnormal F}. Let $h$ be any holomorphic $(n,0)$ form on $\{\Psi<-t\}$
such that $\int_{\{\Psi<-t\}}|h|^2e^{-\varphi}c(-\Psi)<+\infty$ and $h \in \mathcal{O} (K_M)_{z_0}\otimes J_{z_0}$ for any $z_0\in Z_0$.  It is clear that for any complex
number $\alpha$, $F_t+\alpha h$ satisfying $((F_t+\alpha h)-f) \in \mathcal{O} (K_M)_{z_0}\otimes J_{z_0}$ for any $z_0\in Z_0$ and
$\int_{\{\Psi<-t\}}|F_t|^2e^{-\varphi}c(-\Psi) \leq \int_{\{\Psi<-t\}}|F_t+\alpha
h|^2e^{-\varphi}c(-\Psi)$. Note that
\begin{equation}\nonumber
\begin{split}
\int_{\{\Psi<-t\}}|F_t+\alpha
h|^2e^{-\varphi}c(-\Psi)-\int_{\{\psi<-t\}}|F_t|^2e^{-\varphi}c(-\Psi)\geq 0
\end{split}
\end{equation}
(By considering $\alpha \to 0$) implies
\begin{equation}\nonumber
\begin{split}
\mathfrak{R} \int_{\{\Psi<-t\}}F_t\bar{h}e^{-\varphi}c(-\Psi)=0,
\end{split}
\end{equation}
then we have
\begin{equation}\nonumber
\begin{split}
\int_{\{\Psi<-t\}}|F_t+h|^2e^{-\varphi}c(-\Psi)=
\int_{\{\Psi<-t\}}(|F_t|^2+|h|^2)e^{-\varphi}c(-\Psi).
\end{split}
\end{equation}
\par
Letting $h=\hat{F}-F_t$, we obtain equality \eqref{orhnormal F}.
\end{proof}

The following lemma shows the  lower semicontinuity property of $G(t)$.
\begin{Lemma}$G(t)$ is decreasing with respect to $t\in
[T,+\infty)$, such that $\lim \limits_{t \to t_0+0}G(t)=G(t_0)$ for any $t_0\in
[T,+\infty)$, and if $G(t)<+\infty$ for some $t>T$, then $\lim \limits_{t \to +\infty}G(t)=0$. Especially, $G(t)$ is lower semicontinuous on $[T,+\infty)$.
 \label{semicontinuous}
\end{Lemma}
\begin{proof}By the definition of $G(t)$, it is clear that $G(t)$ is decreasing on
$[T,+\infty)$. If $G(t)<+\infty$ for some $t>T$, by the dominated convergence theorem, we know $\lim\limits_{t\to +\infty}G(t)=0$. It suffices
to prove $\lim \limits_{t \to t_0+0}G(t)=G(t_0)$ . We prove it by
contradiction: if not, then $\lim \limits_{t \to t_0+0}G(t)<
G(t_0)$.

By using Lemma \ref{existence of F}, for any $t>t_0$, there exists a unique holomorphic $(n,0)$ form
$F_t$ on $\{\Psi<-t\}$ satisfying
$\int_{\{\Psi<-t\}}|F_t|^2e^{-\varphi}c(-\Psi)=G(t)$ and $(F_t-f) \in \mathcal{O} (K_M)_{z_0}\otimes J_{z_0}$ for any $z_0\in Z_0$. Note that $G(t)$ is decreasing with respect to $t$. We have $\int_{\{\Psi<-t\}}|F_t|^2e^{-\varphi}c(-\psi)\leq \lim
\limits_{t \to t_0+0}G(t)$ for any $t>t_0$. If $\lim\limits_{t \to t_0+0}G(t)=+\infty$, the equality $\lim \limits_{t \to t_0+0}G(t)=G(t_0)$ obviously holds, thus it suffices to prove the case $\lim\limits_{t \to t_0+0}G(t)<+\infty$. It follows from $\int_{\{\Psi<-t\}}|F_t|^2e^{-\varphi}c(-\Psi)\le \lim\limits_{t \to t_0+0}G(t)<+\infty$  holds for any $t\in (t_0,t_1]$ (where $t_1>t_0$ is a fixed number) and Lemma \ref{global convergence property of module} that there exists a subsequence of $\{F_t\}$ (denoted by $\{F_{t_j}\}$) compactly convergent to a holomorphic $(n,0)$ form $\hat{F}_{t_0}$ on $\{\Psi<-t_0\}$ satisfying
$$\int_{\{\Psi<-t_0\}}|\hat{F}_{t_0}|^2e^{-\varphi}c(-\Psi)\le \lim\limits_{t \to t_0+0}G(t)<+\infty$$
and $(\hat{F}_{t_0}-f)_{z_0} \in \mathcal{O} (K_M)_{z_0}\otimes J_{z_0}$ for any $z_0\in Z_0$.

Then we obtain that $G(t_0)\leq
\int_{\{\Psi<-t_0\}}|\hat{F}_{t_0}|^2e^{-\varphi}c(-\Psi)
\leq \lim \limits_{t\to t_0+0} G(t)$,
which contradicts $\lim \limits_{t\to t_0+0} G(t) <G(t_0)$. Thus we have $\lim \limits_{t \to t_0+0}G(t)=G(t_0)$.
\end{proof}

We consider the derivatives of $G(t)$ in the following lemma.

\begin{Lemma}
\label{derivatives of G}
Assume that $G(t_1)<+\infty$, where $t_1\in (T,+\infty)$. Then for any $t_0>t_1$, we have
\begin{equation}\nonumber
\begin{split}
\frac{G(t_1)-G(t_0)}{\int^{t_0}_{t_1} c(t)e^{-t}dt}\leq
\liminf\limits_{B \to
0+0}\frac{G(t_0)-G(t_0+B)}{\int_{t_0}^{t_0+B}c(t)e^{-t}dt},
\end{split}
\end{equation}
i.e.
\begin{equation}\nonumber
\frac{G(t_0)-G(t_1)}{\int_{T_1}^{t_0}
c(t)e^{-t}dt-\int_{T_1}^{t_1} c(t)e^{-t}dt} \geq
\limsup \limits_{B \to 0+0}
\frac{G(t_0+B)-G(t_0)}{\int_{T_1}^{t_0+B}
c(t)e^{-t}dt-\int_{T_1}^{t_0} c(t)e^{-t}dt}.
\end{equation}
\end{Lemma}

\begin{proof}
It follows from Lemma \ref{semicontinuous} that $G(t)<+\infty$ for any $t>t_1$. By Lemma \ref{existence of F}, there exists a holomorphic $(n,0)$ form $F_{t_0}$ on $\{\Psi<-t_0\}$, such that $(F_{t_0}-f)\in \mathcal{O} (K_M)_{z_0}\otimes J_{z_0}$ for any $z_0\in Z_0$ and $G(t_0)=\int_{\{\Psi<-t_0\}}|F_{t_0}|^2e^{-\varphi}c(-\Psi)$.

It suffices to consider that $\liminf\limits_{B\to 0+0} \frac{G(t_0)-G(t_0+B)}{\int_{t_0}^{t_0+B}c(t)e^{-t}dt}\in [0,+\infty)$ because of the decreasing property of $G(t)$. Then there exists $1\ge B_j\to 0+0$ (as $j\to+\infty$) such that
\begin{equation}
	\label{derivatives of G c(t)1}
\lim\limits_{j\to +\infty} \frac{G(t_0)-G(t_0+B_j)}{\int_{t_0}^{t_0+B_j}c(t)e^{-t}dt}=\liminf\limits_{B\to 0+0} \frac{G(t_0)-G(t_0+B)}{\int_{t_0}^{t_0+B}c(t)e^{-t}dt}
\end{equation}
and $\{\frac{G(t_0)-G(t_0+B_j)}{\int_{t_0}^{t_0+B_j}c(t)e^{-t}dt}\}_{j\in\mathbb{N}^{+}}$ is bounded. As $c(t)e^{-t}$ is decreasing and positive on $(t,+\infty)$, then
\begin{equation}\label{derivatives of G c(t)2}
\begin{split}
\lim\limits_{j\to +\infty} \frac{G(t_0)-G(t_0+B_j)}{\int_{t_0}^{t_0+B_j}c(t)e^{-t}dt}
=&\big(\lim\limits_{j\to +\infty} \frac{G(t_0)-G(t_0+B_j)}{B_j})(\frac{1}{\lim\limits_{t\to t_0+0}c(t)e^{-t}}\big)\\
=&\big(\lim\limits_{j\to +\infty} \frac{G(t_0)-G(t_0+B_j)}{B_j})(\frac{e^{t_0}}{\lim\limits_{t\to t_0+0}c(t)}\big).
\end{split}
\end{equation}
Hence $\{\frac{G(t_0)-G(t_0+B_j)}{B_j}\}_{j\in\mathbb{N}^+}$ is uniformly bounded with respect to $j$.

As $t \leq v_{t_0,j}(t)$, the decreasing property of $c(t)e^{-t}$ shows that
\begin{equation}\nonumber
e^{-\Psi+v_{t_0,B_j}(\Psi)}c(-v_{t_0,B_j}(\Psi))\geq c(-\Psi).
\end{equation}
\par
It follows from Lemma \ref{L2 method in JM concavity} that, for any $B_j$, there exists holomorphic
$(n,0)$ form $\tilde{F}_j$ on $\{\Psi<-t_1\}$ such that
\begin{flalign}
&\int_{\{\Psi<-t_1\}}|\tilde{F}_j-(1-b_{t_0,B_j}(\Psi))F_{t_0}|^2e^{-\varphi}c(-\Psi)\nonumber\\
\leq &
\int_{\{\Psi<-t_1\}}|\tilde{F}_j-(1-b_{t_0,B_j}(\Psi))F_{t_0}|^2e^{-\varphi}e^{-\Psi+v_{t_0,B_j}(\Psi)}c(-v_{t_0,B_j}(\Psi))\nonumber\\
\leq &
\int^{t_0+B_j}_{t_1}c(t)e^{-t}dt\int_{\{\Psi<-t_1\}}\frac{1}{B_j}
\mathbb{I}_{\{-t_0-B_j<\Psi<-t_0\}}|F_{t_0}|^2e^{-\varphi-\Psi}\nonumber\\
\leq &
\frac{e^{t_0+B_j}\int^{t_0+B_j}_{t_1}c(t)e^{-t}dt}{\inf
\limits_{t\in(t_0,t_0+B_j)}c(t)}\int_{\{\Psi<-t_1\}}\frac{1}{B_j}
\mathbb{I}_{\{-t_0-B_j<\Psi<-t_0\}}|F_{t_0}|^2e^{-\varphi}c(-\Psi)\nonumber\\
= &
\frac{e^{t_0+B_j}\int^{t_0+B_j}_{t_1}c(t)e^{-t}dt}{\inf
\limits_{t\in(t_0,t_0+B_j)}c(t)}\times
\bigg(\int_{\{\Psi<-t_1\}}\frac{1}{B_j}\mathbb{I}_{\{\Psi<-t_0\}}|F_{t_0}|^2e^{-\varphi}c(-\Psi)\nonumber\\
&-\int_{\{\Psi<-t_1\}}\frac{1}{B_j}\mathbb{I}_{\{\Psi<-t_0-B_j\}}|F_{t_0}|^2e^{-\varphi}c(-\Psi)\bigg)\nonumber\\
\leq &
\frac{e^{t_0+B_j}\int^{t_0+B_j}_{t_1}c(t)e^{-t}dt}{\inf
\limits_{t\in(t_0,t_0+B_j)}c(t)} \times
\frac{G(t_0)-G(t_0+B_j)}{B_j}<+\infty.
\label{derivative of G 1}
\end{flalign}

Note that $b_{t_0,B_j}(t)=0$ for $t\le-t_0-B_j$, $b_{t_0,B_j}(t)=1$ for $t\ge t_0$, $v_{t_0,B_j}(t)>-t_0-B_j$ and $c(t)e^{-t}$ is decreasing with respect to $t$. It follows from inequality \eqref{derivative of G 1} that $(F_j-F_{t_0})_{z_0}\in \mathcal{O} (K_M)_{z_0}\otimes I(\Psi+\varphi)_{z_0} \subset \mathcal{O} (K_M)_{z_0}\otimes J_{z_0}$ for any $z_0\in Z_0$.

Note that
\begin{equation}\label{derivative of G 2}
\begin{split}
&\int_{\{\Psi<-t_1\}}|\tilde{F}_j|^2e^{-\varphi}c(-\Psi)\\
\le&2\int_{\{\Psi<-t_1\}}|\tilde{F}_j-(1-b_{t_0,B_j}(\Psi))F_{t_0}|^2e^{-\varphi}c(-\Psi)
+2\int_{\{\Psi<-t_1\}}|(1-b_{t_0,B_j}(\Psi))F_{t_0}|^2e^{-\varphi}c(-\Psi)\\
\le&2
\frac{e^{t_0+B_j}\int^{t_0+B_j}_{t_1}c(t)e^{-t}dt}{\inf
\limits_{t\in(t_0,t_0+B_j)}c(t)} \times
\frac{G(t_0)-G(t_0+B_j)}{B_j}
+2\int_{\{\Psi<-t_0\}}|F_{t_0}|^2e^{-\varphi}c(-\Psi).
\end{split}
\end{equation}

We also note that $B_j\le 1$, $\frac{G(t_0)-G(t_0+B_j)}{B_j}$ is uniformly bounded with respect to $j$ and $G(t_0)=\int_{\{\Psi<-t_0\}}|F_{t_0}|^2e^{-\varphi}c(-\Psi)$. It follows from inequality \eqref{derivative of G 2} that we know $\int_{\{\Psi<-t_1\}}|\tilde{F}_j|^2e^{-\varphi}c(-\Psi)$ is uniformly bounded with respect to $j$.

It follows from Lemma \ref{global convergence property of module} that there exists a subsequence of $\{\tilde{F}_j\}_{j\in \mathbb{N}^+}$ compactly convergent to a holomorphic $(n,0)$ form $\tilde{F}_{t_1}$ on $\{\Psi<-t_1\}$ which satisfies
$$\int_{\{\Psi<-t_1\}}|\tilde{F}_{t_1}|^2e^{-\varphi}c(-\Psi)\le \liminf_{j\to+\infty} \int_{\{\Psi<-t_1\}}|\tilde{F}_j|^2e^{-\varphi}c(-\Psi)<+\infty,$$
and $(\tilde{F}_{t_1}-F_{t_0})_{z_0}\in \mathcal{O} (K_M)_{z_0}\otimes  J_{z_0}$ for any $z_0\in Z_0$.

Note that
$\lim_{j\to+\infty}b_{t_0,B_j}(t)=\mathbb{I}_{\{t\ge -t_0\}}$ and
\begin{equation}\nonumber
v_{t_0}(t):=\lim_{j\to+\infty}v_{t_0,B_j}(t)=\left\{
\begin{aligned}
&-t_0  &\text{ if } & x<-t_0, \\
&\ t  &\text{ if }  & x\ge t_0 .
\end{aligned}
\right.
\end{equation}

It follows from inequality \eqref{derivative of G 1} and Fatou's lemma that

\begin{flalign}
\label{derivative of G 3}
&\int_{\{\Psi<-t_0\}}|\tilde{F}_{t_1}-F_{t_0}|^2e^{-\varphi}c(-\Psi)
+\int_{\{-t_0\le\Psi<-t_1\}}|\tilde{F}_{t_1}|^2e^{-\varphi}c(-\Psi)\nonumber\\
\leq &
\int_{\{\Psi<-t_1\}}|\tilde{F}_{t_1}-\mathbb{I}_{\{\Psi< -t_0\}}F_{t_0}|^2e^{-\varphi}e^{-\Psi+v_{t_0}(\Psi)}c(-v_{t_0}(\Psi))\nonumber\\
\le&\liminf_{j\to+\infty}\int_{\{\Psi<-t_1\}}|\tilde{F}_j-(1-b_{t_0,B_j}(\Psi))F_{t_0}|^2e^{-\varphi}c(-\Psi)\nonumber\\
\leq &\liminf_{j\to+\infty}
\bigg(\frac{e^{t_0+B_j}\int^{t_0+B_j}_{t_1}c(t)e^{-t}dt}{\inf
\limits_{t\in(t_0,t_0+B_j)}c(t)} \times
\frac{G(t_0)-G(t_0+B_j)}{B_j}\bigg).
\end{flalign}

It follows from Lemma \ref{existence of F}, equality \eqref{derivatives of G c(t)1}, equality \eqref{derivatives of G c(t)2} and inequality \eqref{derivative of G 3} that we have

\begin{equation}
\label{derivative of G 4}
\begin{split}
&\int_{\{\Psi<-t_1\}}|\tilde{F}_{t_1}|^2e^{-\varphi}c(-\Psi)
-\int_{\{\Psi<-t_0\}}|F_{t_0}|^2e^{-\varphi}c(-\Psi)\\
\le&\int_{\{\Psi<-t_0\}}|\tilde{F}_{t_1}-F_{t_0}|^2e^{-\varphi}c(-\Psi)
+\int_{\{-t_0\le\Psi<-t_1\}}|\tilde{F}_{t_1}|^2e^{-\varphi}c(-\Psi)\\
\leq &
\int_{\{\Psi<-t_1\}}|\tilde{F}_{t_1}-\mathbb{I}_{\{\Psi< -t_0\}}F_{t_0}|^2e^{-\varphi}e^{-\Psi+v_{t_0}(\Psi)}c(-v_{t_0}(\Psi))\\
\le&\liminf_{j\to+\infty}\int_{\{\Psi<-t_1\}}|\tilde{F}_j-(1-b_{t_0,B_j}(\Psi))F_{t_0}|^2e^{-\varphi}c(-\Psi)\\
\leq &\liminf_{j\to+\infty}
\big(\frac{e^{t_0+B_j}\int^{t_0+B_j}_{t_1}c(t)e^{-t}dt}{\inf
\limits_{t\in(t_0,t_0+B_j)}c(t)} \times
\frac{G(t_0)-G(t_0+B_j)}{B_j}\big)\\
\le &\bigg(\int^{t_0}_{t_1}c(t)e^{-t}dt\bigg)\liminf\limits_{B\to 0+0} \frac{G(t_0)-G(t_0+B)}{\int_{t_0}^{t_0+B}c(t)e^{-t}dt}.
\end{split}
\end{equation}
Note that $(\tilde{F}_{t_1}-F_{t_0})_{z_0}\in \mathcal{O} (K_M)_{z_0}\otimes  J_{z_0}$ for any $z_0\in Z_0$. It follows from the definition of $G(t)$ and inequality \eqref{derivative of G 4} that we have

\begin{equation}
\label{derivative of G 5}
\begin{split}
&G(t_1)-G(t_0)\\
\le&\int_{\{\Psi<-t_1\}}|\tilde{F}_{t_1}|^2e^{-\varphi}c(-\Psi)
-\int_{\{\Psi<-t_0\}}|F_{t_0}|^2e^{-\varphi}c(-\Psi)\\
\le&\int_{\{\Psi<-t_1\}}|\tilde{F}_{t_1}-\mathbb{I}_{\{\Psi< -t_0\}}F_{t_0}|^2e^{-\varphi}c(-\Psi)\\
\leq &
\int_{\{\Psi<-t_1\}}|\tilde{F}_{t_1}-\mathbb{I}_{\{\Psi< -t_0\}}F_{t_0}|^2e^{-\varphi}e^{-\Psi+v_{t_0}(\Psi)}c(-v_{t_0}(\Psi))\\\
\le &\big(\int^{t_0}_{t_1}c(t)e^{-t}dt\big)\liminf\limits_{B\to 0+0} \frac{G(t_0)-G(t_0+B)}{\int_{t_0}^{t_0+B}c(t)e^{-t}dt}.
\end{split}
\end{equation}

Lemma \ref{derivatives of G} is proved.
\end{proof}

The following property of concave
functions will be used in the proof of Theorem \ref{main theorem}.
\begin{Lemma}[see \cite{G16}]
Let $H(r)$ be a lower semicontinuous function on $(0,R]$. Then $H(r)$ is concave
if and only if
\begin{equation}\nonumber
\begin{split}
\frac{H(r_1)-H(r_2)}{r_1-r_2} \leq
\liminf\limits_{r_3 \to r_2-0}
\frac{H(r_3)-H(r_2)}{r_3-r_2}
\end{split}
\end{equation}
holds for any $0<r_2<r_1 \leq R$.
\label{characterization of concave function}
\end{Lemma}

\subsection{Some required results}

We recall the following two well-known results about Lelong numbers.

\begin{Lemma}[\cite{skoda1972}\label{l:skoda}]
	Let $u$ be a plurisubharmonic function on $\Delta^n\subset\mathbb{C}^n$. If $v(dd^cu,o )<1$, then $e^{-2u}$ is $L^1$ on a neighborhood of $o$, where $o\in\Delta^n$ is the origin.
\end{Lemma}

\begin{Lemma}
	\label{l:1d-MIS}Let $\varphi$ be a subharmonic function on the unit disc $\Delta\subset\mathbb{C}$. Then we have $\mathcal{I}(\varphi)_o=(w^k)_o$ if and only if $v(dd^c(\varphi),o)\in[2k,2k+2)$.
\end{Lemma}
\begin{proof}
	For convenience of the readers, we give a proof of this lemma by using Lemma \ref{l:skoda} and Siu's Decomposition Theorem. It follows from Siu's Decomposition Theorem that $\varphi=v(dd^c(\varphi),o)\log|w|+\varphi_1$, where $\varphi_1$ is subharmonic function on $\Delta$ satisfying $v(dd^c(\varphi_1),o)=0$.
	
	For any positive integer $m$, if $v(dd^c(\varphi),o)<2m$, then there exists $p>1$ such that $pv(dd^c(\varphi),o)<2m$. Note that $2pm-2p-pv(dd^c(\varphi),o)\ge 2m-2-pv(dd^c(\varphi),o)>-2$. Using Lemma \ref{l:skoda} and H\"older inequality, we get that there exists a neighborhood $V$ of $o$ such that
	\begin{displaymath}
		\begin{split}
			\int_{V}|w^{m-1}|^{2}e^{-\varphi}&\le\int_{V}|w|^{2m-2-v(dd^c(\varphi),o)}e^{-\varphi_1}\\
			&\le\left(\int_{V}|w|^{2pm-2p-pv(dd^c(\varphi),o)}\right)^{\frac{1}{p}}\left(\int_{V}e^{-q\varphi_1}\right)^{\frac{1}{q}}\\
			&<+\infty,
		\end{split}
	\end{displaymath}
	where $\frac{1}{p}+\frac{1}{q}=1$. Then we have $(w^{m-1},o)\in\mathcal{I}(\varphi)_o$.
	
	For any positive integer $m$, if $v(dd^c(\varphi),o)\ge 2m$, then there exists a constant $C$ such that $\varphi\le 2m\log|w|+C$ near $o$, which implies that $(w^{m-1},o)\not\in\mathcal{I}(\varphi)_o$.
	
	Thus, we have $\mathcal{I}(\varphi)_o=(w^k)_o$ if and only if $v(dd^c(\varphi),o)\in[2k,2k+2)$.
	\end{proof}

We recall some basic properties of Green functions. Let $\Omega$ be an open Riemann surface, which admits a nontrivial Green function $G_{\Omega}$, and let $z_0\in\Omega$.

\begin{Lemma}[see \cite{S-O69}, see also \cite{Tsuji}] 	\label{l:green-sup}Let $w$ be a local coordinate on a neighborhood of $z_0$ satisfying $w(z_0)=0$.  $G_{\Omega}(z,z_0)=\sup_{v\in\Delta_{\Omega}^*(z_0)}v(z)$, where $\Delta_{\Omega}^*(z_0)$ is the set of negative subharmonic function on $\Omega$ such that $v-\log|w|$ has a locally finite upper bound near $z_0$. Moreover, $G_{\Omega}(\cdot,z_0)$ is harmonic on $\Omega\backslash\{z_0\}$ and $G_{\Omega}(\cdot,z_0)-\log|w|$ is harmonic near $z_0$.
\end{Lemma}

\begin{Lemma}[see \cite{GY-concavity}]\label{l:G-compact}
For any  open neighborhood $U$ of $z_0$, there exists $t>0$ such that $\{G_{\Omega}(z,z_0)<-t\}$ is a relatively compact subset of $U$.
\end{Lemma}

\subsection{Inner points on open Riemann surfaces}

Let $\Omega$  be an open Riemann surface, which admits a nontrivial Green function $G_{\Omega}$. Let $\psi$ be a negative subharmonic function on $\Omega$, and let $\varphi$ be a Lebesgue measurable function on $\Omega$ such that $\varphi+\psi$ is subharmonic on $\Omega$.

Let $Z_0\subset\{\psi=-\infty\}$ be a discrete subset of $\Omega$, and Let $\mathcal{F}_{z}\supset\mathcal{I}(\varphi+\psi)_{z}$ be an ideal of $\mathcal{O}_{z}$ for any $z\in Z_0$.
Let $f$ be a holomorphic $(1,0)$ form on a neighborhood of $Z_0$.  Let $c(t)$ be a positive measurable function on $(0,+\infty)$ satisfying $c(t)e^{-t}$ is decreasing on $(0,+\infty)$,  $\int_{0}^{+\infty}c(s)e^{-s}ds<+\infty$ and $e^{-\varphi}c(-\psi)$ has  a  positive lower bound on any compact subset of $\Omega\backslash E$, where $E\subset\{\psi=-\infty\}$ is a discrete point subset of $\Omega$.

Denote the minimal $L^{2}$ integrals (see \cite{GY-concavity}, see also \cite{GMY-concavity2})
\begin{equation*}
\begin{split}
\inf\Bigg\{\int_{\{\psi<-t\}}|\tilde{f}|^{2}e^{-\varphi}c(-\psi):(\tilde{f}-f,z)\in(\mathcal{O}(K_{\Omega})&)_{z}\otimes\mathcal{F}_{z} \text{ for any } z\in Z_0 \\&\&{\,}\tilde{f}\in H^{0}(\{\psi<-t\},\mathcal{O}(K_{\Omega}))\Bigg\}
\end{split}
\end{equation*}
by $G(t;c)$ (without misunderstanding, we denote $G(t;c)$ by $G(t)$),  where $t\in[0,+\infty)$ and
$|f|^{2}:=\sqrt{-1}f\wedge\bar{f}$ for any $(1,0)$ form $f$.

Theorem \ref{main theorem} shows that $G(h^{-1}(r))$ is concave with respect to $r\in(0,\int_0^{+\infty}c(s)e^{-s}ds)$ (see also \cite{GY-concavity,GMY-concavity2,GMY-boundary2}), where $h(t)=\int_t^{+\infty}c(s)e^{-s}ds$. In this section, we discuss the case that the concavity degenerates to linearity.

We recall the following two lemmas, which will be used in the following discussion.

 \begin{Lemma}[see \cite{GY-concavity}]
 	\label{l:cu}
 	Let $T$ be a closed positive $(1,1)$ current on $\Omega$. For any open set $U\subset\subset \Omega$ satisfying $U\cap SuppT\not=\emptyset$,  there exists a subharmonic function $\Phi<0$ on $\Omega$, which satisfies the following properties:
 	
 	$(1)$ $i\partial\bar\partial\Phi\leq T$ and $i\partial\bar\partial\Phi\not\equiv0$;
 	
 	$(2)$ $\lim_{t\rightarrow0+0}\left(\inf_{\{G_{\Omega}(z,z_0)\geq-t\}}\Phi(z)\right)=0$;
 	
 	$(3)$ $Supp(i\partial\bar\partial\Phi)\subset U$ and $\inf_{\Omega\backslash U}\Phi>-\infty$.
 \end{Lemma}

 Denote that
 \begin{displaymath}
 	\begin{split}
 		\mathcal{H}^2(c,\varphi,t):=\Bigg\{\tilde f\in H^0(\{\psi<-t\}&,\mathcal{O}(K_{\Omega})):\int_{\{\psi<-t\}}|\tilde f|^2e^{-\varphi}c(-\psi)<+\infty \\&\&{\,}(\tilde{f}-f,z)\in(\mathcal{O}(K_{\Omega}))_{z}\otimes\mathcal{F}_{z} \text{ for any }z\in Z_0\Bigg\}.
 	\end{split}
 \end{displaymath}

 \begin{Lemma}
[\cite{GY-concavity}]
 	\label{l:n} Assume that there exists $t_0\geq0$ such that $G(t_0)\in(0,+\infty)$.
 If $G(\hat{h}^{-1}(r))$ is linear with respect to $r$, then there is no  Lebesgue measurable function $\tilde \varphi\geq\varphi$ such that $\tilde\varphi+\psi$ is subharmonic function on $M$ and satisfies:
	
	$(1)$ $\tilde\varphi\not=\varphi$ and $\mathcal{I}(\tilde\varphi+\psi)=\mathcal{I}(\varphi+\psi)$;
	
	$(2)$ $\lim_{t\rightarrow 0+0}\sup_{\{\psi\geq-t\}}(\tilde\varphi-\varphi)=0$;
	
	$(3)$ there exists an open subset $U\subset\subset\Omega$ such that $\sup_{\Omega\backslash U}(\tilde\varphi-\varphi)<+\infty$, $e^{-\tilde\varphi}c(-\psi)$ has a positive lower bound on $U$ and $\int_{U}|F_1-F_2|^2e^{-\varphi}c(-\psi)<+\infty$ for any $F_1\in\mathcal{H}^2(c,\tilde\varphi,t)$ and $F_2\in\mathcal{H}^2(c,\varphi,t)$, where $t\ge 0$ such that $U\subset\subset\{\psi<-t\}$. 	
 \end{Lemma}

We recall some notations (see \cite{OF81}, see also \cite{guan-zhou13ap,GY-concavity,GMY-concavity2}).  Let $p:\Delta\rightarrow\Omega$ be the universal covering from unit disc $\Delta$ to $\Omega$.  we call the holomorphic function $f$  on $\Delta$ a multiplicative function,
 if there is a character $\chi$, which is the representation of the fundamental group of $\Omega$, such that $g^{\star}f=\chi(g)f$,
 where $|\chi|=1$ and $g$ is an element of the fundamental group of $\Omega$.
It is known that for any harmonic function $u$ on $\Omega$,
there exist a $\chi_{u}$ (the  character associate to $u$) and a multiplicative function $f_u\in\mathcal{O}^{\chi_{u}}(\Omega)$,
such that $|f_u|=p^{*}(e^{u})$.
Let $z_0\in \Omega$.
Recall that for the Green function $G_{\Omega}(z,z_0)$,
there exist a $\chi_{z_0}$ and a multiplicative function $f_{z_0}\in\mathcal{O}^{\chi_{z_0}}(\Omega)$, such that $|f_{z_0}(z)|=p^{*}\left(e^{G_{\Omega}(z,z_0)}\right)$ .

Firstly, we consider the case that $Z_0$ is finite. Let $Z_0=\{z_1,z_2,\ldots ,z_m\}\subset\Omega$ be a finite subset of $\Omega$ satisfying that $z_j\not=z_k$ for any $j\not=k$.

The following Theorem gives a characterization of the concavity of $G(h^{-1}(r))$ degenerating to linearity.
\begin{Theorem}[see \cite{GY-concavity3}]
	\label{thm:m-points}
 Assume that $G(0)\in(0,+\infty)$ and $(\psi-2p_jG_{\Omega}(\cdot,z_j))(z_j)>-\infty$ for  $j\in\{1,2,\ldots,m\}$, where $p_j=\frac{1}{2}v(dd^c(\psi),z_j)>0$. Then $G(h^{-1}(r))$ is linear with respect to $r$ if and only if the following statements hold:
	
	$(1)$ $\psi=2\sum_{1\le j\le m}p_jG_{\Omega}(\cdot,z_j)$;
	
	$(2)$ $\varphi+\psi=2\log|g|+2\sum_{1\le j\le m}G_{\Omega}(\cdot,z_j)+2u$ and $\mathcal{F}_{z_j}=\mathcal{I}(\varphi+\psi)_{z_j}$ for any $j\in\{1,2,\ldots,m\}$, where $g$ is a holomorphic function on $\Omega$ such that $ord_{z_j}(g)=ord_{z_j}(f)$ for any $j\in\{1,2,\ldots,m\}$ and $u$ is a harmonic function on $\Omega$;
	
	$(3)$ $\prod_{1\le j\le m}\chi_{z_j}=\chi_{-u}$, where $\chi_{-u}$ and $\chi_{z_j}$ are the  characters associated to the functions $-u$ and $G_{\Omega}(\cdot,z_j)$ respectively;
	
	$(4)$  $\lim_{z\rightarrow z_k}\frac{f}{gp_*\left(f_u\left(\prod_{1\le j\le m}f_{z_j}\right)\left(\sum_{1\le j\le m}p_{j}\frac{d{f_{z_{j}}}}{f_{z_{j}}}\right)\right)}=c_0$ for any $k\in\{1,2,\ldots,m\}$, where $c_0\in\mathbb{C}\backslash\{0\}$ is a constant independent of $k$, $f_{u}$ is a holomorphic function $\Delta$ such that $|f_u|=p^*(e^u)$ and $f_{z_j}$ is a holomorphic function on $\Delta$ such that $|f_{z_j}|=p^*\left(e^{G_{\Omega}(\cdot,z_j)}\right)$ for any $j\in\{1,2,\ldots,m\}$.
\end{Theorem}

\begin{Remark}\label{r:equivalent}
	The statements $(2)$, $(3)$ and $(4)$ hold if and only if the following statements hold:
	
		$(1)'$ $\varphi+\psi=2\log|g_1|$ and $\mathcal{F}_{z_j}=\mathcal{I}(\varphi+\psi)_{z_j}$ for any $j\in\{1,2,\ldots,m\}$, where $g_1$ is a holomorphic function on $\Omega$ such that $ord_{z_j}(g_1)=ord_{z_j}(f)+1$ for any $j\in\{1,2,\ldots,m\}$;

	$(2)'$  $\frac{ord_{z_j}(g_1)}{p_j}\lim_{z\rightarrow z_j}\frac{f}{dg_1}=c_0$ for any $j\in\{1,2,\ldots,m\}$, where $c_0\in\mathbb{C}\backslash\{0\}$ is a constant independent of $j$;
\end{Remark}
\begin{proof}
	 If statements $(2)$ and $(3)$ in Theorem \ref{thm:m-points} hold, following from the definitions of $\chi_{z_j}$ and $\chi_{-u}$, we know that there exists a (single-value) holomorphic function $g_2$ on $\Omega$ such that $g_2=p_*\left(f_u\prod_{1\le j\le m}f_{z_j} \right)$ on $\Omega$. Let $g_1=gg_2$ on $\Omega$, thus $$\varphi+\psi=2\log|g_1|$$ and $ord_{z_j}(g_1)=ord_{z_j}(g)+ord_{z_j}(g_2)=ord_{z_j}(f)+1$ for any $j\in\{1,2,\ldots,m\}$. Note that $ord_{z_k}(g)=ord_{z_k}(dg_1)$ and
	 \begin{equation}
	 	\label{eq:0319a}
	 	\begin{split}
	 		&\lim_{z\rightarrow z_k}\frac{dg_1}{gp_*\left(f_u\left(\prod_{1\le j\le m}f_{z_j}\right)\left(\sum_{1\le j\le m}p_{j}\frac{d{f_{z_{j}}}}{f_{z_{j}}}\right)\right)}\\
	 		=&\lim_{z\rightarrow z_k}\frac{d\left(gp_*\left(f_u\prod_{1\le j\le m}f_{z_j} \right)\right)}{gp_*\left(f_u\left(\prod_{1\le j\le m}f_{z_j}\right)\left(p_{k}\frac{d{f_{z_{k}}}}{f_{z_{k}}}\right)\right)}\\
	 		=&\frac{ord_{z_k}(g_1)}{p_k}
	 	\end{split}
	 \end{equation}
	 for any $k\in\{1,2,\ldots,m\}$. 	 Furthermore, if statement $(4)$ in Theorem \ref{thm:m-points} holds, following from equality \eqref{eq:0319a}, we have  $\frac{ord_{z_j}(g_1)}{p_j}\lim_{z\rightarrow z_j}\frac{f}{dg_1}=c_0$ for any $j\in\{1,2,\ldots,m\}$. Thus, the statements $(2)$, $(3)$ and $(4)$ in Theorem \ref{thm:m-points} implies the statements $(1)'$ and $(2)'$.
  	
	 If statement $(1)'$  holds, it follows from Weierstrass Theorem on open Riemann surface (see \cite{OF81}) that there exists a holomorphic function $g$ on $\Omega$ such that  $$ord_{z_j}(g)=ord_{z_j}(g_1)-1=ord_{z_j}(f)$$ for any $j\in\{1,2,\ldots,m\}$ and $g_2:=\frac{g_1}{g}$ is a holomorphic function on $\Omega$ satisfying $g_2(z)\not=0$ for any $z\in\Omega\backslash\{z_1,z_2,\dots,z_m\}$. Thus, there exists a harmonic function $u$ on $\Omega$ such that $$2u=2\log|g_2|-\sum_{1\le j\le m}2G_{\Omega}(\cdot,z_j)=\varphi+\psi-2\log|g|-2\sum_{1\le j\le m}G_{\Omega}(\cdot,z_j).$$
	 Note that
	 $g_2=c_1p_*\left(f_u\prod_{1\le j\leq m}f_{z_j} \right)$
	 and
	 \begin{equation}\label{eq:0319b}
	 	\begin{split}
	 			&\lim_{z\rightarrow z_k}\frac{dg_1}{gp_*\left(f_u\left(\prod_{1\le j\le m}f_{z_j}\right)\left(\sum_{1\le j\le m}p_{j}\frac{d{f_{z_{j}}}}{f_{z_{j}}}\right)\right)}\\
	 		=&\lim_{z\rightarrow z_k}\frac{d\left(c_1gp_*\left(f_u\prod_{1\le j\le m}f_{z_j} \right)\right)}{gp_*\left(f_u\left(\prod_{1\le j\le m}f_{z_j}\right)\left(p_{k}\frac{d{f_{z_{k}}}}{f_{z_{k}}}\right)\right)}\\
	 		=&c_1\frac{ord_{z_k}(g_1)}{p_k},
	 	\end{split}
	 \end{equation}
	  where $c_1\in\mathbb{C}$ is a constant satisfying $|c_1|=1$.
 Furthermore, if statement $(2)'$ holds, it follows from 	equality \eqref{eq:0319b} that  $\lim_{z\rightarrow z_k}\frac{f}{gp_*\left(f_u\left(\prod_{1\le j\le m}f_{z_j}\right)\left(\sum_{1\le j\le m}p_{j}\frac{d{f_{z_{j}}}}{f_{z_{j}}}\right)\right)}=c_0c_1$ for any $k\in\{1,2,\ldots,m\}$. Thus, the statements $(1)'$ and $(2)'$ implies the statements $(2)$, $(3)$ and $(4)$ in Theorem \ref{thm:m-points}.	 	
	 \end{proof}

We give a generalization of Theorem \ref{thm:m-points}, which will be used in the proof of Proposition \ref{p:n-linearity1}.
\begin{Proposition}
	\label{p:inner1}
	 Let $G(0)\in(0,+\infty)$ and $p_j=\frac{1}{2}v(dd^c(\psi),z_j)>0$ for any $j\in\{1,2,\ldots,m\}$. For any $j\in\{1,2,\ldots,m\}$, assume that one of the following conditions holds:
	
	$(A)$ $\varphi+a\psi$ is  subharmonic near $z_j$ for some $a\in[0,1)$;
	
	$(B)$ $(\psi-2p_jG_{\Omega}(\cdot,z_j))(z_j)>-\infty$.
	
 Then $G(h^{-1}(r))$ is linear with respect to $r$ if and only if the following statements hold:
	
	$(1)$ $\psi=2\sum_{1\le j\le m}p_jG_{\Omega}(\cdot,z_j)$;
	
	$(2)$ $\varphi+\psi=2\log|g|+2\sum_{1\le j\le m}G_{\Omega}(\cdot,z_j)+2u$ and $\mathcal{F}_{z_j}=\mathcal{I}(\varphi+\psi)_{z_j}$ for any $j\in\{1,2,\ldots,m\}$, where $g$ is a holomorphic function on $\Omega$ such that $ord_{z_j}(g)=ord_{z_j}(f)$ for any $j\in\{1,2,\ldots,m\}$ and $u$ is a harmonic function on $\Omega$;
	
	$(3)$ $\prod_{1\le j\le m}\chi_{z_j}=\chi_{-u}$, where $\chi_{-u}$ and $\chi_{z_j}$ are the  characters associated to the functions $-u$ and $G_{\Omega}(\cdot,z_j)$ respectively;
	
	$(4)$  $\lim_{z\rightarrow z_k}\frac{f}{gp_*\left(f_u\left(\prod_{1\le j\le m}f_{z_j}\right)\left(\sum_{1\le j\le m}p_{j}\frac{d{f_{z_{j}}}}{f_{z_{j}}}\right)\right)}=c_0$ for any $k\in\{1,2,\ldots,m\}$, where $c_0\in\mathbb{C}\backslash\{0\}$ is a constant independent of $k$.
\end{Proposition}
\begin{proof}
The sufficiency follows from
	 Theorem \ref{thm:m-points}. Thus, we just need to prove the necessity.
	 It follows from Theorem \ref{thm:m-points} that it suffices to prove $(\psi-2p_jG_{\Omega}(\cdot,z_j))(z_j)>-\infty$ for any $j\in\{1,2,\ldots,m\}$.
	
	 We prove $(\psi-2p_jG_{\Omega}(\cdot,z_j))(z_j)>-\infty$ for any $j\in\{1,2,\ldots,m\}$ by contradiction: if not, there exists
	 $j\in\{1,2,\ldots,m\}$ such that $(\psi-2p_jG_{\Omega}(\cdot,z_j))(z_j)=-\infty$, then we have $\varphi+a\psi$ is subharmonic near $z_j$ for some $a\in[0,1)$. As $\psi$ is  subharmonic function on $\Omega$, it follows from  Siu's Decomposition Theorem that $\psi=2p_jG_{\Omega}(\cdot,z_j)+\psi_1$ such that $v(dd^c(\psi_1),z_j)=0$ and $\psi_1(z_j)=-\infty$. Note that $\psi_1$ is subharmonic and not harmonic near $z_j$.
  There exists a closed  positive $(1,1)$ current $T\not\equiv0$, such that $SuppT\subset\subset V$ and $T\leq \frac{1-a}{2}i\partial\bar\partial\psi_1$ on $V$, where $V\Subset\Omega$ is an open neighborhood of $z_j$ satisfying that  $\varphi+a\psi$ is subharmonic on a neighborhood of $\overline{V}$. Without loss of generality, assume that $\{z\in \overline{V}:\mathcal{I}(\varphi+\psi)_z\not=\mathcal{O}_z\}=\{z\in \overline{V}:v(dd^c(\varphi+\psi),z)\ge2\}=\{z_j\}$.

By Lemma \ref{l:cu}, there exists a subharmonic function $\Phi<0$ on $\Omega$, which satisfies the following properties: $i\partial\bar\partial\Phi\leq T$ and $i\partial\bar\partial\Phi\not\equiv0$; $\lim_{t\rightarrow0+0}(\inf_{\{G_{\Omega}(\cdot,z_j)\geq-t\}}\Phi(z))=0$; $Supp(i\partial\bar\partial\Phi)\subset V$ and $\inf_{\Omega\backslash V}\Phi>-\infty$. It follows from Lemma \ref{l:green-sup}, $v(dd^{c}\psi,z_j)>0$ and $\psi<0$ on $\Omega$ that $\lim_{t\rightarrow0+0}(\inf_{\{\psi\geq-t\}}\Phi(z))=0$.

Set $\tilde\varphi=\varphi-\Phi$, then $\tilde\varphi+\psi=\varphi+a\psi+2(1-a)p_jG_{\Omega}(\cdot,z_j)+(1-a)\psi_1-\Phi$ on $V$. As $(1-a)\psi_1-\Phi$ is subharmonic on $V$, we have $\tilde\varphi+\psi$ is subharmonic on $V$. It is clear that $\tilde\varphi\geq\varphi$ and $\tilde\varphi\not=\varphi$. It follows from $SuppT\subset\subset V$ and $i\partial\bar\partial\Phi\leq T\le  i\partial\bar\partial\psi_1$ that $\tilde\varphi+\psi$ is subharmonic on $\Omega$ and $\mathcal{I}(\tilde\varphi+\psi)=\mathcal{I}(\varphi+\psi)$.

It follows from Remark \ref{rem:linear} that we can assume that $c(t)>e^{\frac{(1+a)t}{2}}$ for any $t>0$. $T\leq \frac{1-a}{2}i\partial\bar\partial\psi_1$ on $V$ and $i\partial\bar\partial\Phi\subset\subset V$ show that $\frac{1-a}{2}\psi-\Phi$ is subharmonic on $\Omega$, which implies that
$e^{-\tilde\varphi}c(-\psi)\geq e^{-\varphi-a\psi}e^{\Phi-\frac{1-a}{2}\psi}$
 has positive lower bound on $V$. Notice that $\inf_{\Omega\backslash V}(\varphi-\tilde\varphi)=\inf_{\Omega\backslash V}\Phi>-\infty$ and $\int_{V}|F_1-F_2|^2e^{-\varphi}c(-\psi)\leq C\int_{V}|F_1-F_2|^2e^{-\varphi-\psi}<+\infty$ for any $F_1\in\mathcal{H}^2(c,\tilde\varphi,t)$ and $F_2\in\mathcal{H}^2(c,\varphi,t)$, where $V\subset\subset\{\psi<-t\}$, then $\tilde\varphi$ satisfies the conditions in Lemma \ref{l:n}, which contradicts to the result of  Lemma \ref{l:n}. Thus $\psi_1(z_j)>-\infty$.

 Thus, Proposition \ref{p:inner1} holds.
\end{proof}

Now, we consider the case that $Z_0$ is infinite. Let $Z_0=\{z_j:j\in\mathbb{Z}_{\ge1}\}\subset\Omega$ be a  discrete set of infinite points satisfying $z_j\not=z_k$ for any $j\not=k$.

 The following result gives a necessary condition for $G(h^{-1}(r))$ is linear.

\begin{Proposition}[\cite{GY-concavity3}]	\label{p:infinite}
 Assume that $G(0)\in(0,+\infty)$ and $(\psi-2p_jG_{\Omega}(\cdot,z_j))(z_j)>-\infty$ for  $j\in\mathbb{Z}_{\ge1}$, where $p_j=\frac{1}{2}v(dd^c(\psi),z_j)>0$. Assume that $G(h^{-1}(r))$ is linear with respect to $r$. Then the following statements hold:
	
	$(1)$ $\psi=2\sum_{j\in\mathbb{Z}_{\ge1}}p_jG_{\Omega}(\cdot,z_j)$;

	$(2)$ $\varphi+\psi=2\log|g|$ and $\mathcal{F}_{z_j}=\mathcal{I}(\varphi+\psi)_{z_j}$ for any $j\in\mathbb{Z}_{\ge1}$, where $g$ is a holomorphic function on $\Omega$ such that $ord_{z_j}(g)=ord_{z_j}(f)+1$ for any $j\in\mathbb{Z}_{\ge1}$;

	$(3)$  $\frac{p_j}{ord_{z_j}(g)}\lim_{z\rightarrow z_j}\frac{dg}{f}=c_0$ for any $j\in\mathbb{Z}_{\ge1}$, where $c_0\in\mathbb{C}\backslash\{0\}$ is a constant independent of $j$;
	
		$(4)$ $\sum_{j\in\mathbb{Z}_{\ge1}}p_j<+\infty$.
\end{Proposition}

We give a generalization of Proposition \ref{p:infinite}, which will be used in the proof of Proposition \ref{p:n-linearity1}.

\begin{Proposition}
 	\label{p:inner2} 	 Let $G(0)\in(0,+\infty)$ and $p_j=\frac{1}{2}v(dd^c(\psi),z_j)>0$ for any $j\in\mathbb{Z}_{\ge1}$. For any $j\in\mathbb{Z}_{\ge1}$, assume that one of the following conditions holds:
	
	$(A)$ $\varphi+a\psi$ is  subharmonic near $z_j$ for some $a\in[0,1)$;
	
	$(B)$ $(\psi-2p_jG_{\Omega}(\cdot,z_j))(z_j)>-\infty$.

 If $G(h^{-1}(r))$ is linear with respect to $r$, then the following statements hold:
	
	$(1)$ $\psi=2\sum_{j\in\mathbb{Z}_{\ge1}}p_jG_{\Omega}(\cdot,z_j)$;

	$(2)$ $\varphi+\psi=2\log|g|$ and $\mathcal{F}_{z_j}=\mathcal{I}(\varphi+\psi)_{z_j}$ for any $j\in\mathbb{Z}_{\ge1}$, where $g$ is a holomorphic function on $\Omega$ such that $ord_{z_j}(g)=ord_{z_j}(f)+1$ for any $j\in\mathbb{Z}_{\ge1}$;

	$(3)$  $\frac{p_j}{ord_{z_j}(g)}\lim_{z\rightarrow z_j}\frac{dg}{f}=c_0$ for any $j\in\mathbb{Z}_{\ge1}$, where $c_0\in\mathbb{C}\backslash\{0\}$ is a constant independent of $j$;
	
		$(4)$ $\sum_{j\in\mathbb{Z}_{\ge1}}p_j<+\infty$.	
 \end{Proposition}

 \begin{proof} The proof of Proposition \ref{p:inner2} is similar to Proposition \ref{p:inner1}.

	 It follows from Proposition \ref{p:infinite} that it suffices to prove $(\psi-2p_jG_{\Omega}(\cdot,z_j))(z_j)>-\infty$ for any $j\in\mathbb{Z}_{\ge1}$.
	
	 We prove $(\psi-2p_jG_{\Omega}(\cdot,z_j))(z_j)>-\infty$ for any $j\in\mathbb{Z}_{\ge1}$ by contradiction: if not, there exists
	 $j\in\mathbb{Z}_{\ge1}$ such that $(\psi-2p_jG_{\Omega}(\cdot,z_j))(z_j)=-\infty$, then we have $\varphi+a\psi$ is subharmonic near $z_j$ for some $a\in[0,1)$. As $\psi$ is  subharmonic function on $\Omega$, it follows from  Siu's Decomposition Theorem that $\psi=2p_jG_{\Omega}(\cdot,z_j)+\psi_1$ such that $v(dd^c(\psi_1),z_j)=0$ and $\psi_1(z_j)=-\infty$. Note that $\psi_1$ is subharmonic and not harmonic near $z_j$.
  There exists a closed  positive $(1,1)$ current $T\not\equiv0$, such that $SuppT\subset\subset V$, $T\leq \frac{1-a}{2}i\partial\bar\partial\psi_1$ on $V$, where $V\Subset\Omega$ is an open neighborhood of $z_j$ satisfying that  $\varphi+a\psi$ is subharmonic on a neighborhood of $\overline{V}$. Without loss of generality, assume that $\{z\in \overline{V}:\mathcal{I}(\varphi+\psi)_z\not=\mathcal{O}_z\}=\{z\in \overline{V}:v(dd^c(\varphi+\psi),z)\ge2\}=\{z_j\}$.

Using Lemma \ref{l:cu}, there exists a subharmonic function $\Phi<0$ on $\Omega$, which satisfies the following properties: $i\partial\bar\partial\Phi\leq T$ and $i\partial\bar\partial\Phi\not\equiv0$; $\lim_{t\rightarrow0+0}(\inf_{\{G_{\Omega}(\cdot,z_j)\geq-t\}}\Phi(z))=0$; $Supp(i\partial\bar\partial\Phi)\subset V$ and $\inf_{\Omega\backslash V}\Phi>-\infty$. It following from Lemma \ref{l:green-sup}, $v(dd^{c}\psi,z_j)>0$ and $\psi<0$ on $\Omega$ that $\lim_{t\rightarrow0+0}(\inf_{\{\psi\geq-t\}}\Phi(z))=0$.

Set $\tilde\varphi=\varphi-\Phi$, then $\tilde\varphi+\psi=\varphi+a\psi+2(1-a)p_jG_{\Omega}(\cdot,z_j)+(1-a)\psi_1-\Phi$ on $V$. As $(1-a)\psi_1-\Phi$ is subharmonic on $V$, we have $\tilde\varphi+\psi$ is subharmonic on $V$. It is clear that $\tilde\varphi\geq\varphi$ and $\tilde\varphi\not=\varphi$. It follows from $SuppT\subset\subset V$ and $i\partial\bar\partial\Phi\leq T\le  i\partial\bar\partial\psi_1$ that $\tilde\varphi+\psi$ is subharmonic on $\Omega$ and $\mathcal{I}(\tilde\varphi+\psi)=\mathcal{I}(\varphi+\psi)$.

It follows from Remark \ref{rem:linear} that we can assume that $c(t)>e^{\frac{(1+a)t}{2}}$ for any $t>0$. $T\leq \frac{1-a}{2}i\partial\bar\partial\psi_1$ on $V$ and $i\partial\bar\partial\Phi\subset\subset V$ show that $\frac{1-a}{2}\psi-\Phi$ is subharmonic on $\Omega$, which implies that
$e^{-\tilde\varphi}c(-\psi)\geq e^{-\varphi-a\psi}e^{\Phi-\frac{1-a}{2}\psi}$
 has positive lower bound on $V$. Notice that $\inf_{\Omega\backslash V}(\varphi-\tilde\varphi)=\inf_{\Omega\backslash V}\Phi>-\infty$ and $\int_{V}|F_1-F_2|^2e^{-\varphi}c(-\psi)\leq C\int_{V}|F_1-F_2|^2e^{-\varphi-\psi}<+\infty$ for any $F_1\in\mathcal{H}^2(c,\tilde\varphi,t)$ and $F_2\in\mathcal{H}^2(c,\varphi,t)$, where $V\subset\subset\{\psi<-t\}$, then $\tilde\varphi$ satisfies the conditions in Lemma \ref{l:n}, which contradicts to the result of  Lemma \ref{l:n}. Thus $\psi_1(z_j)>-\infty$.

 Thus, Proposition \ref{p:inner2} holds.
 \end{proof}

\section{Proofs of Theorem \ref{main theorem}, Remark \ref{infty2}, Corollary \ref{necessary condition for linear of G} and Remark \ref{rem:linear}}
We firstly prove Theorem \ref{main theorem}.

\begin{proof} We firstly show that if $G(t_0)<+\infty$ for some $t_0> T$, then $G(t_1)<+\infty$ for any $T< t_1<t_0$. As $G(t_0)<+\infty$, it follows from Lemma \ref{existence of F} that there exists a unique
holomorphic $(n,0)$ form $F_{t_0}$ on $\{\Psi<-t\}$ satisfying
$$\ \int_{\{\Psi<-t_0\}}|F_{t_0}|^2e^{-\varphi}c(-\Psi)=G(t_0)<+\infty$$  and
$\ (F_{t_0}-f)_{z_0}\in
\mathcal{O} (K_M)_{z_0} \otimes J_{z_0}$, for any  $ z_0\in Z_0$.

It follows from Lemma \ref{L2 method in JM concavity} that there exists a holomorphic $(n,0)$ form $\tilde{F}_1$ on $\{\Psi<-t_1\}$ such that

 \begin{equation}\label{main theorem 1}
  \begin{split}
      & \int_{\{\Psi<-t_1\}}|\tilde{F}_1-(1-b_{t_0,B}(\Psi))F_{t_0}|^2e^{-\varphi+v_{t_0,B}(\Psi)-\Psi}c(-v_{t_0,B}(\Psi)) \\
      \le & (\int_{t_1}^{t_0+B}c(s)e^{-s}ds)
       \int_{M}\frac{1}{B}\mathbb{I}_{\{-t_0-B<\Psi<-t_0\}}|F_{t_0}|^2e^{-\varphi-\Psi}<+\infty.
  \end{split}
\end{equation}
Note that $b_{t_0,B}(t)=0$ on $\{\Psi<-t_0-B\}$ and $v_{t_0,B}(\Psi)> -t_0-B$. We have $e^{v_{t_0,B}(\Psi)}c(-v_{t_0,B}(\Psi))$ has a positive lower bound. It follows from inequality \eqref{main theorem 1} that
we have  $\ (\tilde{F}_{1}-F_{t_0})_{z_0}
\in \mathcal{O} (K_M)_{z_0}\otimes I(\Psi+\varphi)_{z_0} \subset \mathcal{O} (K_M)_{z_0}\otimes J_{z_0}$ for any $z_0\in Z_0$, which implies that
$(\tilde{F}_{1}-f)_{z_0}\in
\mathcal{O} (K_M)_{z_0} \otimes J_{z_0}$, for any  $ z_0\in Z_0$. As $v_{t_0,B}(\Psi)\ge\Psi$ and $c(t)e^{-t}$ is decreasing with respect to $t$, it follows from inequality \eqref{main theorem 1} that we have
 \begin{equation}\label{main theorem 2}
  \begin{split}
      & \int_{\{\Psi<-t_1\}}|\tilde{F}_1-(1-b_{t_0,B}(\Psi))F_{t_0}|^2e^{-\varphi}c(-\Psi)
      \\
      \le&\int_{\{\Psi<-t_1\}}|\tilde{F}_1-(1-b_{t_0,B}(\Psi))F_{t_0}|^2e^{-\varphi+v_{t_0,B}(\Psi)-\Psi}c(-v_{t_0,B}(\Psi)) \\
      \le & (\int_{t_1}^{t_0+B}c(s)e^{-s}ds)
       \int_{M}\frac{1}{B}\mathbb{I}_{\{-t_0-B<\Psi<-t_0\}}|F_{t_0}|^2e^{-\varphi-\Psi}<+\infty.
  \end{split}
\end{equation}
Then we have
 \begin{equation}\label{main theorem 3}
  \begin{split}
  &\int_{\{\Psi<-t_1\}}|\tilde{F}_1|^2e^{-\varphi}c(-\Psi)\\
     \le & 2\int_{\{\Psi<-t_1\}}|\tilde{F}_1-(1-b_{t_0,B}(\Psi))F_{t_0}|^2e^{-\varphi}c(-\Psi)
      +2\int_{\{\Psi<-t_1\}}|(1-b_{t_0,B}(\Psi))F_{t_0}|^2e^{-\varphi}c(-\Psi)\\
      \le & 2(\int_{t_1}^{t_0+B}c(s)e^{-s}ds)
       \int_{M}\frac{1}{B}\mathbb{I}_{\{-t_0-B<\Psi<-t_0\}}|F_{t_0}|^2e^{-\varphi-\Psi}
       +2\int_{\{\Psi<-t_0\}}|F_{t_0}|^2e^{-\varphi}c(-\Psi)\\
       <&+\infty.
  \end{split}
\end{equation}
Hence we have $G(t_1)\le \int_{\{\Psi<-t_1\}}|\tilde{F}_1|^2e^{-\varphi}c(-\Psi)<+\infty$.

Now, it follows from Lemma \ref{semicontinuous}, Lemma \ref{derivatives of G} and Lemma \ref{characterization of concave function} that we know $G(h^{-1}(r))$ is concave with respect to $r$. It follows from Lemma \ref{semicontinuous} that $\lim\limits_{t\to T+0}G(t)=G(T)$ and $\lim\limits_{t\to+\infty}G(t)=0$.

Theorem \ref{main theorem} is proved.

\end{proof}

Now we prove Remark \ref{infty2}.
\begin{proof}Note that if there exists a positive decreasing concave function $g(t)$ on $(a,b)\subset\mathbb{R}$ and $g(t)$ is not a constant function, then $b<+\infty$.

Assume that $G(t_0)<+\infty$ for some $t_0\geq T$. As $f_{z_0}\notin
\mathcal{O} (K_M)_{z_0} \otimes J_{z_0}$ for some  $ z_0\in Z_0$, Lemma \ref{characterization of g(t)=0} shows that $G(t_0)\in(0,+\infty)$. Following from Theorem \ref{main theorem} we know $G({h}^{-1}(r))$ is concave with respect to $r\in(\int_{T_1}^{T}c(t)e^{-t}dt,\int_{T_1}^{+\infty}c(t)e^{-t}dt)$ and $G({h}^{-1}(r))$ is not a constant function, therefore we obtain $\int_{T_1}^{+\infty}c(t)e^{-t}dt<+\infty$, which contradicts to $\int_{T_1}^{+\infty}c(t)e^{-t}dt=+\infty$. Thus we have $G(t)\equiv+\infty$.

When $G(t_2)\in(0,+\infty)$ for some $t_2\in[T,+\infty)$, Lemma \ref{characterization of g(t)=0} shows that $f_{z_0}\notin
\mathcal{O} (K_M)_{z_0} \otimes J_{z_0}$, for any  $ z_0\in Z_0$. Combining the above discussion, we know $\int_{T_1}^{+\infty}c(t)e^{-t}dt<+\infty$. Using Theorem \ref{main theorem}, we obtain that $G(\hat{h}^{-1}(r))$ is concave with respect to  $r\in (0,\int_{T}^{+\infty}c(t)e^{-t}dt)$, where $\hat{h}(t)=\int_{t}^{+\infty}c(l)e^{-l}dl$.

Thus, Remark \ref{infty2} holds.
\end{proof}

Now we prove Corollary \ref{necessary condition for linear of G}.

\begin{proof} As $G(h^{-1}(r))$ is linear with respect to $r\in[0,\int_T^{+\infty}c(s)e^{-s}ds)$, we have $G(t)=\frac{G(T_1)}{\int_{T_1}^{+\infty}c(s)e^{-s}ds}\int_{t}^{+\infty}c(s)e^{-s}ds$ for any $t\in[T,+\infty)$ and $T_1 \in (T,+\infty)$.

We follow the notation and the construction in Lemma \ref{derivatives of G}. Let $t_0>t_1> T$ be given. It follows from $G(h^{-1}(r))$ is linear with respect to $r\in[0,\int_T^{+\infty}c(s)e^{-s}ds)$ that we know that all inequalities in \eqref{derivative of G 5} should be equalities, i.e., we have

\begin{equation}
\label{necessary condition for linear of G 1}
\begin{split}
&G(t_1)-G(t_0)\\
=&\int_{\{\Psi<-t_1\}}|\tilde{F}_{t_1}|^2e^{-\varphi}c(-\Psi)
-\int_{\{\Psi<-t_0\}}|F_{t_0}|^2e^{-\varphi}c(-\Psi)\\
=&\int_{\{\Psi<-t_1\}}|\tilde{F}_{t_1}-\mathbb{I}_{\{\Psi< -t_0\}}F_{t_0}|^2e^{-\varphi}c(-\Psi)\\
= &
\int_{\{\Psi<-t_1\}}|\tilde{F}_{t_1}-\mathbb{I}_{\{\Psi< -t_0\}}F_{t_0}|^2e^{-\varphi}e^{-\Psi+v_{t_0}(\Psi)}c(-v_{t_0}(\Psi))\\
= &\big(\int^{t_0}_{t_1}c(t)e^{-t}dt\big)\liminf\limits_{B\to 0+0} \frac{G(t_0)-G(t_0+B)}{\int_{t_0}^{t_0+B}c(t)e^{-t}dt}.
\end{split}
\end{equation}
Note that $G(t_0)=\int_{\{\Psi<-t_0\}}|F_{t_0}|^2e^{-\varphi}c(-\Psi)$. Equality \eqref{necessary condition for linear of G 1} shows that $G(t_1)=\int_{\{\Psi<-t_1\}}|\tilde{F}_{t_1}|^2e^{-\varphi}c(-\Psi)$.

Note that on $\{\Psi\ge -t_0\}$, we have $e^{-\Psi+v_{t_0}(\Psi)}c(-v_{t_0}(\Psi))=c(-\Psi)$. It follows from
\begin{equation}\nonumber
\begin{split}
&\int_{\{\Psi<-t_1\}}|\tilde{F}_{t_1}-\mathbb{I}_{\{\Psi< -t_0\}}F_{t_0}|^2e^{-\varphi}c(-\Psi)\\
= &
\int_{\{\Psi<-t_1\}}|\tilde{F}_{t_1}-\mathbb{I}_{\{\Psi< -t_0\}}F_{t_0}|^2e^{-\varphi}e^{-\Psi+v_{t_0}(\Psi)}c(-v_{t_0}(\Psi))\\
\end{split}
\end{equation}
that we have (note that $v_{t_0}(\Psi)=-t_0$ on  $\{\Psi< -t_0\}$)
\begin{equation}
\begin{split}
\label{necessary condition for linear of G 2}
&\int_{\{\Psi<-t_0\}}|\tilde{F}_{t_1}-F_{t_0}|^2e^{-\varphi}c(-\Psi)\\
= &
\int_{\{\Psi<-t_0\}}|\tilde{F}_{t_1}-F_{t_0}|^2e^{-\varphi}e^{-\Psi-t_0}c(t_0).\\
\end{split}
\end{equation}
As $\int_{T}^{+\infty}c(t)e^{-t}dt<+\infty$ and $c(t)e^{-t}$ is decreasing with respect to $t$, we know that there exists $t_2>t_0$ such that $c(t)e^{-t}<c(t_0)e^{-t_0}-\epsilon$ for any $t\ge t_2$, where $\epsilon>0$ is a constant. Then equality \eqref{necessary condition for linear of G 2} implies that
\begin{equation}
\begin{split}
\label{necessary condition for linear of G 3}
&\epsilon\int_{\{\Psi<-t_2\}}|\tilde{F}_{t_1}-F_{t_0}|^2e^{-\varphi-\Psi}\\
\le &
\int_{\{\Psi<-t_2\}}|\tilde{F}_{t_1}-F_{t_0}|^2e^{-\varphi}(e^{-\Psi-t_0}c(t_0)-c(-\Psi))\\
\le &
\int_{\{\Psi<-t_0\}}|\tilde{F}_{t_1}-F_{t_0}|^2e^{-\varphi}(e^{-\Psi-t_0}c(t_0)-c(-\Psi))\\
= &0.
\end{split}
\end{equation}
Note that $e^{-\varphi-\Psi}\ge e^{-(\varphi+\psi)}|F|^2$, $\varphi+\psi$ is a plurisubharmonic function and the integrand in \eqref{necessary condition for linear of G 3} is nonnegative, we must have $\tilde{F}_{t_1}|_{\{\Psi<-t_0\}}=F_{t_0}$.

It follows from Lemma \ref{existence of F} that for any $t>T$, there exists a unique holomorphic $(n,0)$ form $F_t$ on $\{\Psi<-t\}$ satisfying
$$\ \int_{\{\Psi<-t\}}|F_t|^2e^{-\varphi}c(-\Psi)=G(t)$$  and
$\ (F_t-f)\in
\mathcal{O} (K_M)_{z_0} \otimes J_{z_0}$, for any  $ z_0\in Z_0$.  By the above discussion, we know $F_{t}=F_{t'}$ on $\{\Psi<-\max{\{t,t'\}}\}$ for any $t\in(T,+\infty)$ and $t'\in(T,+\infty)$. Hence combining $\lim_{t\rightarrow T+0}G(t)=G(T)$, we obtain that there  exists a unique holomorphic $(n,0)$ form $\tilde{F}$ on $\{\Psi<-T\}$ satisfying $(\tilde{F}-f)_{z_0}\in
\mathcal{O} (K_M)_{z_0} \otimes J_{z_0}$ for any  $z_0\in Z_0$ and $G(t)=\int_{\{\Psi<-t\}}|\tilde{F}|^2e^{-\varphi}c(-\Psi)$ for any $t\ge T$.

Secondly, we prove equality \eqref{other a also linear}.

As $a(t)$ is a nonnegative measurable function on $(T,+\infty)$, then there exists a sequence of functions $\{\sum\limits_{j=1}^{n_i}a_{ij}\mathbb{I}_{E_{ij}}\}_{i\in\mathbb{N}^+}$ $(n_i<+\infty$ for any $i\in\mathbb{N}^+)$ satisfying that $\sum\limits_{j=1}^{n_i}a_{ij}\mathbb{I}_{E_{ij}}$ is increasing with respect to $i$ and $\lim\limits_{i\to +\infty}\sum\limits_{j=1}^{n_i}a_{ij}\mathbb{I}_{E_{ij}}=a(t)$ for any $t\in(T,+\infty)$, where $E_{ij}$ is a Lebesgue measurable subset of $(T,+\infty)$ and $a_{ij}\ge 0$ is a constant for any $i,j$.  It follows from Levi's Theorem that it suffices to prove the case that $a(t)=\mathbb{I}_{E}(t)$, where $E\subset\subset (T,+\infty)$ is a Lebesgue measurable set.

Note that $G(t)=\int_{\{\Psi<-t\}}|\tilde{F}|^2e^{-\varphi}c(-\Psi)=\frac{G(T_1)}{\int_{T_1}^{+\infty}c(s)e^{-s}ds}
\int_{t}^{+\infty}c(s)e^{-s}ds$ where $T_1 \in (T,+\infty)$, then
  \begin{equation}\label{linear 3.4}
\int_{\{-t_1\le\Psi<-t_2\}}|\tilde{F}|^2e^{-\varphi}c(-\Psi)=\frac{G(T_1)}{\int_{T_1}^{+\infty}c(s)e^{-s}ds}
\int_{t_2}^{t_1}c(s)e^{-s}ds
  \end{equation}
  holds for any $T\le t_2<t_1<+\infty$. It follows from the dominated convergence theorem and equality \eqref{linear 3.4} that
    \begin{equation}\label{linear 3.5}
\int_{\{z\in M:-\Psi(z)\in N\}}|\tilde{F}|^2e^{-\varphi}=0
  \end{equation}
  holds for any $N\subset\subset (T,+\infty)$ such that $\mu(N)=0$, where $\mu$ is the Lebesgue measure on $\mathbb{R}$.

  As $c(t)e^{-t}$ is decreasing on $(T,+\infty)$, there are at most countable points denoted by $\{s_j\}_{j\in \mathbb{N}^+}$ such that $c(t)$ is not continuous at $s_j$. Then there is a decreasing sequence of open sets $\{U_k\}$,  such that
$\{s_j\}_{j\in \mathbb{N}^+}\subset U_k\subset (T,+\infty)$ for any $k$, and $\lim\limits_{k \to +\infty}\mu(U_k)=0$. Choosing any closed interval $[t'_2,t'_1]\subset (T,+\infty)$, then we have
\begin{equation}\label{linear 3.6}
\begin{split}
&\int_{\{-t'_1\le\Psi<-t'_2\}}|\tilde{F}|^2e^{-\varphi}\\
=&\int_{\{z\in M:-\Psi(z)\in(t'_2,t'_1]\backslash U_k\}}|\tilde{F}|^2e^{-\varphi}+
\int_{\{z\in M:-\Psi(z)\in[t'_2,t'_1]\cap U_k\}}|\tilde{F}|^2e^{-\varphi}\\
=&\lim_{n\to+\infty}\sum_{i=0}^{n-1}\int_{\{z\in M:-\Psi(z)\in I_{i,n}\backslash U_k\}}|\tilde{F}|^2e^{-\varphi}+
\int_{\{z\in M:-\Psi(z)\in[t'_2,t'_1]\cap U_k\}}|\tilde{F}|^2e^{-\varphi},
\end{split}
\end{equation}
where $I_{i,n}=(t'_1-(i+1)\alpha_n,t'_1-i\alpha_n]$ and $\alpha_n=\frac{t'_1-t'_2}{n}$. Note that
\begin{equation}\label{linear 3.7}
\begin{split}
&\lim_{n\to+\infty}\sum_{i=0}^{n-1}\int_{\{z\in M:-\Psi(z)\in I_{i,n}\backslash U_k\}}|\tilde{F}|^2e^{-\varphi}\\
\le&\limsup_{n\to+\infty}\sum_{i=0}^{n-1}\frac{1}{\inf_{I_{i,n}\backslash U_k}c(t)}\int_{\{z\in M:-\Psi(z)\in I_{i,n}\backslash U_k\}}|\tilde F|^2e^{-\varphi}c(-\Psi).
\end{split}
\end{equation}

It follows from equality \eqref{linear 3.4} that inequality \eqref{linear 3.7} becomes
\begin{equation}\label{linear 3.8}
\begin{split}
&\lim_{n\to+\infty}\sum_{i=0}^{n-1}\int_{\{z\in M:-\Psi(z)\in I_{i,n}\backslash U_k\}}|\tilde F|^2e^{-\varphi}\\
\le&\frac{G(T_1)}{\int_{T_1}^{+\infty}c(s)e^{-s}ds}
\limsup_{n\to+\infty}\sum_{i=0}^{n-1}\frac{1}{\inf_{I_{i,n}\backslash U_k}c(t)}\int_{I_{i,n}\backslash U_k}c(s)e^{-s}ds.
\end{split}
\end{equation}
It is clear that $c(t)$ is uniformly continuous and has positive lower bound and upper bound on $[t'_2,t'_1]\backslash U_k$. Then we have
\begin{equation}\label{linear 3.9}
\begin{split}
&\limsup_{n\to+\infty}\sum_{i=0}^{n-1}\frac{1}{\inf_{I_{i,n}\backslash U_k}c(t)}\int_{I_{i,n}\backslash U_k}c(s)e^{-s}ds \\
\le&\limsup_{n\to+\infty}\sum_{i=0}^{n-1}\frac{\sup_{I_{i,n}\backslash U_k}c(t)}{\inf_{I_{i,n}\backslash U_k}c(t)}\int_{I_{i,n}\backslash U_k}e^{-s}ds\\
=&\int_{(t'_2,t'_1]\backslash U_k}e^{-s}ds.
\end{split}
\end{equation}

Combining inequality \eqref{linear 3.6}, \eqref{linear 3.8} and \eqref{linear 3.9}, we have
\begin{equation}\label{linear 3.10}
\begin{split}
&\int_{\{-t'_1\le\Psi<-t'_2\}}|\tilde{F}|^2e^{-\varphi}\\
=&\int_{\{z\in M:-\Psi(z)\in(t'_2,t'_1]\backslash U_k\}}|\tilde{F}|^2e^{-\varphi}+
\int_{\{z\in M:-\Psi(z)\in[t'_2,t'_1]\cap U_k\}}|\tilde{F}|^2e^{-\varphi}\\
\le&\frac{G(T_1)}{\int_{T_1}^{+\infty}c(s)e^{-s}ds}\int_{(t'_2,t'_1]\backslash U_k}e^{-s}ds+
\int_{\{z\in M:-\Psi(z)\in[t'_2,t'_1]\cap U_k\}}|\tilde{F}|^2e^{-\varphi}.
\end{split}
\end{equation}
Let $k\to +\infty$, following from equality \eqref{linear 3.5} and inequality \eqref{linear 3.10}, then we obtain that
\begin{equation}\label{linear 3.11}
\begin{split}
\int_{\{-t'_1\le\Psi<-t'_2\}}|\tilde{F}|^2e^{-\varphi}
\le\frac{G(T_1)}{\int_{T_1}^{+\infty}c(s)e^{-s}ds}\int_{t'_2}^{t'_1}e^{-s}ds.
\end{split}
\end{equation}
Following from a similar discussion we can obtain that
\begin{equation}\nonumber
\begin{split}
\int_{\{-t'_1\le\Psi<-t'_2\}}|\tilde{F}|^2e^{-\varphi}
\ge\frac{G(T_1)}{\int_{T_1}^{+\infty}c(s)e^{-s}ds}\int_{t'_2}^{t'_1}e^{-s}ds.
\end{split}
\end{equation}
Then combining inequality \eqref{linear 3.11}, we know
\begin{equation}\label{linear 3.12}
\begin{split}
\int_{\{-t'_1\le\Psi<-t'_2\}}|\tilde{F}|^2e^{-\varphi}
=\frac{G(T_1)}{\int_{T_1}^{+\infty}c(s)e^{-s}ds}\int_{t'_2}^{t'_1}e^{-s}ds.
\end{split}
\end{equation}
Then it is clear that for any open set $U\subset (T,+\infty)$ and compact set $V\subset (T,+\infty)$,
$$
\int_{\{z\in M;-\Psi(z)\in U\}}|\tilde{F}|^2e^{-\varphi}
=\frac{G(T_1)}{\int_{T_1}^{+\infty}c(s)e^{-s}ds}\int_{U}e^{-s}ds,
$$
and
$$
\int_{\{z\in M;-\Psi(z)\in V\}}|\tilde{F}|^2e^{-\varphi}
=\frac{G(T_1)}{\int_{T_1}^{+\infty}c(s)e^{-s}ds}\int_{V}e^{-s}ds.
$$
As $E\subset\subset (T,+\infty)$, then $E\cap(t_2,t_1]$ is a Lebesgue measurable subset of $(T+\frac{1}{n},n)$ for some large $n$, where $T\le t_2<t_1\le+\infty$. Then there exists a sequence of compact sets $\{V_j\}$ and a sequence of open subsets $\{V'_j\}$ satisfying $V_1\subset \ldots \subset V_j\subset V_{j+1}\subset\ldots \subset E\cap(t_2,t_1]\subset \ldots \subset V'_{j+1}\subset V'_j\subset \ldots\subset V'_1\subset\subset (T,+\infty)$ and $\lim\limits_{j\to +\infty}\mu(V'_j-V_j)=0$, where $\mu$ is the Lebesgue measure on $\mathbb{R}$. Then we have
\begin{equation}\nonumber
\begin{split}
\int_{\{-t'_1\le\Psi<-t'_2\}}|\tilde{F}|^2e^{-\varphi}\mathbb{I}_E(-\Psi)
=&\int_{z\in M:-\Psi(z)\in E\cap (t_2,t_1]}|\tilde{F}|^2e^{-\varphi}\\
\le&\liminf_{j\to+\infty}\int_{\{z\in M:-\Psi(z)\in V'_j\}}|\tilde{F}|^2e^{-\varphi}\\
\le&\liminf_{j\to+\infty}\frac{G(T_1)}{\int_{T_1}^{+\infty}c(s)e^{-s}ds}\int_{V'_j}e^{-s}ds\\
\le&\frac{G(T_1)}{\int_{T_1}^{+\infty}c(s)e^{-s}ds}\int_{E\cap(t_2,t_1]}e^{-s}ds\\
=&\frac{G(T_1)}{\int_{T_1}^{+\infty}c(s)e^{-s}ds}\int_{t_2}^{t_1}e^{-s}\mathbb{I}_E(s)ds,
\end{split}
\end{equation}
and
\begin{equation}\nonumber
\begin{split}
\int_{\{-t'_1\le\Psi<-t'_2\}}|\tilde{F}|^2e^{-\varphi}\mathbb{I}_E(-\Psi)
\ge&\liminf_{j\to+\infty}\int_{\{z\in M:-\Psi(z)\in V_j\}}|\tilde{F}|^2e^{-\varphi}\\
\ge&\liminf_{j\to+\infty}\frac{G(T_1)}{\int_{T_1}^{+\infty}c(s)e^{-s}ds}\int_{V_j}e^{-s}ds\\
=&\frac{G(T_1)}{\int_{T_1}^{+\infty}c(s)e^{-s}ds}\int_{t_2}^{t_1}e^{-s}\mathbb{I}_E(s)ds,
\end{split}
\end{equation}
which implies that
$$\int_{\{-t'_1\le\Psi<-t'_2\}}|\tilde{F}|^2e^{-\varphi}\mathbb{I}_E(-\Psi)=
\frac{G(T_1)}{\int_{T_1}^{+\infty}c(s)e^{-s}ds}\int_{t_2}^{t_1}e^{-s}\mathbb{I}_E(s)ds.$$
Hence we know that equality \eqref{other a also linear} holds.

Corollary \ref{necessary condition for linear of G} is proved.
\end{proof}

Now we prove Remark \ref{rem:linear}.

\begin{proof}[Proof of Remark \ref{rem:linear}]
By the definition of $G(t;\tilde{c})$, we have $G(t_0;\tilde{c})\le\int_{\{\Psi<-t_0\}}|\tilde{F}|^2e^{-\varphi}\tilde{c}(-\Psi)$, where $\tilde{F}$ is the holomorphic $(n,0)$ form on $\{\Psi<-T\}$ such that $G(t)=\int_{\{\Psi<-t\}}|\tilde{F}|^2e^{-\varphi}c(-\Psi)$ for any $t\ge T$.  Hence we only consider the case $G(t_0;\tilde{c})<+\infty$.

By the definition of $G(t;\tilde{c})$, we can choose a holomorphic $(n,0)$ form $F_{t_0,\tilde{c}}$ on $\{\Psi<-t_0\}$ satisfying $\ (F_{t_0,\tilde{c}}-f)_{z_0}\in
\mathcal{O} (K_M)_{z_0} \otimes J_{z_0}$, for any  $ z_0\in Z_0$ and $\int_{ \{ \Psi<-t_0\}}|F_{t_0,\tilde{c}}|^2e^{-\varphi}\tilde{c}(-\Psi)<+\infty$. As $\mathcal{H}^2(\tilde{c},t_0)\subset \mathcal{H}^2(c,t_0)$, we have $\int_{ \{ \Psi<-t_0\}}|F_{t_0,\tilde{c}}|^2e^{-\varphi}c(-\Psi)<+\infty$. Using
Lemma \ref{existence of F}, we obtain that
\begin{equation}\nonumber
\begin{split}
\int_{ \{ \Psi<-t\}}|F_{t_0,\tilde{c}}|^2e^{-\varphi}c(-\Psi)
=&\int_{ \{ \Psi<-t\}}|\tilde{F}|^2e^{-\varphi}c(-\Psi)\\
+&\int_{ \{ \Psi<-t\}}|F_{t_0,\tilde{c}}-\tilde{F}|^2e^{-\varphi}c(-\Psi)
\end{split}
\end{equation}
for any $t\ge t_0,$ then
\begin{equation}\label{linear 3.13}
\begin{split}
\int_{ \{-t_3\le \Psi<-t_4\}}|F_{t_0,\tilde{c}}|^2e^{-\varphi}c(-\Psi)
=&\int_{ \{-t_3\le \Psi<-t_4\}}|\tilde{F}|^2e^{-\varphi}c(-\Psi)\\
+&\int_{\{-t_3\le \Psi<-t_4\}}|F_{t_0,\tilde{c}}-\tilde{F}|^2e^{-\varphi}c(-\Psi)
\end{split}
\end{equation}
holds for any $t_3>t_4\ge t_0$. It follows from the dominated convergence theorem, equality \eqref{linear 3.13}, \eqref{linear 3.5} and $c(t)>0$ for any $t>T$, that
\begin{equation}\label{linear 3.14}
\begin{split}
\int_{ \{z\in M:-\Psi(z)=t\}}|F_{t_0,\tilde{c}}|^2e^{-\varphi}
=\int_{\{z\in M:-\Psi(z)=t\}}|F_{t_0,\tilde{c}}-\tilde{F}|^2e^{-\varphi}
\end{split}
\end{equation}
holds for any $t>t_0$.

Choosing any closed interval $[t'_4,t'_3]\subset (t_0,+\infty)\subset (T,+\infty)$. Note that $c(t)$ is uniformly continuous and have positive lower bound and upper bound on $[t'_4,t'_3]\backslash U_k$, where $\{U_k\}$ is the decreasing sequence of open subsets of $(T,+\infty)$, such that $c$ is continuous on $(T,+\infty)\backslash U_k$ and $\lim\limits_{k \to +\infty}\mu(U_k)=0$. Take $N=\cap_{k=1}^{+\infty}U_k.$ Note that
\begin{equation}\label{linear 3.15}
\begin{split}
&\int_{ \{-t'_3\le\Psi<-t'_4\}}|F_{t_0,\tilde{c}}|^2e^{-\varphi}\\
=&\lim_{n\to+\infty}\sum_{i=0}^{n-1}\int_{\{z\in M:-\Psi(z)\in S_{i,n}\backslash U_k\}}|F_{t_0,\tilde{c}}|^2e^{-\varphi}
+\int_{\{z\in M:-\Psi(z)\in(t'_4,t'_3]\cap U_k\}}|F_{t_0,\tilde{c}}|^2e^{-\varphi}\\
\le&\limsup_{n\to+\infty}\sum_{i=0}^{n-1}\frac{1}{\inf_{S_{i,n}}c(t)}\int_{\{z\in M:-\Psi(z)\in S_{i,n}\backslash U_k\}}|F_{t_0,\tilde{c}}|^2e^{-\varphi}c(-\Psi)\\
&+\int_{\{z\in M:-\Psi(z)\in(t'_4,t'_3]\cap U_k\}}|F_{t_0,\tilde{c}}|^2e^{-\varphi},
\end{split}
\end{equation}
where $S_{i,n}=(t'_4-(i+1)\alpha_n,t'_3-i\alpha_n]$ and $\alpha_n=\frac{t'_3-t'_4}{n}$.
It follows from equality \eqref{linear 3.13},\eqref{linear 3.14}, \eqref{linear 3.5} and the dominated convergence theorem that
\begin{equation}\label{linear 3.16}
\begin{split}
&\int_{\{z\in M:-\Psi(z)\in S_{i,n}\backslash U_k\}}|F_{t_0,\tilde{c}}|^2e^{-\varphi}c(-\Psi)\\
=&\int_{\{z\in M:-\Psi(z)\in S_{i,n}\backslash U_k\}}|\tilde F|^2e^{-\varphi}c(-\Psi)
+\int_{\{z\in M:-\Psi(z)\in S_{i,n}\backslash U_k\}}|F_{t_0,\tilde{c}}-\tilde F|^2e^{-\varphi}c(-\Psi).
\end{split}
\end{equation}
As $c(t)$ is uniformly continuous and have positive lower bound and upper bound on $[t'_3,t'_4]\backslash U_k$, combining equality \eqref{linear 3.16}, we have
\begin{equation}\label{linear 3.17}
\begin{split}
&\limsup_{n\to+\infty}\sum_{i=0}^{n-1}\frac{1}{\inf_{S_{i,n}\backslash U_k}c(t)}\int_{\{z\in M:-\Psi(z)\in S_{i,n}\backslash U_k\}}|F_{t_0,\tilde{c}}|^2e^{-\varphi}c(-\Psi)\\
=&\limsup_{n\to+\infty}\sum_{i=0}^{n-1}\frac{1}{\inf_{S_{i,n}\backslash U_k}c(t)}(\int_{\{z\in M:-\Psi(z)\in S_{i,n}\backslash U_k\}}|\tilde F|^2e^{-\varphi}c(-\Psi)\\
&+\int_{\{z\in M:-\Psi(z)\in S_{i,n}\backslash U_k\}}|F_{t_0,\tilde{c}}-\tilde F|^2e^{-\varphi}c(-\Psi))\\
\le & \limsup_{n\to+\infty}\sum_{i=0}^{n-1}\frac{\sup_{S_{i,n}\backslash U_k}c(t)}{\inf_{S_{i,n}\backslash U_k}c(t)}(\int_{\{z\in M:-\Psi(z)\in S_{i,n}\backslash U_k\}}|\tilde F|^2e^{-\varphi}\\
&+\int_{\{z\in M:-\Psi(z)\in S_{i,n}\backslash U_k\}}|F_{t_0,\tilde{c}}-\tilde F|^2e^{-\varphi})\\
=&\int_{\{z\in M:-\Psi(z)\in (t'_4,t'_3]\backslash U_k\}}|\tilde F|^2e^{-\varphi}
+\int_{\{z\in M:-\Psi(z)\in (t'_4,t'_3]\backslash U_k\}}|F_{t_0,\tilde{c}}-\tilde F|^2e^{-\varphi}.
\end{split}
\end{equation}
If follows from inequality \eqref{linear 3.15} and \eqref{linear 3.17} that
\begin{equation}\label{linear 3.18}
\begin{split}
&\int_{ \{-t'_3\le\Psi<-t'_4\}}|F_{t_0,\tilde{c}}|^2e^{-\varphi}\\
\le & \int_{\{z\in M:-\Psi(z)\in (t'_4,t'_3]\backslash U_k\}}|\tilde F|^2e^{-\varphi}
+\int_{\{z\in M:-\Psi(z)\in (t'_4,t'_3]\backslash U_k\}}|F_{t_0,\tilde{c}}-\tilde F|^2e^{-\varphi}\\
&+\int_{ \{z\in M: -\Psi(z)\in(t'_4,t'_3]\cap U_k\}}|F_{t_0,\tilde{c}}|^2e^{-\varphi}.
\end{split}
\end{equation}
It follows from $F_{t_0,\tilde{c}}\in\mathcal{H}^2(c,t_0)$ that $\int_{ \{-t'_3\le\Psi<-t'_4\}}|F_{t_0,\tilde{c}}|^2e^{-\varphi}<+\infty$. Let $k\to+\infty,$ by equality \eqref{linear 3.5}, inequality \eqref{linear 3.18} and the dominated theorem, we have
\begin{equation}\label{linear 3.19}
\begin{split}
&\int_{ \{-t'_3\le\Psi<-t'_4\}}|F_{t_0,\tilde{c}}|^2e^{-\varphi}\\
\le & \int_{\{z\in M:-\Psi(z)\in (t'_4,t'_3]\}}|\tilde F|^2e^{-\varphi}
+\int_{\{z\in M:-\Psi(z)\in (t'_4,t'_3]\backslash N\}}|F_{t_0,\tilde{c}}-\tilde F|^2e^{-\varphi}\\
&+\int_{ \{z\in M: -\Psi(z)\in(t'_4,t'_3]\cap N\}}|F_{t_0,\tilde{c}}|^2e^{-\varphi}.
\end{split}
\end{equation}
By similar discussion, we also have that
\begin{equation}\nonumber
\begin{split}
&\int_{ \{-t'_3\le\Psi<-t'_4\}}|F_{t_0,\tilde{c}}|^2e^{-\varphi}\\
\ge & \int_{\{z\in M:-\Psi(z)\in (t'_4,t'_3]\}}|\tilde F|^2e^{-\varphi}
+\int_{\{z\in M:-\Psi(z)\in (t'_4,t'_3]\backslash N\}}|F_{t_0,\tilde{c}}-\tilde F|^2e^{-\varphi}\\
&+\int_{ \{z\in M: -\Psi(z)\in(t'_4,t'_3]\cap N\}}|F_{t_0,\tilde{c}}|^2e^{-\varphi}.
\end{split}
\end{equation}
then combining inequality \eqref{linear 3.19}, we have
\begin{equation}\label{linear 3.20}
\begin{split}
&\int_{ \{-t'_3\le\Psi<-t'_4\}}|F_{t_0,\tilde{c}}|^2e^{-\varphi}\\
= & \int_{\{z\in M:-\Psi(z)\in (t'_4,t'_3]\}}|\tilde F|^2e^{-\varphi}
+\int_{\{z\in M:-\Psi(z)\in (t'_4,t'_3]\backslash N\}}|F_{t_0,\tilde{c}}-\tilde F|^2e^{-\varphi}\\
&+\int_{ \{z\in M: -\Psi(z)\in(t'_4,t'_3]\cap N\}}|F_{t_0,\tilde{c}}|^2e^{-\varphi}.
\end{split}
\end{equation}
Using equality \eqref{linear 3.5}, \eqref{linear 3.14} and Levi's Theorem, we have
\begin{equation}\label{linear 3.21}
\begin{split}
&\int_{ \{z\in M:-\Psi(z)\in U\}}|F_{t_0,\tilde{c}}|^2e^{-\varphi}\\
= & \int_{\{z\in M:-\Psi(z)\in U\}}|\tilde F|^2e^{-\varphi}
+\int_{\{z\in M:-\Psi(z)\in U\backslash N\}}|F_{t_0,\tilde{c}}-\tilde F|^2e^{-\varphi}\\
&+\int_{ \{z\in M: -\Psi(z)\in U\cap N\}}|F_{t_0,\tilde{c}}|^2e^{-\varphi}
\end{split}
\end{equation}
holds for any open set $U\subset\subset (t_0,+\infty)$, and
\begin{equation}\label{linear 3.22}
\begin{split}
&\int_{ \{z\in M:-\Psi(z)\in V\}}|F_{t_0,\tilde{c}}|^2e^{-\varphi}\\
= & \int_{\{z\in M:-\Psi(z)\in V\}}|\tilde F|^2e^{-\varphi}
+\int_{\{z\in M:-\Psi(z)\in V\backslash N\}}|F_{t_0,\tilde{c}}-\tilde F|^2e^{-\varphi}\\
&+\int_{ \{z\in M: -\Psi(z)\in V\cap N\}}|F_{t_0,\tilde{c}}|^2e^{-\varphi}
\end{split}
\end{equation}
holds for any compact set $V\subset (t_0,+\infty)$. For any measurable set $E\subset\subset (t_0,+\infty)$, there exists a sequence of compact set $\{V_l\}$, such that $V_l\subset V_{l+1}\subset E$ for any $l$ and $\lim\limits_{l\to +\infty}\mu(V_l)=\mu(E)$, hence by equality \eqref{linear 3.22}, we have
\begin{equation}\label{linear 3.23}
\begin{split}
\int_{ \{\Psi<-t_0\}}|F_{t_0,\tilde{c}}|^2e^{-\varphi}\mathbb{I}_E(-\Psi)
\ge&\lim_{l \to +\infty} \int_{ \{\Psi<-t_0\}}|F_{t_0,\tilde{c}}|^2e^{-\varphi}\mathbb{I}_{V_j}(-\Psi)\\
\ge&\lim_{l \to +\infty} \int_{ \{\Psi<-t_0\}}|\tilde F|^2e^{-\varphi}\mathbb{I}_{V_j}(-\Psi)\\
=& \int_{ \{\Psi<-t_0\}}|\tilde F|^2e^{-\varphi}\mathbb{I}_{V_j}(-\Psi).
\end{split}
\end{equation}
It is clear that for any $t>t_0$, there exists a sequence of functions $\{\sum_{j=1}^{n_i}\mathbb{I}_{E_{i,j}}\}_{i=1}^{+\infty}$ defined on $(t,+\infty)$, satisfying $E_{i,j}\subset\subset (t,+\infty)$, $\sum_{j=1}^{n_{i+1}}\mathbb{I}_{E_{i+1,j}}(s)\ge \sum_{j=1}^{n_{i}}\mathbb{I}_{E_{i,j}}(s)$ and $\lim\limits_{i\to+\infty}\sum_{j=1}^{n_i}\mathbb{I}_{E_{i,j}}(s)=\tilde{c}(s)$ for any $s>t$. Combining Levi's Theorem and inequality \eqref{linear 3.23}, we have
\begin{equation}\label{linear 3.24}
\begin{split}
\int_{ \{\Psi<-t_0\}}|F_{t_0,\tilde{c}}|^2e^{-\varphi}\tilde{c}(-\Psi)
\ge\int_{ \{\Psi<-t_0\}}|\tilde F|^2e^{-\varphi}\tilde{c}(-\Psi).
\end{split}
\end{equation}
By the definition of $G(t_0,\tilde{c})$, we have $G(t_0,\tilde{c})=\int_{ \{\Psi<-t_0\}}|\tilde F|^2e^{-\varphi}\tilde{c}(-\Psi).$ Equality \eqref{other c also linear} is proved.
\end{proof}

\section{A necessary condition for $G(h^{-1}(r))$ is linearity on open Riemann surfaces}\label{sec:n}

Let $\Omega$ be an open Riemann surface, and let $K_{\Omega}$ be the canonical (holomorphic) line bundle on $\Omega$. Let $dV_{\Omega}$ be a continuous volume form on $\Omega$.  Let $\psi$ be a  subharmonic function on $\Omega$, and let $\varphi$ be a Lebesgue measurable function on $\Omega$ such that $\varphi+\psi$ is subharmonic on $\Omega$. Let $F$ be a holomorphic function on $\Omega$. Let $T\in[-\infty,+\infty)$ such that $-T\le\sup\{\psi(z)-2\log|F(z)|:z\in\Omega \,\&\,F(z)\not=0\}$. Denote that
$$\Psi:=\min\{\psi-2\log|F|,-T\}.$$
For any $z\in \Omega$ satisfying $F(z)=0$, we set $\Psi(z)=-T$. Note that $\Psi$ is subharmonic function on $\{\Psi<-T\}$.

 Let $Z_0$ be a   subset of $ \cap_{t>T}\overline{\{\Psi<-t\}}$. Denote that $Z_1:=\{z\in Z_0:v(dd^c(\psi),z)\ge2ord_{z}(F)\}$ and $Z_2:=\{z\in Z_0:v(dd^c(\psi),z)<2ord_{z}(F)\}$. Denote that $Z'_1:=\{z\in Z_0:v(dd^c(\psi),z)>2ord_{z}(F)\}$. Note that $Z'_1\backslash\{\Psi<-t\}\subset\{z\in\Omega:F(z)=0\}$ and $\{\Psi<-t\}\cup Z'_1$ is an open Riemann surface for any $t\ge T$. Assume that $Z'_1$ is a discrete subset of $\{\Psi<-T\}\cup Z'_1$.

Let $c(t)$ be a positive measurable function on $(T,+\infty)$ satisfying $c(t)e^{-t}$ is decreasing on $(0,+\infty)$, $c(t)e^{-t}$ is integrable near $+\infty$,  and  $e^{-\varphi}c(-\Psi)$ has a positive lower bound on $K\cap\{\Psi<-T\}$ for any compact subset $K$ of $\Omega\backslash E$, where $E\subset\{\Psi=-\infty\}$ is a discrete  subset of $\Omega$.

Let $f$ be a holomorphic $(1,0)$ form on $\{\Psi<-t_0\}\cap V$, where $V\supset Z_0$ is an open subset of $\Omega$ and $t_0>T$
is a real number.
Let $J_{z}$ be an $\mathcal{O}_{\Omega,z}$-submodule of $H_{z}$ such that $I(\Psi+\varphi)_{z}\subset J_{z}$,
where $z\in Z_0$ and $H_{z}:=\{h_{z}\in J(\Psi)_{z}:\int_{\{\Psi<-t\}\cap U}|h|^2e^{-\varphi}c(-\Psi)dV_{\Omega}<+\infty$ for some $t>T$ and some neighborhood $U$ of $z\}$.
Denote
\begin{equation*}
\begin{split}
\inf\Bigg\{ \int_{ \{ \Psi<-t\}}|\tilde{f}|^2&e^{-\varphi}c(-\Psi): \tilde{f}\in
H^0(\{\Psi<-t\},\mathcal{O} (K_{\Omega})  ) \\
&\&\, (\tilde{f}-f)_{z}\in
\mathcal{O} (K_{\Omega})_{z} \otimes J_{z}\text{ for any }  z\in Z_0 \Bigg\}
\end{split}
\end{equation*}
by $G(t;c,\Psi,\varphi,J,f)$, where $t\in[T,+\infty)$ and $|f|^2:=\sqrt{-1}^{n^2}f\wedge \bar{f}$ for any $(1,0)$ form $f$.

Note that there exists a subharmonic function $\psi_1$ on $\Omega_1:= \{\Psi<-T\}\cup Z'_1$ such that $\psi_1+2\log|F|=\psi$. Let $K_{\Omega_1}$ be the canonical (holomorphic) line bundle on $\Omega_1$. For any $z\in \Omega_1$, if $F(z)=0$, we know that $e^{-\varphi}c(-\Psi)=e^{-\varphi}c(-\psi_1)$ has a positive lower bound on $V'\backslash\{z\}$, where $V'\Subset \Omega_1$ is a neighborhood of $z$. Combining $e^{-\varphi}c(-\psi_1)\le Ce^{-\varphi-\psi_1}=Ce^{-\varphi-\psi+2\log|F|}$ on $V'$, we have $v(dd^c(\varphi+\psi),z)\ge 2ord_{z_0}(F)$. Hence we have $\varphi+\psi_1$ is a subharmonic function on $\Omega_1$.

For any $z\in Z'_1$, if $F(z)\not=0$, we know that $J_{z}$ and $I(\varphi+\Psi)_z=\mathcal{I}(\varphi+\psi_1)_z$ are ideals of $\mathcal{O}_{\Omega_1,z}$, and we denote $J_{z}$ by $\mathcal{F}_{z}$. For any $z\in Z'_1$, if $F(z)=0$, we know that $e^{-\varphi}c(-\Psi)$ has a positive lower bound on $V'\backslash\{z\}$, where $V'\subset \Omega$ is a neighborhood of $z$, hence $I(\varphi+\Psi)_z\subset J_{z}\subset H_{z}$ shows that for any $h_z\in J(\Psi)_z$, $h_{z}\in I(\varphi+\Psi)_z$ if and only if there exists a holomorphic extension $\tilde h$ of $h$ near $z$ such that $(\tilde h,z)\in \mathcal{I}(\varphi+\psi_1)_z$, and we get that there exists an ideal $\mathcal{F}_{z}$ of $\mathcal{O}_{\Omega_1,z}$ such that for any $h_z\in J(\Psi)_z$, $h_{z}\in J_{z}$ if and only if there exists a holomorphic extension $\tilde h$ of $h$ near $z$ such that  $(\tilde h,z)\in\mathcal{F}_{z}$.

Let $f_1$ be a holomorphic $(1,0)$ form on $V_1$, where $V_1\supset Z'_1$ is an open subset of $\Omega_1$. Denote
\begin{equation*}
\begin{split}
\inf\Bigg\{ \int_{ \{ \psi_1<-t\}}|\tilde{f}|^2e^{-\varphi}c(-\psi_1): \tilde{f}\in
H^0(\{\psi_1<-t\},\mathcal{O} (K_{\Omega_1})  ) \\
\&\, (\tilde{f}-f_1,z)\in
\mathcal{O} (K_{\Omega_1})_{z} \otimes \mathcal{F}_{z},\text{ for any }  z\in Z'_1 \Bigg\}
\end{split}
\end{equation*}
by $G_{f_1}(t)$, where $t\in[T,+\infty)$.

Note that $e^{-\varphi}c(-\psi_1)$ has a positive lower bound on any compact subset of $\Omega_1\backslash(E\cup Z'_1)$. Let $f_2$ be a holomorphic $(1,0)$ form on $V\cap(\{\Psi<-t_0\}\cup Z'_1)=V\cap\{\psi_1<-t_0\}$ satisfying that $(f_2)_z\in
\mathcal{O} (K_{\Omega})_{z} \otimes H_{z}$ for any $z\in Z_0$, where $V\supset Z_0$ is an open subset of $\Omega$ and $t_0>T$
is a real number. Theorem \ref{main theorem} shows that
$G(h^{-1}(r);c,\Psi,\varphi,J,f_2)$ and $G_{f_2}(h^{-1}(r))$ are concave with respect to $r$, where  $h(t)=\int_{t}^{+\infty}c(s)e^{-s}ds$. We give a relationship between $G(h^{-1}(r);c,\Psi,\varphi,J,f_2)$ and $G_{f_2}(h^{-1}(r))$, which will be used in the proof of Proposition \ref{p:n-linearity1}.

\begin{Lemma}\label{l:inner}
	If $H_z=I(\varphi+\Psi)_z$ for any $z\in Z_0\backslash Z'_1$, then $G(t;c,\Psi,\varphi,J,f_2)=\tilde G_{f_2}(t)$ for any $t\ge T$.
\end{Lemma}
\begin{proof}
	$H_z=I(\varphi+\Psi)_z$ for any $z\in Z_0\backslash Z'_1$ shows that $J_{z_0}=H_z=I(\varphi+\Psi)_z$ for any $z\in Z_0\backslash Z'_1$. For any $t\ge T$ and holomorphic $(1,0)$ form $\tilde f$ on $\{\psi_1<-t\}$ satisfying $(\tilde f-f_2,z)\in
\mathcal{O} (K_{\Omega_1})_{z} \otimes \mathcal{F}_{z}$ for any $z\in Z'_1$ and $\int_{\{\psi_1<-t\}}|\tilde f|^2e^{-\varphi}c(-\psi_1)<+\infty$, it follows from $(f_2)_z\in
\mathcal{O} (K_{\Omega})_{z} \otimes H_{z}$ for any $z\in Z_0$ and the definition of $H_z$ that $(\tilde f-f_2)_z\in
\mathcal{O} (K_{\Omega})_{z} \otimes J_{z}$ for any $z\in Z_0$.
As $\mu(Z'_1)=0$, where $\mu$ is the Lebesgue measure on $\Omega$,   the definitions of $G(t;c,\Psi,\varphi,J,f_2)$ and $G_{f_2}(t)$ show that $G(t;c,\Psi,\varphi,J,f_2)\le\tilde G_{f_2}(t)$ for any $t\ge T$.
	
	For any $t\ge T$, let  $\tilde f$ be a holomorphic $(1,0)$ form on $\{\Psi<-t\}$ satisfying $(\tilde f-f_2)_{z}\in\mathcal{O}(K_{\Omega})_z\otimes J_z$ for any $z\in Z_0$ and $\int_{\{\Psi<-t\}}|\tilde f|^2e^{-\varphi}c(-\Psi)<+\infty$. For any $z_0\in Z'_1\backslash\{\Psi<-t\}$, we have $F(z_0)=0$. Note that $v(dd^c(\psi),z_0)>2ord_{z_0}(F)$, then $e^{-\varphi}c(-\Psi)$ has a positive lower bound on $V'\backslash\{z_0\}\subset\{\Psi<-t\}$, where $V'\Subset\Omega_1$ is a neighborhood  of $z_0$. Following from $\int_{\{\Psi<-t\}}|\tilde f|^2e^{-\varphi}c(-\Psi)<+\infty$, we get that there exists a holomorphic $(1,0)$ form $\tilde f_1$ on $\{\psi_1<-t\}=\{\Psi<-t\}\cup Z'_1$ such that $\tilde f_1=\tilde f$ on $\{\Psi<-t\}$, which implies that $(\tilde f_1-f_2,z)\in\mathcal{O}(K_{\Omega})_z\otimes \mathcal{F}_z$ for any $z\in Z'_1$ and $\int_{\{\psi_1<-t\}}|\tilde f_1|^2e^{-\varphi}c(-\psi_1)=\int_{\{\Psi<-t\}}|\tilde f|^2e^{-\varphi}c(-\Psi)$. By the definitions of $G(t;c,\Psi,\varphi,J,f_2)$ and $G_{f_2}(t)$, we have  $G(t;c,\Psi,\varphi,J,f_2)\ge\tilde G_{f_2}(t)$ for any $t\ge T$.
	
	Thus, Lemma \ref{l:inner} holds.
\end{proof}

We give a necessary condition for the concavity of $G(h^{-1}(r))$ degenerating to linearity.
\begin{Proposition}
	\label{p:n-linearity1}
For any $z\in Z_1$,	assume that one of the following conditions holds:
	
	$(A)$ $\varphi+a\psi$ is  subharmonic near $z$ for some $a\in[0,1)$;
	
	$(B)$ $(\psi-2p_z\log|w|)(z)>-\infty$, where $p_z=\frac{1}{2}v(dd^c(\psi),z)$ and $w$ is a local coordinate on a neighborhood of $z$ satisfying that $w(z)=0$.
	
 If there exists $t_1\ge T$ such that $G(t_1)\in(0,+\infty)$ and	$G(h^{-1}(r))$ is linear with respect to $r\in(0,\int_T^{+\infty}c(s)e^{-s}ds)$, then the following statements hold:

	$(1)$ $\varphi+\psi=2\log|g|+2\log|F|$ on $\{\Psi<-T\}\cup Z'_1$, $J_{z}=I(\varphi+\Psi)_z$ for any $z\in Z'_1$ and there exists a holomorphic $(1,0)$ form $f_1$ on $(\{\Psi<-t_0\}\cap V)\cup Z'_1$  such that $f_1=f$ on $\{\Psi<-t_0\}\cap V$, where $g$ is a holomorphic function  on $\{\Psi<-T\}\cup Z'_1$ such that $ord_z(g)=ord_z(f_1)+1$ for any $z\in Z'_1$;
	
	$(2)$ $\{z\in Z_1:v(dd^c(\psi),z)>2ord_z(F)\}\not=\emptyset$ and $\psi=2\sum_{ z\in Z'_1}\big(p_z-ord_z(F)\big)G_{\Omega_t}(\cdot, z)+2\log|F|-t$ on $\{\Psi<-t\}\cup Z'_1$ for any $t>T$, where  $\Omega_t=\{\Psi<-t\}\cup Z'_1$ and $G_{\Omega_t}(\cdot, z)$ is the Green function on $\Omega_t$;
	
	$(3)$ $\frac{p_z-ord_z(F)}{ord_{z}(g)}\lim_{z'\rightarrow z}\frac{dg(z')}{f(z')}=c_0$ for any $z\in Z'_1$, where $c_0\in\mathbb{C}\backslash\{0\}$ is a constant independent of $z\in Z'_1$.
	
	$(4)$ $\sum_{z\in Z'_1}\big(p_z-ord_z(F)\big)<+\infty$.
\end{Proposition}

\begin{proof}

It follows from Remark \ref{rem:linear} that
 we can assume that $c(t)\ge \frac{e^t}{t^2}$ near $+\infty$.
	
	We prove Proposition \ref{p:n-linearity1} in three steps.

	\
	
	\emph{Step 1. $H_z=I(\varphi+\Psi)_z$ for any $z\in Z_2$.}
	
	\
	
	Fixed any $z_0\in Z_2$, the following remark shows that we can assume that $v(dd^c(\psi),z_0)+v(dd^c(\varphi+\psi),z_0)\not\in\mathbb{Z}$ and $c(t)\ge1$ near $+\infty$.
	\begin{Remark}
		\label{r:not integer}
		Let $a\in(0,1)$. Let $\tilde\varphi=\varphi+a(\psi-2\log|F|)$ and $\tilde\psi=(1-a)\psi+2a\log|F|$. Denote that $\tilde\Psi:=\min\{\tilde\psi-2\log| F|,(1-a)T\}=(1-a)\Psi$.
		Let $\tilde c(t)=c\left(\frac{t}{1-a}\right)e^{-\frac{at}{1-a}}$ be a function on $((1-a)T,+\infty)$, and we have $\tilde c(t)\ge 1$ near $+\infty$. It is clear that $\tilde\psi$ and $\tilde\varphi+\tilde\psi=\varphi+\psi$ are subharmonic functions.
		
		Note that $e^{-\tilde\varphi}\tilde c(-\tilde\Psi)=e^{-\varphi-a(\psi-2\log|F|)}c(-\Psi)e^{a\Psi}=e^{-\varphi}c(-\Psi)$ on $\{\Psi<-T\}$, $\varphi+\Psi=\tilde\varphi+\tilde\Psi$ on $\{\Psi<-t\}$ and $\tilde c(t)e^{-t}=c\left(\frac{t}{1-a}\right)e^{-\frac{t}{1-a}}$.
		As $z_0\in Z_2=\{z\in Z_0:2ord_{z}(F)>v(dd^c(\psi),z)\}$, we can choose $a\in(0,1)$ such that $v(dd^c(\tilde\psi),z_0)+v(dd^c(\tilde\varphi+\tilde\psi),z_0)=v(dd^c(\psi),z_0)+v(dd^c(\varphi+\psi),z_0)+a(2ord_{z_0}(F)-v(dd^c(\psi),z_0))\not\in \mathbb{Z}$.
	\end{Remark}

Denote that $\Psi_1:=\min\{\Psi,T_0\}$ and $\varphi_1:=2\max\{\psi+T_0,2\log|F|\}$, where $T_0>T$. Note that $I(\varphi+\Psi)_{z_0}=I(\varphi+\Psi_1)_{z_0}$ and $\frac{1}{2}\varphi_1+\Psi=\psi$.	It follows from Proposition \ref{module isomorphism} that  $H_{z_0} =I(\varphi+\Psi_1)_{z_0}$ if and only if $\mathcal{H}_{z_0}=\mathcal{I}(\varphi+\Psi_1+\varphi_1)_{z_0}$, where $\mathcal{H}_{z_0}:=\{(h,z_0)\in\mathcal{O}_{\Omega,z_0}:|h|^2e^{-\varphi-\varphi_1}c(-\Psi_1)$ is integrable near $z_0\}$.

Now, we prove $\mathcal{H}_{z_0}=\mathcal{I}(\varphi+\Psi_1+\varphi_1)_{z_0}$. Without loss of generality, we can assume that $\Omega=\Delta$ (the unit disc in $\mathbb{C}$) and $z_0=o$ (the origin). As $c(t)e^{-t}$ is decreasing, we have $\mathcal{I}(\varphi+\Psi_1+\varphi_1)_{o}\subset \mathcal{H}_{o}$. For any $(h,o)\in \mathcal{H}_o$, there exists $r_1>0$ such that $\int_{\Delta_{r_1}}|h|^2e^{-\varphi-\varphi_1}c(-\Psi_1)<+\infty$, which implies that
\begin{equation}
	\label{eq:0316a}\int_{\Delta_{r_1}}|h|^2e^{-\varphi-\varphi_1}<+\infty,
\end{equation}
where $\Delta_{r_1}=\{z\in\mathbb{C}:|z|<r_1\}$.
Denote that $x_1:=v(dd^c(\psi),o)$ and $x_2:=v(dd^c(\varphi+\psi),o)$. It follows from  Siu's Decomposition Theorem that
\begin{equation}
	\label{eq:0316b}\psi=x_1\log|w|+\tilde\psi,
\end{equation}
	where $\tilde\psi$ is a subharmonic function on $\Delta$ satisfying that $v(dd^c(\tilde\psi),o)=0$. As $v(dd^c(\psi),o)<2ord_o(F)$, we have
	\begin{equation}
		\label{eq:0316c}\psi\le\frac{1}{2}\varphi_1=\max\{\psi+T_0,2\log|F|\}\le x_1\log|w|+C_1
	\end{equation}
	near $o$, where $C_1$ is a constant, which implies that $v(dd^c(\frac{1}{2}\varphi_1),o)=x_1$. As $x_2=v(dd^c(\varphi+\psi),o)$, we have
	\begin{equation}
		\label{eq:0316d}\varphi+\psi\le x_2\log|w|+C_2
	\end{equation}
	near $o$, where $C_2$ is a constant. Denote that $k:=ord_o(h)$.
	Combining inequality \eqref{eq:0316a}, equality \eqref{eq:0316b}, inequality \eqref{eq:0316c} and inequality \eqref{eq:0316d}, we get that there exists $r_2\in(0,r_1)$ such that
	\begin{equation}
		\label{eq:0316e}
		\begin{split}
			&\int_{\Delta_{r_2}}|w|^{2k-x_1-x_2}e^{\tilde\psi}\\
			\le &C_3\int_{\Delta_{r_2}}|h|^2e^{-\frac{1}{2}\varphi_1-\varphi-\psi+\tilde\psi}\\
			=&C_3\int_{\Delta_{r_2}}|h|^2e^{-\frac{1}{2}\varphi_1-\varphi-x_1\log|w|}\\
			\le&C_3e^{C_1}\int_{\Delta_{r_2}}|h|^2e^{-\varphi-\varphi_1}\\
			<&+\infty.
		\end{split}
	\end{equation}
	For any $p>1$, as $v(dd^c(\tilde\psi),o)=0$, it follows from lemma \ref{l:skoda} that there exists $r_3\in(0,r_2)$ such that $\int_{\Delta_{r_3}}e^{-\frac{q}{p}\tilde\psi}<+\infty$, where $\frac{1}{p}+\frac{1}{q}=1$.
	It follows from inequality \eqref{eq:0316e} and H\"older inequality that
	\begin{equation}
		\label{eq:0316f}
		\begin{split}
		&\int_{\Delta_{r_3}}|w|^{\frac{2k-x_1-x_2}{p}}\\
		\le&\left(\int_{\Delta_{r_3}}|w|^{2k-x_1-x_2}e^{\tilde\psi}\right)^{\frac{1}{p}}\left(\int_{\Delta_{r_3}}e^{-\frac{q}{p}\tilde\psi}\right)^{\frac{1}{q}}\\
		<&+\infty,
		\end{split}
	\end{equation}
	which shows that $|w|^{\frac{2k-x_1-x_2}{p}}$ is integrable near $o$ for any $p>1$.
	As $x_1+x_2=v(dd^c(\psi),o)+v(dd^c(\varphi+\psi),o)\not\in \mathbb{Z}$,  we have $|w|^{2k-x_1-x_2}$ is integrable near $o$. Note that $v(dd^c(\varphi+\Psi_1+\varphi_1),o)=v(dd^c(\varphi+\psi+\frac{1}{2}\varphi_1),o)=x_1+x_2$. It follows from Lemma \ref{l:1d-MIS} that  $(w^k,o)\in\mathcal{I}(\varphi+\Psi_1+\varphi_1)_o$, which shows that $(h,o)\in\mathcal{I}(\varphi+\Psi_1+\varphi_1)_o$. We obtain that $\mathcal{H}_{o}=\mathcal{I}(\varphi+\Psi_1+\varphi_1)_{o}.$
	
	Thus, we have $H_z=I(\varphi+\Psi)_z$ for any $z\in Z_2$.

	\
	
	\emph{Step 2. $H_z=I(\varphi+\Psi)_z$ for any $z\in Z_1\backslash Z'_1$.}
	
	\
	
		Let $z_0\in Z_1\backslash Z'_1$, where $Z'_1=\{z\in Z_0:v(dd^c(\psi),z)>2ord_{z}(F)\}$. We have $v(dd^c(\psi),z_0)=2ord_{z_0}(F)$. Denote that $\Psi_1:=\min\{\Psi,T_0\}$ and $\varphi_1:=2\max\{\psi+T_0,2\log|F|\}$, where $T_0>T$. Note that $I(\varphi+\Psi)_{z_0}=I(\varphi+\Psi_1)_{z_0}$ and $\frac{1}{2}\varphi_1+\Psi=\psi$.	  It follows from Proposition \ref{module isomorphism} that  $H_{z_0} =I(\varphi+\Psi_1)_{z_0}$ if and only if $\mathcal{H}_{z_0}=\mathcal{I}(\varphi+\Psi_1+\varphi_1)_{z_0}$, where $\mathcal{H}_{z_0}:=\{(h,z_0)\in\mathcal{O}_{\Omega,z_0}:|h|^2e^{-\varphi-\varphi_1}c(-\Psi_1)$ is integrable near $z_0\}$. As $c(t)e^{-t}$ is decreasing, we have $\mathcal{I}(\varphi+\Psi_1+\varphi_1)_{z_0}\subset \mathcal{H}_{z_0}$.
Thus, it suffices to prove $\mathcal{H}_{z_0}\subset \mathcal{I}(\varphi+\Psi_1+\varphi_1)_{z_0}$.

Without loss of generality, we can assume that $\Omega=\Delta$ (the unit disc in $\mathbb{C}$) and $z_0=o$ (the origin). Denote that $x_1:=v(dd^c(\psi),o)$ and $x_2:=v(dd^c(\varphi+\psi),o)$.

	 Firstly, we prove $\mathcal{H}_{o}\subset \mathcal{I}(\varphi+\Psi_1+\varphi_1)_{o}$ under  condition $(A)$ ($\varphi+a\psi$ is subharmonic near $o$ for some $a\in[0,1)$).  For any  $(h,o)\in\mathcal{H}_o$, as $c(t)\ge e^{at}$ near $+\infty$, we get that there exists $r_1\in(0,1)$ such that
	 \begin{equation}
	 	\label{eq:0317a}\int_{\Delta_{r_1}}|h|^2e^{-\varphi-a\Psi_1-\varphi_1}\le C\int_{\Delta_{r_1}}|h|^2e^{-\varphi-\varphi_1}c(-\Psi_1)<+\infty.
	 \end{equation}
	Note that $\varphi+a\Psi_1+\varphi_1=\varphi+a\psi+(2-a)\max\{\psi,2\log|F|\}$ and $v(dd^c(\varphi+a\psi),o)=v(dd^c(\varphi+\psi),o)-(1-a)v(dd^c(\psi),o)=x_2-(1-a)x_1$. As $v(dd^c(\psi),o)=2ord_{o}(F)$, we have $v(dd^c(\varphi+a\Psi_1+\varphi_1),o)=x_2-(1-a)x_1+(2-a)x_1=x_2+x_1$. Denote that $k:=ord_o(h)$. Note that $v(dd^c(\varphi+\Psi_1+\varphi_1),o)=x_1+x_2$. It follows from inequality \eqref{eq:0317a} and Lemma \ref{l:1d-MIS} that $(h,o)\in\mathcal{I}(\varphi+\Psi_1+\varphi_1)_{o}$. Thus,  $\mathcal{H}_{o}\subset \mathcal{I}(\varphi+\Psi_1+\varphi_1)_{o}$ holds under condition $(A)$.
	
	Now, we prove $\mathcal{H}_{o}\subset \mathcal{I}(\varphi+\Psi_1+\varphi_1)_{o}$ under condition $(B)$ $\big((\psi-x_1\log|w|)(o)>-\infty\big)$. For any $(h,o)\in \mathcal{H}_o$, there exists $r_2>0$ such that $\int_{\Delta_{r_2}}|h|^2e^{-\varphi-\varphi_1}c(-\Psi_1)<+\infty$, which implies that
\begin{equation}
	\label{eq:0317b}\int_{\Delta_{r_2}}|h|^2e^{-\varphi-\varphi_1}<+\infty.
\end{equation} It follows from  Siu's Decomposition Theorem that
\begin{equation*}
\psi=x_1\log|w|+\tilde\psi,
\end{equation*}
	where $\tilde\psi$ is a subharmonic function on $\Delta$ satisfying that $v(dd^c(\tilde\psi),o)=0$.
	Note that $\tilde\psi(o)>-\infty$ and $\varphi+\varphi_1\le \varphi+2x_1\log|w|+C_1=\varphi+\psi-\tilde\psi+x_1\log|w|\le (x_1+x_2)\log|w|-\tilde\psi+C_2$ near $o$, where $C_1$ and $C_2$ are constants. Denote that $k:=ord_o(h)$. Note that $e^{\tilde\psi}$ is subharmonic, then it follows from inequality \eqref{eq:0317b} and the  sub-mean value inequality that there exists $r_3\in(0,r_2)$ such that
	\begin{equation}
			\label{eq:0317c}
			\begin{split}
			&\int_{\Delta_{r_3}}|w|^{2k-x_1-x_2}\\
			=&2\pi\int_0^{r_3}r^{2k+1-x_1-x_2}dr\\
			\le& \frac{1}{e^{\tilde\psi(o)}}\int_0^{r_3}\int_{0}^{2\pi}r^{2k+1-x_1-x_2}e^{\tilde\psi(re^{i\theta})}d\theta dr\\
			=&\frac{1}{e^{\tilde\psi(o)}}\int_{\Delta_{r_3}}|w|^{2k-x_1-x_2}e^{\tilde\psi}\\
			\le&\frac{C}{e^{\tilde\psi(o)}}\int_{\Delta_{r_3}}|h|^2e^{-\varphi-\varphi_1}\\
			<&+\infty.
			\end{split}
	\end{equation}
	As $v(dd^c(\varphi+\Psi_1+\varphi_1),o)=x_1+x_2$, it follows from inequality \eqref{eq:0317c} and Lemma \ref{l:1d-MIS} that $(h,o)\in\mathcal{I}(\varphi+\Psi_1+\varphi_1)_{o}$. Thus,  $\mathcal{H}_{o}\subset \mathcal{I}(\varphi+\Psi_1+\varphi_1)_{o}$ holds under condition $(B)$.

\

\emph{Step 3. The four statements hold.}

\

It follows from Lemma \ref{characterization of g(t)=0} that there exists a holomorphic $(1,0)$ form $f_{t_1}$ on $\{\Psi<-t_1\}$ such that $(f_{t_1}-f)_{z}\in\mathcal{O}(K_{\Omega})_z\otimes J_z$ for any $z\in Z_0$ and $\int_{\{\Psi<-t_1\}}|f_{t_1}|^2e^{-\varphi}c(-\Psi)<+\infty$, which implies that $(f_{t_1})_{z}\in\mathcal{O}(K_{\Omega})_z\otimes H_z$ for any $z\in Z_0$.

  Note that $\{\Psi<-t\}\cup Z'_1$ is an open Riemann surface for any $t\ge T$ and there exists a subharmonic function $\psi_1$ on $\Omega_1=\{\Psi<-T\}\cup Z'_1$ such that $\psi_1+2\log|F|=\psi$.
  For any $z_0\in Z'_1\backslash\{\Psi<-t_1\}$, we have $F(z_0)=0$. Note that $v(dd^c(\psi),z_0)>2ord_{z_0}(F)$, then $e^{-\varphi}c(-\Psi)$ has a positive lower bound on $V'\backslash\{z_0\}\subset\{\Psi<-t\}$, where $V'\Subset\Omega_1$ is a neighborhood  of $z_0$. Following from $\int_{\{\Psi<-t_1\}}| f_{t_1}|^2e^{-\varphi}c(-\Psi)<+\infty$, we get that there exists a holomorphic $(1,0)$ form $\tilde f_{t_1}$ on $\{\psi_1<-t_1\}=\{\Psi<-t_1\}\cup Z'_1$ such that $\tilde f_{t_1}= f_{t_1}$ on $\{\Psi<-t_1\}$, which implies that $(\tilde f_{t_1}-f,z)\in\mathcal{O}(K_{\Omega})_z\otimes \mathcal{F}_z$ for any $z\in Z'_1$ and $(\tilde f_{t_1}-f)_z\in\mathcal{O}(K_{\Omega})_z\otimes J_z$ for any $z\in Z_0$. By the definition of $G(t;c,\Psi,\varphi,J,f)$, we have  $G(t;c,\Psi,\varphi,J,f)=G(t;c,\Psi,\varphi,J,\tilde f_{t_1})$ for any $t\ge T$. It follows from Lemma \ref{l:inner} and $H_z=I(\varphi+\Psi)_z$ for any $z\in Z_0\backslash Z'_1$ that $G_{\tilde f_{t_1}}(t)=G(t;c,\Psi,\varphi,J,\tilde f_{t_1})$ for any $t\ge T$, which implies that $G_{\tilde f_{t_1}}(h^{-1}(r))$ is linear with respect to $r$.

Denote that $\tilde Z'_1:=\{z\in  Z'_1:v(dd^c(\varphi+\psi_1),z)>0\}$ and denote
\begin{equation*}
\begin{split}
\inf\Bigg\{ \int_{ \{ \psi_1<-t\}}|\tilde{f}|^2&e^{-\varphi}c(-\psi_1): \tilde{f}\in
H^0(\{\psi_1<-t\},\mathcal{O} (K_{\Omega})  ) \\
&\&\, (\tilde{f}-\tilde f_{t_1},z)\in
\mathcal{O} (K_{\Omega})_{z} \otimes \mathcal{F}_z,\text{ for any }  z\in \tilde Z'_1 \Bigg\}
\end{split}
\end{equation*}
by $\tilde G(t)$, where $t\in[T,+\infty)$. For any $z\in Z'_1\backslash \tilde Z'_1$,  it follows from $v(dd^c(\varphi+\psi_1),z)=0$ and $\mathcal{I}(\varphi+\psi_1)_z\subset\mathcal{F}_z$ that $\mathcal{F}_z=\mathcal{O}_{\Omega,z}$. Hence,  we have $\tilde G(h^{-1}(r))=G_{\tilde f_{t_1}}(h^{-1}(r))$ is linear with respect to $r$.

Let $z_0\in\tilde Z'_1$, if $(\psi_1-2p'_{z_0}\log|w|)(z)=-\infty$, where $p'_{z_0}=\frac{1}{2}v(dd^c(\psi_1),z_0)$ and $w$ is a local coordinate on a neighborhood of $z_0$ satisfying that $w(z_0)=0$, we have $(\psi-2p_{z_0}\log|w|)(z_0)=-\infty$ (where $p_{z_0}=\frac{1}{2}v(dd^c(\psi),z_0)$), thus there exists $a\in[0,1)$ such that $\varphi+a\psi$ is subharmonic near $z_0$. As $v(dd^c(\varphi+\psi_1),z_0)=v(dd^c(\varphi+\psi),z_0)-2ord_{z_0}(F)>0$, there exists $a'\in[a,1)$ such that $v(dd^c(\varphi+a'\psi))-2a'ord_{z_0}(F)>0$, which implies that $\varphi+a'\psi_1$ is subharmonic near $z_0$.
As $G(t_1;c,\Psi,\varphi,J,\tilde f_{t_1})=\tilde G(t_1)\in(0,+\infty)$, we have $\tilde Z'_1\not=\emptyset$.  It follows from Proposition \ref{p:inner1}, Remark \ref{r:equivalent} and Proposition \ref{p:inner2} (replace $\Omega$, $\psi$, $Z_0$ and $c(\cdot)$ by $\{\psi_1<-t\}=\{\Psi<-t\}\cup Z'_1$, $\psi_1+t$, $\tilde Z'_1$ and $c(\cdot+t)$ respectively, where $t>T$) that the following statements hold:

$(a)$ $\varphi+\psi_1=2\log|g|$ and $\mathcal{F}_z=\mathcal{I}(\varphi+\psi_1)_z$ for any $z\in\tilde Z'_1$, where $g$ is a holomorphic function on $\Omega_1= \{\Psi<-T\}\cup Z'_1$ such that $ord_z(g)=ord_z(\tilde f_{t_1})+1$ for any $z\in\tilde Z'_1$;

$(b)$ $\psi_1+t=2\sum_{z\in\tilde Z'_1}p'_zG_{\Omega_t}(\cdot,z)$ on $\{\psi_1<-t\}$ for any $t>T$;

$(c)$ $\frac{p'_z}{ord_z(g)}\lim_{z'\rightarrow z}\frac{dg(z')}{f_{\tilde t_1}(z')}=c_0$ for any $z\in\tilde Z'_1$, where $c_0\in\mathbb{C}\backslash\{0\}$ is a constant independent of $z\in\tilde Z'_1$;

$(d)$ $\sum_{z\in\tilde Z'_1}p'_z<+\infty$.

By definition of $\psi_1$, we know that $p'_z=p_z-ord_{z}(F)$ for any $z\in \Omega_1$.
It follows from $(b)$ and $Z'_1=\{z\in Z_0:v(dd^c(\psi),z)>2ord_z(F)\}$  that  $Z'_1=\tilde Z'_1$.
It follows from $\mathcal{F}_z=\mathcal{I}(\varphi+\psi_1)_z$ for any $z\in Z'_1$ and $H_z=I(\varphi+\Psi)_z$ for any $z\in Z_0\backslash Z'_1$ that $J_z=I(\varphi+\Psi)_z$ for any $z\in Z_0$. Note that for any $z\in Z'_1$ and any $h_z\in H_z$, there exists a holomorphic extension $\tilde h$ of $h$ near $z$. As $(f_{t_1}-f)_z\in\mathcal{O}(K_{\Omega})_z\otimes J_z$ for any $z\in Z_0$ and $\tilde f_{t_1}$ is a holomorphic $(1,0)$ form on $\{\psi_1<-t_1\}$ such that $\tilde f_{t_1}=f_{t_1}$ on $\{\Psi<-t_1\}$, then there exists a holomorphic $(1,0)$ form $f_1$ on $(\{\Psi<-t_1\}\cap V)\cup Z'_1$  such that $f_1=f$ on $\{\Psi<-t_1\}\cap V$, which shows that $(\tilde f_{t_1}-f_1,z)\in\mathcal{O}(K_{\Omega})_z\otimes\mathcal{I}(2\log|g|)_z$ for any $z\in Z'_1$.   $ord_z(g)=ord_z(\tilde f_{t_1})+1$  for any $z\in Z'_1$ implies that  $ord_z(g)=ord_z(\tilde f_{1})+1$  for any $z\in Z'_1$.

Thus, the four statements in Proposition \ref{p:n-linearity1} hold.
\end{proof}

\section{Proof of Theorem \ref{thm:linearity1}}

 The necessity of Theorem \ref{thm:linearity1} follows from Proposition \ref{p:n-linearity1}. In the following, we prove the sufficiency.

As $\psi=2\sum_{z\in Z'_1}p_zG_{\Omega_t}(\cdot,z)+2\log|F|-t$ on $\{\Psi<-t\}$, it follows from Lemma \ref{l:G-compact} that $$Z_0=Z_0\cap(\cap_{s>T}\overline{\{\Psi<-s\}})=Z'_1.$$   Note that $\{\Psi<-t\}\cup Z'_1$ is an open Riemann surface for any $t\ge T$ and there exists a subharmonic function $\psi_1$ on $\Omega_1=\{\Psi<-T\}\cup Z'_1$ such that $\psi_1=\psi-2\log|F|=2\sum_{z\in Z'_1}p_zG_{\Omega_t}(\cdot,z)-t$ for any $t>T$. Note that $f_1$ is a holomorphic $(1,0)$ form  on $(\{\Psi<-t_0\}\cap V)\cup Z'_1$ such that $f_1=f$ on $\{\Psi<-t_0\}\cap V$. Following the definition of $G_{f_1}(t)$ in Section \ref{sec:n}, Lemma \ref{l:inner} shows that
$$G_{f_1}(t)=G(t)$$
 for any $t\ge T$. It follows from Theorem \ref{thm:m-points} and Remark \ref{r:equivalent} (replace $\Omega$, $c(\cdot)$ and $\psi$ by $\{\psi_1<-t\}=\{\Psi<-t\}\cup Z'_1$, $c(\cdot+\tilde{t})$ and $\psi_1+\tilde{t}$ respectively, where $\tilde{t}>T$) that $G_{f_1}(h_{\tilde{t}}^{-1}(r)+\tilde{t})$ is linear with respect to $r\in (0,\int_{0}^{+\infty}c(s+\tilde{t})e^{-s}ds)$ for any $\tilde{t}> T$, where $h_{\tilde{t}}(t)=\int_t^{+\infty}c(s+\tilde{t})e^{-s}ds=e^{\tilde{t}}\int_{t+\tilde{t}}^{+\infty}c(s)e^{-s}ds=e^{\tilde{t}}h(t+\tilde{t})$. Note that $G(h^{-1}(r))=G_{f_1}(h^{-1}(r))=G_{f_1}(h_{\tilde{t}}^{-1}(e^{\tilde{t}}r)+\tilde{t})$ for any $r\in(0,\int_{\tilde{t}}^{+\infty}c(s)e^{-s}ds)$. Hence we have $G(h^{-1}(r))$ is linear with respect to $r\in (0,\int_{0}^{+\infty}c(s)e^{-s}ds)$.

Thus, Theorem \ref{thm:linearity1} holds.

%%%------------------------------------------------------------------------

\vspace{.1in} {\em Acknowledgements}.
The first author and the second author were supported by National Key R\&D Program of China 2021YFA1003103. The first author was supported by NSFC-11825101, NSFC-11522101 and NSFC-11431013.

\bibliographystyle{references}
\bibliography{xbib}

\begin{thebibliography}{100}

\bibitem{BGMY7}S.J. Bao, Q.A. Guan, Z.T. Mi and Z. Yuan, Concavity property of minimal $L^2$ integrals with Lebesgue measurable gain \uppercase\expandafter{\romannumeral7}-negligible weights,
https://www.researchgate.net/publication/358215153.

\bibitem{BGY-boundary}S.J. Bao, Q.A. Guan and Z. Yuan, Boundary points, minimal $L^2$ integrals and concavity property, 	arXiv:2203.01648.




\bibitem{cao17}J.Y. Cao, Ohsawa-Takegoshi extension theorem for compact K{\"a}hler manifolds and applications, Complex and symplectic geometry, 19-38, Springer INdAM Ser., 21, Springer, Cham, 2017.
		
\bibitem{cdM17}J.Y. Cao, J-P. Demailly and S. Matsumura, A general extension theorem for cohomology classes on non reduced analytic subspaces, Sci. China Math. 60 (2017), no. 6, 949-962, DOI 10.1007/s11425-017-9066-0.
		

	
\bibitem{DEL18}T. Darvas, E. Di Nezza and H.C. Lu, Monotonicity of nonpluripolar products and complex Monge-Amp{\'e}re equations with prescribed singularity, Anal. PDE 11 (2018), no. 8, 2049-2087.
		
\bibitem{DEL21}T. Darvas, E. Di Nezza and H.C. Lu, The metric geometry of singularity types, J. Reine Angew. Math. 771 (2021), 137-170.






\bibitem{Demaillybook}J.-P Demailly, Complex analytic and differential geometry, electronically accessible at https://www-fourier.ujf-grenoble.fr/\textasciitilde demailly/manuscripts/agbook.pdf.

\bibitem{DemaillySoc}J.-P Demailly, Multiplier ideal sheaves and analytic methods in algebraic geometry, School on Vanishing Theorems and Effective Result in Algebraic Geometry (Trieste,2000),1-148,ICTP lECT.Notes, 6, Abdus Salam Int. Cent. Theoret. Phys., Trieste, 2001.

\bibitem{DemaillyAG}J.-P Demailly, Analytic Methods in Algebraic Geometry, Higher Education Press, Beijing, 2010.
		
\bibitem{DEL}J.-P Demailly, L. Ein and R. Lazarsfeld, A subadditivity property of multiplier ideals, Michigan Math. J. 48 (2000) 137-156.
		
\bibitem{DK01}J.-P Demailly and J. Koll\'ar, Semi-continuity of complex singularity exponents and K\"ahler-Einstein metrics on Fano orbifolds, Ann. Sci. \'Ec. Norm. Sup\'er. (4) 34 (4) (2001) 525-556.
			
\bibitem{DP03}J.-P Demailly and T. Peternell, A Kawamata-Viehweg vanishing theorem on compact K\"ahler manifolds, J. Differential Geom. 63 (2) (2003) 231-277.

\bibitem{OF81}O. Forster, Lectures on Riemann surfaces, Grad. Texts in Math., 81, Springer-Verlag, New York-Berlin, 1981.

		

\bibitem{FoW18}J.E. Forn{\ae}ss and J.J.  Wu, A global approximation result by Bert Alan Taylor and the strong openness conjecture in $\mathbb{C}^n$, J. Geom. Anal. 28 (2018), no. 1, 1-12.
		
\bibitem{FoW20}J.E. Forn{\ae}ss and J.J.  Wu, Weighted approximation in $\mathbb{C}$, Math. Z. 294 (2020), no. 3-4, 1051-1064.

\bibitem{G-R}H. Grauert and R. Remmert, Coherent analytic sheaves, Grundlehren der mathematischen Wissenchaften, 265, Springer-Verlag, Berlin, 1984.


		
\bibitem{G16}Q.A. Guan, A sharp effectiveness result of Demailly's strong Openness conjecture, Adv.Math. 348 (2019) :51-80.
		

		

		
\bibitem{GMY-concavity2}Q.A. Guan, Z.T. Mi and Z. Yuan, Concavity property of minimal $L^2$ integrals with Lebesgue measurable gain \uppercase\expandafter{\romannumeral2}, https://www.researchgate.net/publication/354464147.



\bibitem{GMY-boundary2}Q.A. Guan, Z.T. Mi and Z. Yuan, Boundary points, minimal $L^2$ integrals and concavity property \uppercase\expandafter{\romannumeral2}: on weakly pseudoconvex K\"ahler manifolds, arXiv:2203.07723v2.

		
\bibitem{GY-concavity}Q.A. Guan and Z. Yuan, Concavity property of minimal $L^2$ integrals with Lebesgue measurable gain,  https://www.researchgate.net/publication/353794984.
		

		

		


\bibitem{GY-concavity3}Q.A. Guan and Z. Yuan, Concavity property of minimal $L^2$ integrals with Lebesgue measurable gain III-----open Riemann surfaces, https://www.researchgate.net/publication/356171464.



	
\bibitem{guan-zhou13ap}Q.A. Guan and X.Y. Zhou, A solution of an $L^{2}$ extension problem with an optimal estimate and applications, Ann. of Math. (2) 181 (2015), no. 3, 1139--1208.

\bibitem{GZSOC}Q.A. Guan and X.Y Zhou, A proof of Demailly's strong openness conjecture, Ann. of Math. (2) 182 (2015), no. 2, 605-616.

\bibitem{GZeff}Q.A. Guan and X.Y Zhou, Effectiveness of Demailly's strong openness conjecture and related problems, Invent. Math. 202 (2015), no. 2, 635-676.


		
\bibitem{GZ20}Q.A. Guan and X.Y. Zhou, Restriction formula and subadditivity property related to multiplier ideal sheaves, J. Reine Angew. Math. 769, 1-33 (2020).

\bibitem{Guenancia}H. Guenancia, Toric plurisubharmonic functions and analytic adjoint ideal sheaves, Math. Z. 271 (3-4) (2012) 1011-1035.
		
		

\bibitem{JonssonMustata}M. Jonsson and M. Musta\c{t}\u{a}, Valuations and asymptotic invariants for sequences of ideals, Annales de L'Institut Fourier A. 2012, vol. 62, no.6, pp. 2145-2209.

\bibitem{JM13}M. Jonsson and M. Musta\c{t}\u{a}, An algebraic approach to the openness conjecture of Demailly and Koll\'{a}r, J. Inst. Math. Jussieu (2013), 1--26.		

\bibitem{K16}D. Kim, Skoda division of line bundle sections and pseudo-division, Internat. J. Math. 27 (2016), no. 5, 1650042, 12 pp.
		
\bibitem{KS20}D. Kim and H. Seo, Jumping numbers of analytic multiplier ideals (with an appendix by Sebastien Boucksom), Ann. Polon. Math., 124 (2020), 257-280.

\bibitem{Lazarsfeld}R. Lazarsfeld,
		Positivity in Algebraic Geometry. \uppercase\expandafter{\romannumeral1}. Classical Setting: Line Bundles and Linear Series. Ergebnisse der Mathematik und ihrer Grenzgebiete. 3. Folge. A Series of Modern Surveys in Mathematics [Results in Mathematics and Related Areas. 3rd Series. A Series of Modern Surveys in Mathematics], 48. Springer-Verlag, Berlin, 2004;\\
		R. Lazarsfeld,
		Positivity in Algebraic Geometry. \uppercase\expandafter{\romannumeral2}. Positivity for vector bundles, and multiplier ideals. Ergebnisse der Mathematik und ihrer Grenzgebiete. 3. Folge. A Series of Modern Surveys in Mathematics [Results in Mathematics and Related Areas. 3rd Series. A Series of Modern Surveys in Mathematics], 49. Springer-Verlag, Berlin, 2004.

\bibitem{McNeal and Varolin}J. D. McNeal, D. Varolin, $L^2$ estimate for the $\bar{\partial}$ operator. Bull.Math. Sci. 5 (2015), no.2, 179-249.
		
\bibitem{Nadel}A. Nadel, Multiplier ideal sheaves and K\"ahler-Einstein metrics of positive scalar curvature, Ann. of Math. (2) 132 (3) (1990) 549-596.

\bibitem{S-O69}L. Sario and K. Oikawa, Capacity functions, Grundl. Math. Wissen. 149, Springer-Verlag, New York, 1969. Mr 0065652. Zbl 0059.06901.



\bibitem{Siu96}Y.T. Siu, The Fujita conjecture and the extension theorem of Ohsawa-Takegoshi, Geometric Complex Analysis, World Scientific, Hayama, 1996, pp.223-277.
		

	
\bibitem{Siu05}Y.T. Siu, Multiplier ideal sheaves in complex and algebraic geometry, Sci. China Ser. A 48 (suppl.) (2005) 1-31.
		
\bibitem{Siu09}Y.T. Siu, Dynamic multiplier ideal sheaves and the construction of rational curves in Fano manifolds, Complex Analysis and Digtial Geometry, in: Acta Univ. Upsaliensis Skr. Uppsala Univ. C Organ. Hist., vol.86, Uppsala Universitet, Uppsala, 2009, pp.323-360.

\bibitem{skoda1972}H. Skoda, Sous-ensembles analytiques d'ordre fini ou infini dans $\mathbb{C}^n$, Bull. Soc. Math. France 100 (1972) 353-408.



\bibitem{Tian}G. Tian, On K\"ahler-Einstein metrics on certain K\"ahler manifolds with $C_1(M)>0$, Invent. Math. 89 (2) (1987) 225-246.

\bibitem{Tsuji}M. Tsuji, Potential theory in modern function theory, Maruzen Co., Ltd., Tokyo, 1959. MR 0114894. Zbl 0087.28401.

\bibitem{ZZ2018}X.Y Zhou and L.F.Zhu,An optimal $L^2$ extension theorem on weakly pseudoconvex K\"ahler manifolds, J. Differential Geom.110(2018), no.1, 135-186.


\bibitem{ZZ2019}X.Y Zhou and L.F.Zhu, Optimal $L^2$ extension of sections from subvarieties in weakly pseudoconvex manifolds. Pacific J. Math. 309 (2020), no. 2, 475-510.

\bibitem{ZhouZhu20siu's}X.Y. Zhou and L.F. Zhu, Siu's lemma, optimal $L^2$ extension and applications to twisted pluricanonical sheaves, Math. Ann. 377 (2020), no. 1-2, 675-722.

\end{thebibliography}

\end{document}